\documentclass[11pt, a4paper]{amsart}
\usepackage{amssymb,amsmath,amsthm,bbold}
\usepackage[mathscr]{euscript}
\usepackage{geometry}
\usepackage{tikz}
\usetikzlibrary{cd, intersections, calc}
\usepackage{amscd}
\usepackage[all]{xy}
\definecolor{BleuTresFonce}{rgb}{0.215, 0.215, 0.36}
\definecolor{EgyptianBlue}{rgb}{0.06, 0.2, 0.65}
\usepackage[colorlinks,final,hyperindex]{hyperref}
\hypersetup{citecolor=BleuTresFonce, urlcolor=blue, linkcolor=EgyptianBlue}
\usepackage[font=small, labelfont=bf, margin=10pt]{caption}
\usepackage{csquotes}

\usepackage{color}

\numberwithin{equation}{subsection}

\newtheorem*{conjecture*}{Conjecture}
\newtheorem{conjecture}[equation]{Conjecture}
\newtheorem{theorem}[equation]{Theorem}
\newtheorem*{theorem*}{Theorem}

\newtheorem{defi}[equation]{Definition}
\newtheorem{lemma}[equation]{Lemma}
\newtheorem{prop}[equation]{Proposition}
\newtheorem{prop-def}[equation]{Proposition-Definition}
\newtheorem{sled}[equation]{Corollary}
\newtheorem{corollary}[equation]{Corollary}
\newtheorem{remark}[equation]{Remark}
\theoremstyle{definition}
\newtheorem{example}[equation]{Example}
\newtheorem*{notation*}{Notations}
\newtheorem*{example*}{Example}
\newtheorem*{claim*}{Claim}

\newcommand{\Gr}{\operatorname{\Gamma}}
\newcommand{\parti}{\operatorname{\Pi_{\mathsf{gr}}}}
\newcommand{\Path}{\operatorname{\mathsf{P}}}
\newcommand{\K}{\operatorname{\mathsf{K}}}
\newcommand{\Cyc}{\operatorname{\mathsf{C}}}
\newcommand{\St}{\operatorname{\mathsf{St}}}
\DeclareMathOperator{\Ver}{Vert}
\DeclareMathOperator{\Edge}{Edge}
\DeclareMathOperator{\Leav}{Leav}
\DeclareMathOperator{\In}{In}
\newcommand{\Aut}{\operatorname{\mathsf{Aut}}}
\DeclareMathOperator{\Id}{Id}
\DeclareMathOperator{\Hom}{Hom}
\newcommand{\GrCol}{\operatorname{\mathsf{GrCol}}}
\newcommand{\C}{\operatorname{\mathcal{C}}}

\newcommand{\M}{\operatorname{\mathcal{M}}}
\newcommand{\bM}{\operatorname{\overline{\mathcal{M}}}}
\newcommand{\Pro}{\operatorname{\mathbb{P}}}
\DeclareMathOperator{\PGL}{PGL}
\DeclareMathOperator{\Aff}{Aff}

\DeclareMathOperator{\Conf}{Conf}
\DeclareMathOperator{\NConf}{NConf}
\newcommand{\D}{\operatorname{\mathcal{D}}}

\newcommand{\Disc}{\operatorname{\mathbb{D}}}
\newcommand{\La}{\operatorname{\mathcal{L}}}

\newcommand{\G}{\operatorname{\mathcal{G}}}

\newcommand{\B}{\operatorname{\mathcal{B}}}

\newcommand{\Cobar}{\operatorname{\mathsf{\Omega}}}
\newcommand{\cokoszul}{\operatorname{\text{!`}}}
\newcommand{\Orb}{\operatorname{\mathcal{O}}}

\newcommand{\Q}{\operatorname{\mathcal{Q}}}
\newcommand{\I}{\operatorname{\mathcal{I}}}
\newcommand{\E}{\operatorname{\mathcal{E}}}
\newcommand{\R}{\operatorname{\mathcal{R}}}

\newcommand{\T}{\operatorname{\mathbb{T}}}
\newcommand{\Pop}{\operatorname{\mathcal{P}}}
\usepackage[OT2,OT1]{fontenc}
\newcommand\cyrillic[1]{{\fontencoding{OT2}\fontfamily{wncyr}\selectfont #1}}
\newcommand\mathcyr[1]{\text{\cyrillic{#1}}}
\newcommand\Sha{\textnormal{\mathcyr{Sh}}}
\newcommand{\forget}{\operatorname{\mathrm{f}}}
\newcommand{\LT}{\operatorname{\mathsf{LT}}}

\newcommand{\grpermlex}{\operatorname{\mathsf{graphpermlex}}}
\DeclareMathOperator{\ind}{ind}
\DeclareMathOperator{\res}{res}
\newcommand{\Com}{\operatorname{\mathsf{gcCom}}}
\newcommand{\Susp}{\operatorname{\mathcal{S}}}
\newcommand{\Lie}{\operatorname{\mathsf{gcLie}}}
\newcommand{\Ass}{\operatorname{\mathsf{gcAss}}}
\newcommand{\Pois}{\operatorname{\mathsf{gcPois}}}

\newcommand{\Gerst}{\operatorname{\mathsf{gcGerst}}}

\newcommand{\Grav}{\operatorname{\mathsf{gcGrav}}}
\newcommand{\Hyper}{\operatorname{\mathsf{gcHyper}}}
\newcommand{\Tree}{\operatorname{\mathsf{Tree}}}

\newcommand{\FM}{\operatorname{\mathsf{FM}}}

\newcommand{\eP}{\EuScript{P}}
\newcommand{\eQ}{\EuScript{Q}}

\newcommand{\beM}{\overline{\EuScript{M}}}
\newcommand{\eM}{\EuScript{M}}
\newcommand{\MPOp}{\mathbf{\mathsf{MSop}}}

\newcommand{\bQ}{{\mathbb Q}}

\newcommand{\bZ}{{\mathbb Z}}

\newcommand{\cQ}{{\mathcal{Q}}}

\newcommand{\calF}{{\mathcal{F}}}
\newcommand{\op}{{\mathsf{op}}}

\title{Hilbert series for contractads and  \\
	modular compactifications}

\author[A.\,Khoroshkin]{Anton Khoroshkin}
\address{Anton Khoroshkin: \newline
	Department of Mathematics, University of Haifa, Mount Carmel, 3103301, Haifa, Israel
}
\email{khoroshkin@gmail.com}

\author[D.\,Lyskov]{Denis Lyskov}
\address{Laboratory of Algebraic Geometry, National Research University Higher School of Economics, 6 Usacheva street, Moscow 119048, Russia}
\email{ddl2001@yandex.ru}

\begin{document}
	\begin{abstract}
		Contractads are operadic-type algebraic structures well-suited for describing configuration spaces indexed by a simple connected graph $\Gamma$. Specifically, these configuration spaces are defined as $\mathrm{Conf}_{\Gamma}(X):=X^{|V(\Gamma)|}\setminus \cup_{(ij)\in E(\Gamma)} \{x_i=x_j\}$. In this paper, we explore functional equations for the Hilbert series of Koszul dual contractads and provide explicit Hilbert series for fundamental contractads such as the commutative, Lie, associative and the little disks contractads.
		
		Additionally, we focus on a particular contractad derived from the wonderful compactifications of $\mathrm{Conf}_{\Gamma}(\mathbb{k})$, for $\mathbb{k}=\mathbb{R},\mathbb{C}$. First, we demonstrate that for complete multipartite graphs, the associated wonderful compactifications coincide with the modular compactifications introduced by Smyth. Second, we establish that the homology of the complex points and the homology of the real locus of the wonderful contractad are both quadratic and Koszul contractads. We offer a detailed description of generators and relations, extending the concepts of the Hypercommutative operad and cacti operads, respectively.
		Furthermore, using the functional equations for the Hilbert series, we describe the corresponding Hilbert series for the homology of modular compactifications. 
		
	\end{abstract}
	\maketitle
	\setcounter{section}{-1}
	
	\section{Introduction}
	\subsection{Motivation}
	The language of operads became very popular at the end of the 20th century and found numerous applications in different areas of mathematics. Nowadays, there are several excellent textbooks accessible for students \cite{loday2012algebraic,markl2002operads}, which confirms that operads have earned an important place in mathematical research. The notion of an operad was initiated by topologists in the theory of iterated loop spaces~\cite{may1972geometry,boardman1968homotopy,stasheff1963homotopy}. 
	Soon after, it was realized that the language of operads is quite varied, allowing us to gain essential insights into certain categories if we interpret them as categories of algebras/modules over a given operad $\mathcal{P}$. In particular, a well-defined (dg)-algebraic model of an operad $\mathcal{P}$ can provide substantial information about the homotopy category of $\mathcal{P}$-algebras. Moreover, it has been discovered that some operads have descriptions that are as interesting as the descriptions of the categories of algebras over them. Notable examples include the little disks operad and the operad of the Deligne-Mumford compactifications of the moduli spaces of rational curves~\cite{getzler1995operads,etingof2010cohomology}.
	
	Recently, it was noticed that there are various operadic-type algebraic structures where the category of algebras may not be very meaningful, yet the underlying algebraic structure can reveal much about the corresponding space of operations. Examples of such structures include \emph{reconnectads} (\cite{dotsenko2024reconnectads}), $\EuScript{LBS}$-\emph{operads} (\cite{coron2022matroids}), and \emph{contractads} (discovered by the second author in~\cite{lyskov2023contractads}). Famous algebraic methods known for algebras and operads, such as Koszul duality, Gröbner basis theory, and other enumerative techniques, can be easily extended to these newer algebraic structures.
	
	One of the basic applications of Koszul duality theory for algebras and operads is the functional equation for generating series of dimensions. For example, if \(A = \bigoplus_{n \geq 0} A_n\) is a quadratic Koszul algebra. Then the Hilbert series for \(A\) and its Koszul dual algebra \(A^!\) satisfy the following equation:
	\[
	H_A(t) \cdot H_{A^!}(-t) = 1, \text{ where } 
	H_A(t) := \sum_{n \geq 0} \dim A_n t^n, \
	H_{A^!}(t) := \sum_{n \geq 0} \dim A^{!}_n t^n.
	\]
	Respectively, if \(\mathcal{P} = \bigcup_{n \geq 1} \mathcal{P}(n)\) is a quadratic Koszul operad generated by binary operations, then the corresponding Hilbert series for $\mathcal{P}$ and its Koszul dual $\mathcal{P}^{!}$ satisfy the following functional equation:
	\begin{equation}
		\label{eq:koszul:rel:operad}
		\chi_{\mathcal{P}}(-t) \circ \chi_{\mathcal{P}^!}(-t) = t, \text{ where } 
		\chi_{\mathcal{P}}(t):=\sum_{n\geq 1} \dim\mathcal{P}(n)\frac{t^n}{n!}, \ \chi_{\mathcal{P}^{!}}(t):=\sum_{n\geq 1} \dim\mathcal{P}^{!}(n)\frac{t^n}{n!},
	\end{equation}
	
	In particular, the first known description of the Hilbert series for the Deligne-Mumford compactification \(\beM_{0, n+1}\) was obtained by Getzler using a straightforward generalization of the functional equation~\eqref{eq:koszul:rel:operad}.
	
	The goal of this paper is twofold. First, we demonstrate the types of generating series and functional equations one can obtain for contractads, illustrating with examples such as the Commutative, Lie, and little disks contractads. 
	Second, we examine in detail the contractad that generalizes \(\beM_{0, n+1}\). We show that for multipartite graphs, the corresponding smooth algebraic varieties are isomorphic to the one discovered by Smyth (\cite{smyth2009towards}) under the name \emph{modular compactification}. We show that the homology of the complex points and real locus of this contractad are Koszul contractads. As one of the applications, we provide an algebraic description of the homology contractads and their Koszul duals, resulting in a reasonably compact description of the desired generating series.
	
	\subsection{Functional equation for generating series}
	With each simple connected graph\footnote{By simple graph, we mean a graph without multiple edges and without loops.} $\Gamma$ and a given subset of vertices $G$ of $\Gamma$ (such that the induced subgraph $\Gamma|_{G}$ is connected), we can associate a quotient simple connected graph $\Gamma/G$. This is achieved by contracting all vertices of $G$ into a single vertex and removing any resulting loops and multiple edges. This operation is associative and leads to the definition of a contractad, as discovered in~\cite{lyskov2023contractads}.
	
	Roughly speaking, a contractad $\mathcal{P}$ in a monoidal category $\mathcal{C}$ assigns to each simple connected graph $\Gamma$ an object $\mathcal{P}(\Gamma)$ and to each connected subset of vertices a composition rule:
	$$
	\circ^{\Gr}_G: \mathcal{P}(\Gamma/G) \otimes \mathcal{P}(\Gamma|_{G}) \rightarrow \mathcal{P}(\Gamma).
	$$
	This rule satisfies appropriate associativity and symmetry conditions. More generally, for each partition of vertices $I = G_1 \sqcup G_2 \sqcup \ldots \sqcup G_k$ such that each subset $G_i$ is connected, we have a composition map:
	\begin{equation}
		\label{eq::composition}
		\gamma^{\Gamma}: (\mathcal{P} \circ \mathcal{P})(\Gamma) := \bigoplus_{I \vdash \Gamma} \mathcal{P}(\Gamma/I)\otimes\bigotimes_{G \in I} \mathcal{P}(\Gamma|_{G}) \rightarrow \mathcal{P}(\Gamma).
	\end{equation}
	
	As with all operadic-type structures, many interesting examples of contractads can be presented by generators and relations. 
	We consider several basic examples, generalizing the known examples of operads of commutative, associative, Lie, Poisson algebras, the little discs operads and the Gravity and Hypercommutative operads. We show that in all these examples one can find a quadratic Gr\"obner basis, what follows that these contractads are Koszul. 
	Recall that a contractad is Koszul whenever its Koszul complex is acyclic. As a graded vector space, the Koszul complex is a composition $\mathcal{P}^{\cokoszul}\circ\mathcal{P}$. Thus, to derive a functional equation for the Hilbert series of $\mathcal{P}(\Gamma)$, one needs to interpret the composition "$\circ$" used in the Koszul complex~\eqref{eq::composition}.
	
	Unfortunately, the set of connected simple graphs is a very complicated combinatorial object, and it is unclear how to collect the numbers $\dim \mathcal{P}(\Gamma)$ into a single generating series. 
	The easeast solution is to consider \underline{\emph{graphic functions}} that assign a number to each simple connected graph
	(Definition~\ref{def::graphic::func}).
	Composition~\eqref{eq::composition} of graphic collections
	defines an associative product of corresponding graphic functions (introduced by Schmitt in~\cite{schmitt1994incidence}):
	\begin{equation}
		\label{eq::contr::comp}
		(\phi * \psi)(\Gamma) := \sum_{I \vdash \Gamma} \phi(\Gamma/I) \prod_{G \in I} \psi(\Gamma|_{G}).
	\end{equation}
	For a quadratic Koszul contractad $\mathcal{P}$ generated by binary operations, the functional equation will be:
	\begin{equation}
		\label{eq::func::eq::contr}
		\chi(\mathcal{P}) * \chi_{-1}(\mathcal{P}^{!}) = \chi_{-1}(\mathcal{P}^{!}) * \chi(\mathcal{P}) = \varepsilon,
	\end{equation}
	where $\chi_{-1}(\mathcal{P}^{!})(\Gamma) = (-1)^{|V_{\Gamma}| - 1} \dim \mathcal{P}^{!}(\Gamma)$ and $\varepsilon(\Gamma)$ is nonzero only for a graph with one vertex.
	Concequently, the graphic function $\chi(\mathcal{P})$ associated with a Koszul contractad $\mathcal{P}$ uniquely defines the graphic function $\chi(\mathcal{P}^{!})$ for the Koszul dual contractad.
	However, functional equation~\eqref{eq::func::eq::contr} is given as a summation over subgraphs or quotient graphs and it is not clear for us how to use this formula in general. We suggest a partial solution to this problem through the following observation. The \emph{contractad composition}~\eqref{eq::contr::comp} has a meaningful interpretation if we restrict our interest to certain particular families of graphs. These families include paths, cycles, and complete multipartite graphs.
	For example, to a graphic function $\psi$ with values in $\Bbbk$, we assign its Young generating function $F_{\mathsf{Y}}(\psi) \in \Bbbk[[z]] \otimes \Lambda_{\mathbb{Q}}$, by the rule:
	\begin{gather*}
		F_{\mathsf{Y}}(\psi)(z) = \sum_{l(\lambda) \geq 2} \psi(\K_{\lambda}) \frac{m_{\lambda}}{\lambda!} + \sum_{n \geq 1} \sum_{|\lambda| \geq 0} \psi(\K_{(1^n) \cup \lambda}) \frac{z^n}{n!} \frac{m_{\lambda}}{\lambda!}
	\end{gather*}
	Here, $\K_\lambda$ denotes a complete multipartite graph with partition denoted by $\lambda$; $m_\lambda = \mathsf{Sym}(x^{\lambda})$ is an additive basis in the ring of symmetric functions called a monomial basis and $\lambda!:=\lambda_1!\lambda_2!\ldots$.
	
	The following theorem is one of the main observations of this article:
	\begin{theorem*}[Theorem~\ref{thm::Ser::MS}]
		For graphic functions $\psi, \varphi$ the Young symmetric function of their composition is the composition of these functions with respect to the variable $z$:
		$$		F_{\mathsf{Y}}(\psi * \varphi)(z) = F_{\mathsf{Y}}(\psi)(F_{\mathsf{Y}}(\varphi)(z)).
		$$
	\end{theorem*}
	
	This theorem allows us to find precise formulas for the Commutative (\(\Com\)), Lie (\(\Lie\)), and Gerstenhaber (\(\Gerst\)) contractads. Similar to the case of operads, $\Gerst$ is a contractad of the homology of the little disks contractad. In particular, its space of operations is given by the homology of generalized configuration spaces:
	$$
	\Gerst(\Gamma) \simeq H_{\bullet}(\Conf_{\Gamma}(\mathbb{C}); \mathbb{Q}), \text{ where } \Conf_{\Gamma}(X) := X^{|V(\Gamma)|} \setminus \bigcup_{(ij) \in \Gamma} \{x_i = x_j\}.
	$$
	The Hilbert series of the homology of generalized configuration spaces of points on a plane is known to coincide with the chromatic polynomial (see, e.g.,~\cite[\S2.4]{orlik2013arrangements}). Therefore, as a corollary, we obtain the following equality of the generating series of chromatic polynomials for complete multipartite graphs:
	$$
	\sum_{|\lambda| \geq 0} \chi_{\K_{\lambda}}(q) \frac{m_{\lambda}}{\lambda!} = \left(1 + \sum_{n \geq 1} \frac{p_n}{n!}\right)^q.
	$$
	Here, $p_n = \sum x_i^n$ denotes the Newton power sum (see Corollary~\ref{cor::yungchrom} for details).

	\subsection{Wonderful contractad and Modular compactifications}
	In their seminal work~\cite{de1995wonderful}, De Concini and Procesi introduced the concept of a \emph{wonderful compactification} for any hyperplane arrangement. They demonstrated that the wonderful compactification of the braid arrangement is isomorphic to the compactified moduli space $\bM_{0,n+1}$, which represents stable rational curves with $n+1$ marked points.
	
	E. Rains later discovered~\cite{rains2010homology} that these wonderful compactifications exhibit functorial properties with respect to arrangements, thereby defining an operadic-type structure on the union of these spaces. More recently, it was shown in~\cite[\S 2.6]{lyskov2023contractads} that the wonderful compactifications of the complement of \emph{graphical arrangements} $\Bbbk^n \setminus \cup_{(ij)\in\Gamma} \{x_i = x_j\}$  form a contractad $\bM$ in the category of algebraic varieties over $\Bbbk$, termed \emph{the wonderful contractad}. 
	
	In this paper, we focus on the case of complete multipartite graphs $\K_{\lambda}$ on $n$ vertices. We demonstrate that the corresponding wonderful compactification $\bM(\K_{\lambda})$ is isomorphic to specific blowdowns of the Deligne-Mumford compactification $\beM_{0,\K_{\lambda}}$, where all components in a degenerate curve are collapsed if there is no edge in the graph $\K_{\lambda}$ connecting two vertices from this component. These smooth manifolds were initially introduced by Smyth (\cite{smyth2009towards}) under the term \emph{modular compactification}, with the smooth cases classified in~\cite{moon2018birational}.
	
	We further show that the corresponding homology contractads $H_{\bullet}(\bM_{\mathbb{k}})$ for $\Bbbk=\mathbb{C}$, $\mathbb{R}$ are Koszul and provide substantial enumerative information about them. Specifically, we prove the following results for the complex case generalizing the work of Getzler~\cite{getzler1995operads} on Hypercommutative and Gravity operads:
	
	\begin{theorem*}[Theorem~\ref{thm::homologywond_hyper}, Theorem~\ref{thm:gravpres},  Theorem~\ref{thm::hilbert_complex_modular}]\hfill\break
		\begin{itemize}
			\item The homology contractad $H_{\bullet}(\bM_{\mathbb{C}})$ is a quadratic Koszul contractad $\Hyper$, generated by symmetric elements $\nu_{\Gr} \in \Hyper(\Gr)$ of degree $2(|V_\Gamma|-2)$, satisfying the following relations: for each connected graph $\Gr$ and each pair of edges $e,e'\in E_{\Gr}$, we have
			$$
			\sum_{G\colon e\subset G} \nu_{\Gr/G}\circ^{\Gr}_G \nu_{\Gr|_G}=\sum_{G\colon e'\subset G} \nu_{\Gr/G}\circ^{\Gr}_G \nu_{\Gr|_G}.
			$$
			\item The Koszul dual contractad of $\Hyper$, called gravity ($\Grav$), coincides with $H_{\bullet+1}(\M_{\mathbb{C}})$ and  is  generated by symmetric generators $\lambda_{\Gr}$ of degree 1 in each component $\Gr$, satisfying the relations:
			\[
			\forall \Gr \text{ we have }
			\begin{cases}
				\sum_{e \in E_{\Gr}} \lambda_{\Gr/e} \circ_e \lambda_{\Gr|_e}=0;
				\\
				\sum_{e \in E_{\Gr|_G}} \lambda_{\Gr/e} \circ_e \lambda_{\Gr|_e}=\lambda_{\Gr/G} \circ^{\Gr}_G \lambda_{\Gr|_G} \\ \qquad \text{ for all tubes $G\subset \Gr$  with at least 3 vertices.}
			\end{cases}
			\]
			
			\item The generating series for Betti numbers of Modular compactifications $\beM_{0,\K_{\lambda}}(\mathbb{C})$
			\begin{multline*}
				F_{\mathsf{Y}}(\beM(\mathbb{C}))=\sum_{l(\lambda)\geq 2} \left[\sum^{|\lambda|-2}_{i=0}\dim H^{2i}(\beM_{0,\K_{\lambda}}(\mathbb{C}))q^i\right]\frac{m_{\lambda}}{\lambda!}+
				\\
				+ \sum_{n\geq 1, |\lambda|\geq 0} \left[\sum^{|\lambda|+n-2}_{i=0}\dim H^{2i}(\beM_{0,\K_{(1^n)\cup \lambda}}(\mathbb{C}))q^i\right]\frac{m_{\lambda}}{\lambda!}\frac{z^n}{n!}, 
			\end{multline*} is functional inverse (with respect to the variable $z$) of the following function
			$$
			G(z)=\frac{q}{q-1}z-\frac{1}{q(q-1)}\left[\left(1+z+\sum_{n\geq 1}\frac{p_n}{n!}\right)^q-1-\sum_{n\geq 1}\frac{p_nq^n}{n!}\right].
			$$
		\end{itemize}
	\end{theorem*}
	
	Additionally, we describe the real locus contractad $\bM_{\mathbb{R}}$ generalizing the corresponding models for operads discovered in~\cite{etingof2010cohomology,khoroshkin2019real}:
	
	\begin{theorem*}[Theorem~\ref{thm::cell_structure}, Theorem~\ref{thm::homology_real_locus}, Theorem~\ref{thm::hilbert_real_modular}]\hfill\break
		\begin{itemize}
			\item There exists a natural cell decomposition of the topological contractad $\bM_{\mathbb{R}}$ that is compatible with the contractad structure. The corresponding dg-contractad $C^{\mathrm{cell}}_{\bullet}(\bM_{\mathbb{R}})$ of cell complexes with rational coefficients is isomorphic to the cobar complex of the cocontractad $[\Ass^*]^{\tau}$.
			Here $[\Ass]^{\tau}$ is the space of coinvariants of the associative contractad for the action of the involution $\tau$ that reverses the order of how one reads associative words.
			\item The homology contractad $H_{\bullet}(\bM_{\mathbb{R}})$ is a quadratic Koszul contractad generated by binary and ternary operations, with the following isomorphism at the level of graphical collections:
			$$H_{\bullet}(\bM_{\mathbb{R}};\mathbb{Q}) \simeq \Com\circ (\Com_{\mathrm{odd}})^!.$$
			Here $\Com_{\mathrm{odd}}$ is a subcontractad of $\Com$ spanned by graphs with odd number of vertices.
			\item The generating series for (rational) Betti numbers of Modular compactifications $\beM_{0,\K_{\lambda}}(\mathbb{R})$ have a concise presentation
			$$
			F_{\mathsf{Y}}(\beM(\mathbb{R}))
			= \left[\sqrt{q}(z+\mathsf{SINH}_q)+\sqrt{q(z+\mathsf{SINH}_q)^2+1}\right]^{\frac{1}{\sqrt{q}}}-1-\sum_{n\geq 1}\frac{p_n}{n!},
			$$
			where  $\mathsf{SINH}_q=\sum_{n\geq 0}\frac{p_{2n+1}q^n}{(2n+1)!}$.
		\end{itemize}
	\end{theorem*}

	\subsection{Structure of the paper}
	The paper is organized as follows.
	
	In \S\ref{sec::contractads}, we revisit the definitions of contractads (\S\ref{sec:sub::contractads}), Koszul duality (\S\ref{sec::Koszul::contractads}), and Gr\"obner basis theory (\S\ref{sec::Grobner::contractad}) for contractads. We illustrate the general theory with standard examples such as the commutative $\Com$ and Lie $\Lie$ contractads (\S\ref{sec:sub::com::Lie}), the little disks contractad $\D_d$ (\S\ref{sec:disks}), and the associative $\Ass$ and Poisson $\Pois$ contractads (\S\ref{sec:Ass}). We conclude this section by focusing on complete multipartite graphs and explaining their relationship to certain types of colored operads (\S\ref{sec::multipartite}).
	
	In \S\ref{sec::hilbertseries}, we prove the functional equations for the Hilbert series of contractads. We apply these results to the aforementioned examples and derive interesting expressions for the generating series of chromatic polynomials (\S\ref{sec::chromatic}).
	
	In \S\ref{sec::wonderful_modular}, we explore various algebraic and geometric descriptions of the wonderful contractad. In \S\ref{sec::H::wonderful::C}, we describe the cohomology of the contractad $\bM_{\mathbb{C}}$. First, we describe the cohomology rings compatible with the contractad structure (\S\ref{sec::H::MC::ring}). Next, we discuss generalizations of the gravity and hypercommutative operads (\S\ref{sec::gravity}). We then prove that the homology contractad $H_{\bullet}(\bM_{\mathbb{C}})$ is Koszul and coincides with the Hypercommutative contractad $\Hyper$ (\S\ref{sec::H::wonderful}). We also provide a detailed description of the normal monomials of this contractad in~\S\ref{sec::hycom::monomials}. In \S\ref{sec::H::wonder::R}, we address the real locus of the wonderful contractad, describing its cell structure (\S\ref{sec::cell::M::R}), homology contractad (\S\ref{sec::realhomology}), and the corresponding Hilbert series (\S\ref{sec::M::real::Hilb}).

	\section*{Acknowledgments}
	We would like to thank  V.\,Dotsenko, D.\,Piontkovski and A.\,Vainstein for stimulating discussions. The research of the first author was partially supported by The Laszlo N. Tauber Family Foundation. Results of Section 3 were obtained with funding from the HSE Basic Research Program. Results of Section 4 were obtained under support of the grant RSF 24-21-00341 of Russian Science Foundation.

	\tableofcontents

	\section{Contractads}
	\label{sec::contractads}
	In this section, we give a quick introduction to the theory of contractads introduced by the second author~\cite{lyskov2023contractads}. Informally a contractad is a graphical generalisation of operads, whose set of operations are indexed by connected graphs and composition rules are numbered by contractions of connected subgraphs, which explains the nature of terminology. We discuss several definitions of contractads, their presentations in terms of generators and relations,  consider several examples of contractads, and describe Koszul theory and Gr\"obner bases for contractads . In the end, we explore the relationship between contractads, complete multipartite graphs, and coloured operads with outputs of the fixed colour and inputs with an unbounded number of colours.
	\subsection{Contractads}
	\label{sec:sub::contractads}
	In this subsection, we briefly recall several definitions of contractads. For more details, see~\cite[\S 1]{lyskov2023contractads}.
	\subsubsection{Monoidal and Infinitesimal definition of contractads}
	
	Let us call a \textit{graph} to be a finite simple undirected  $\Gr=(V_{\Gr},E_{\Gr})$, where $V_{\Gr}$ represents a set of vertices, and $E_{\Gr}$ represents a set of edges. 
	Simple means that $\Gr$ does not contain loops and double edges.
	Let us explore some particular examples of graphs:
	\begin{itemize}
		\label{typesofgraphs}
		\item the path graph $\Path_n$ on the vertex set $\{1,\cdots, n\}$ with edges $\{(i,i+1)| 1\leq i \leq n-1 \}$,    \item the complete graph $\K_n$ on the vertex set $\{1,\cdots, n\}$ and the edges $\{(i,j)|i\neq j\}$,
		\item the cycle graph $\Cyc_n$ on the vertex set $\{1,\cdots, n\}$ with edges $\{(i,i+1)| 1 \leq i \leq n-1\}\cup \{(n,1)\}$,
		\item the stellar graph $\St_n$ on the vertex set $\{0,1,\cdots, n\}$ with edges $\{(0,i)| 1\leq i \leq n \}$. The vertex $"0"$ adjacent to all vertices is called the "core". 
	\end{itemize}
	Consider the \textit{groupoid of connected graphs} $\mathsf{CGr}$ whose objects are non-empty connected simple graphs and whose morphisms are isomorphisms of graphs and consider a symmetric monoidal category 
	$\C=(\C, \otimes, 1_{\C})$. The main examples we have in mind correspond to the case when $\C$ is the category of topological spaces $(\mathsf{Top},\times)$ or the category of differential graded vector spaces $(\mathsf{dgVect}, \otimes)$ (with Koszul signs rule).
	\begin{defi}
		A graphical collection with values in $\C$ is a contravariant functor $\mathsf{CGr}^{\mathrm{op}}\rightarrow \C$. All graphical collections with values in $\C$ with natural transformations form a category $\GrCol_{\C}$.
	\end{defi}
	Informally, a graphical collection $\Orb$ is a sequence of objects $\{\Orb(\Gr)\}$ indexed by graphs.  For functorial reasons, for each component $\Orb(\Gr)$, there is a right action of the graph automorphisms group $\Aut(\Gr)$. Let us recall the notion of tubes and constructions of induced and contracted graphs.
	\begin{defi}
		\begin{enumerate}
			\item For a graph $\Gr$ and a subset of vertices $S$,  the \textit{induced subgraph} is the graph $\Gr|_S$ with vertex set $S$ and edges coming from the original graph.
			\item A tube of a graph $\Gr$ is a non-empty subset $G$ of vertices such that the induced subgraph $\Gr|_G$ is connected.  If the tube consists of one vertex, we call it trivial.
		\end{enumerate}
	\end{defi}
	
	\begin{defi}
		\begin{enumerate}
			\item A \textit{partition of a graph} $\Gr$ is a partition of the vertex set whose blocks are tubes. We denote by $\parti(\Gr)$ the set of partitions of the graph $\Gr$. We use the notation $I\vdash \Gr$ for partitions.
			\item For a partition $I$ of graph $\Gr$, the contracted graph, denoted $\Gr/I$, is the graph obtained from $\Gr$  by contracting each block of $I$ to a single vertex. Explicitly, vertices of $\Gr/I$ are partition blocks and edges are pairs $\{G\},\{H\}$ of blocks such that their union $G\cup H$ is a tube of $\Gr$.
			\item Given a tube $G\subset V_{\Gr}$, we denote by $\Gr/G$ the contracted graph obtained from $\Gr$ by collapsing $G$ to a single vertex. Explicitly, $\Gr/G$ is the contracted graph associated with a partition $I=\{G\}\cup \{\{v\}|v\not\in G\}$. 
		\end{enumerate}
	\end{defi}
	
	\begin{figure}[ht]
		\centering
		\begin{gather*}
			\vcenter{\hbox{\begin{tikzpicture}[scale=0.7]
						\fill (0,0) circle (2pt);
						\node at (0,0.6) {1};
						\fill (1,0) circle (2pt);
						\node at (1,0.6) {2};
						\fill (2,0) circle (2pt);
						\node at (2,0.6) {3};
						\fill (3,0) circle (2pt);
						\node at (3,0.6) {4};
						\fill (4,0) circle (2pt);
						\node at (4,0.6) {5};
						\draw (0,0)--(1,0)--(2,0)--(3,0)--(4,0);
						\draw[dashed, rounded corners=5pt] (-0.25,-0.25) rectangle ++(1.5,0.5);
						\draw[dashed, rounded corners=5pt] (2.75,-0.25) rectangle ++(1.5,0.5);
						\draw[dashed] (2,0) circle (7pt);
			\end{tikzpicture}}}
			\quad
			\longrightarrow
			\quad
			\vcenter{\hbox{\begin{tikzpicture}[scale=0.7]
						\fill (0,0) circle (2pt);
						\node at (0,0.5) {\{1,2\}};
						\fill (1.5,0) circle (2pt);
						\node at (1.5,0.5) {\{3\}};
						\fill (3,0) circle (2pt);
						\node at (3,0.5) {\{4,5\}};
						\draw (0,0)--(1.5,0)--(3,0);    
			\end{tikzpicture}}}
			\\
			\\
			\vcenter{\hbox{\begin{tikzpicture}[scale=0.7]
						\fill (0,0) circle (2pt);
						\node at (-0.4,-0.3) {1};
						\fill (2,0) circle (2pt);
						\node at (2.4,-0.3) {4};
						\fill (0,2) circle (2pt);
						\node at (-0.4,2.3) {2};
						\fill (1,1) circle (2pt);
						\node at (1,0.5) {5};
						\fill (2,2) circle (2pt);
						\node at (2.4,2.3) {3};
						\draw (0,0)--(2,0)--(2,2)--(0,2)--cycle;
						\draw (0,0)--(1,1)--(2,2);
						\draw (2,0)--(1,1)--(0,2);
						\draw[dashed] (0,0) circle (7pt);
						\draw[dashed] (1,1) circle (7pt);
						\draw[dashed] (2,0) circle (7pt);
						\draw[dashed, rounded corners=5pt] (-0.25,1.75) rectangle ++(2.5,0.5);
			\end{tikzpicture}}}
			\quad
			\longrightarrow
			\vcenter{\hbox{\begin{tikzpicture}[scale=0.75]
						\fill (0,0) circle (2pt);
						\node at (-0.4,-0.3) {\{1\}};
						\fill (1,1) circle (2pt);
						\node at (1,0.5) {\{5\}};
						\fill (1,2) circle (2pt);
						\node at (1,2.4) {\{2,3\}};
						\fill (2,0) circle (2pt);
						\node at (2.4,-0.3) {\{4\}};
						\draw (0,0)--(2,0)--(1,2)-- cycle;
						\draw (0,0)--(1,1)--(1,2);
						\draw (2,0)--(1,1);
			\end{tikzpicture}}}
			\quad
			\quad
			\quad
			\vcenter{\hbox{\begin{tikzpicture}[scale=0.7]
						\fill (-0.63,1.075)  circle (2pt);
						\node at (-0.9,1.4) {1};
						\fill (0.63,1.075)  circle (2pt);
						\draw[dashed] (0.63,1.075) circle  (7pt);
						\node at (0.9,1.4) {2};
						\fill (-1.22,0) circle (2pt);
						\node at (-1.6,0) {6};
						\fill (1.22,0) circle (2pt);
						\node at (1.6,0) {3};
						\draw[dashed] (1.22,0) circle  (7pt);
						\fill (-0.63,-1.075)  circle (2pt);
						\node at (-0.9,-1.4) {5};
						\fill (0.63,-1.075)  circle (2pt);
						\node at (0.9,-1.4) {4};
						\draw[dashed] (0.63,-1.075) circle  (7pt);
						\draw (-1.22,0)--(-0.63,1.075)--(0.63,1.075)--(1.22,0)--(0.63,-1.075)--(-0.63,-1.075)--cycle;
						\draw[dashed] (-1.6,0)[rounded corners=15pt]--(-0.4,1.7)[rounded corners=15pt]--(-0.4,-1.7)[rounded corners=12pt]--cycle;
			\end{tikzpicture}}}
			\quad\longrightarrow\quad
			\vcenter{\hbox{\begin{tikzpicture}[scale=0.7]
						\fill (0.63,1.075)  circle (2pt);
						\node at (0.9,1.5) {\{2\}};
						\fill (-1.22,0) circle (2pt);
						\node at (-2.2,0) {\{1,5,6\}};
						\fill (1.22,0) circle (2pt);
						\node at (1.75,0) {\{3\}};
						\fill (0.63,-1.075)  circle (2pt);
						\node at (0.9,-1.5) {\{4\}};
						\draw (-1.22,0)--(0.63,1.075)--(1.22,0)--(0.63,-1.075)--cycle;
			\end{tikzpicture}}}
		\end{gather*}
		\caption{Examples of partitions of graphs and associated contractions.}
		\label{contrpic}
	\end{figure}
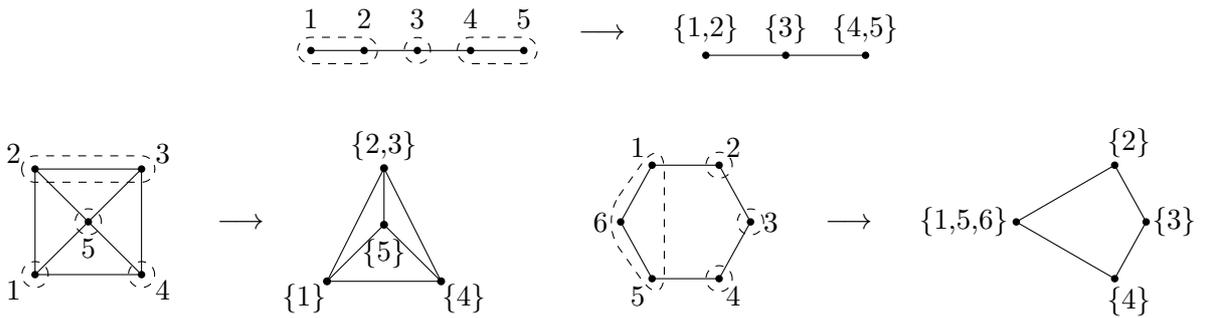
	The category of graphical collections admits a monoidal structure, called the contraction product, and denoted by $\circ$, which is defined by the following formula:
	\begin{equation}
		\label{eq::contract::product}
		(\Pop \circ \Q)(\Gr) := \bigoplus_{I \vdash \Gr} \Pop(\Gr/I) \otimes \bigotimes_{G \in I} \Q(\Gr|_G),
	\end{equation}  
	where the sum ranges over all partitions of $\Gr$. This operation is associative, and the graphical collection $\mathbb{1}$ defined as
	\[
	\mathbb{1}(\Gr):= \begin{cases}
		1_{\C}, \text{ for } \Gr \cong \Path_1,
		\\
		0, \text{ otherwise}.
	\end{cases}
	\] is the unit for this operation.
	\begin{defi}[Monoidal definition of contractads]\label{def:monoidal}
		A contractad is a monoid in the monoidal category of graphical collections equipped with the contraction product $\circ$.
	\end{defi}
	By definition, a contractad is a triple $(\Pop,\gamma,\eta)$, where $\Pop$ is a graphical collection, $\gamma$ is a product map $\gamma\colon \Pop\circ\Pop\to \Pop$, that is a collection of maps
	\[
	\gamma_I^{\Gr}\colon \Pop(\Gr/I)\otimes \bigotimes_{G \in I} \Pop(\Gr|_G)\to \Pop(\Gr),
	\] ranging over all graphs and their partitions, and $\eta$ is a unit $\eta\colon 1_{\C}\to \Pop(\Path_1)$. If $\Pop$ is a set/topological/linear contractad, we let $\Id$ be the unique element in $\Pop(\Path_1)$, arising from the unit map.
	
	In a dual way, we define a cocontractad as a comonoid in the category of graphical collections. In other words, it is a triple $(\Q,\triangle,\epsilon)$, where $\triangle\colon \Q\to\Q\circ\Q$ is a coproduct and $\epsilon\colon \Q(\Path_1)\to 1_{\C}$ is a counit.
	
	An equivalent way to present contractads is via \textit{infinitesimal compositions} that are analogues of infinitesimal compositions $\circ_i\colon \Orb(n)\otimes\Orb(m)\to\Orb(n+m-1)$ for operads. Recall that, for a tube $G$ of a graph $\Gr$, the contracted graph $\Gr/G$ is obtained from $\Gr$ by collapsing $G$ to a single vertex. When $\Pop$ is a contractad, for each pair of a graph $\Gr$ and tube $G$, there exists a map $\circ^{\Gr}_G\colon \Pop(\Gr/G)\otimes \Pop(\Gr|_G)\to \Pop(\Gr)$, called the infinitesimal composition, defined by
	\[
	\Pop(\Gr/G)\otimes\Pop(\Gr|_G) \cong \Pop(\Gr/G)\otimes\Pop(\Gr|_G)\otimes \bigotimes_{v \not\in G} 1_{\C} \overset{\Id \otimes u^{\otimes}}{\hookrightarrow} \Pop(\Gr/G)\otimes\Pop(\Gr|_G)\otimes \bigotimes_{v \not\in G}\Pop(\Gr|_{\{v\}}) \overset{\gamma}{\rightarrow} \Pop(\Gr).
	\]
	Conversely, from infinitesimal compositions, one can recover all the structure maps of a contractad. 
	\\
	\subsubsection{Combinatorial description of contractads}
	Similarly to operads, contractads can be described in terms of decorated rooted trees. A \textit{rooted tree} is a connected directed tree $T$ in which each vertex has at least one input edge and exactly one output edge. This tree should have exactly one external outgoing edge, output. The endpoint of this edge is called the \textit{root}. The endpoints of incoming external edges that are not vertices are called \textit{leaves}. A tree with a single vertex is called a \textit{corolla}. For a rooted tree $T$ and edge $e$, let $T_e$ be the subtree  with the root at $e$ and $T^e$ be the subtree obtained from $T$ by cutting $T_e$.
	
	\begin{defi}
		For a connected graph $\Gr$, a $\Gr$-\textit{admissible} rooted tree is a rooted tree $T$ with leaves labeled by the vertex set $V_{\Gr}$ of the given graph such that, for each edge $e$ of the tree, the leaves of subtree $T_e$ form a tube of $\Gr$.
	\end{defi} 
	
	\begin{figure}[ht] 
		\centering
		\[
		\vcenter{\hbox{\begin{tikzpicture}[scale=0.6]
					\fill (0,0) circle (2pt);
					\fill (0,1.5) circle (2pt);
					\fill (1.5,0) circle (2pt);
					\fill (1.5,1.5) circle (2pt);
					\draw (0,0)--(1.5,0)--(1.5,1.5)--(0,1.5)-- cycle;
					\node at (-0.25,1.75) {$1$};
					\node at (1.75,1.75) {$2$};
					\node at (1.75,-0.25) {$3$};
					\node at (-0.25,-0.25) {$4$};
		\end{tikzpicture}}}
		\qquad
		\vcenter{\hbox{\begin{tikzpicture}[scale=0.75]
					\draw (0,0)--(0,1);
					\draw (0,1)--(1,2);
					\draw (0,1)--(-1,2);
					\draw (1,2)--(1.75,2.75);
					\draw (1,2)--(0.25,2.75);
					\draw (-1,2)--(-1.75,2.75);
					\draw (-1,2)--(-0.25,2.75);
					\node at (1.75,3.1) {$4$};
					\node at (0.25,3.1) {$3$};
					\node at (-1.75,3.1) {$1$};
					\node at (-0.25,3.1) {$2$};
		\end{tikzpicture}}}
		\qquad 
		\vcenter{\hbox{\begin{tikzpicture}[scale=0.75]
					\draw (0,0)--(0,1);
					\draw (0,1)--(1,2);
					\draw (0,1)--(-1,2);
					\draw (0,1)--(0,2);
					\draw (0,2)--(0.75,2.75);
					\draw (0,2)--(-0.75,2.75);
					\node at (-1,2.25) {$1$};
					\node at (-0.75,3.1) {$2$};
					\node at (0.75,3.1) {$3$};
					\node at (1,2.25) {$4$};
		\end{tikzpicture}}}
		\qquad 
		\vcenter{\hbox{\begin{tikzpicture}[scale=0.75]
					\draw (0,0)--(0,1);
					\draw (0,1)--(1,2);
					\draw (0,1)--(-1,2);
					\draw (-1,2)--(-1.75,2.75);
					\draw (-1,2)--(-1,2.75);
					\draw (-1,2)--(-0.25,2.75);
					\node at (-1.75,3.1) {$3$};
					\node at (-1,3.1) {$4$};
					\node at (-0.25,3.1) {$1$};
					\node at (1,2.25) {$2$};
		\end{tikzpicture}}}
		\qquad 
		\vcenter{\hbox{\begin{tikzpicture}[scale=0.75]
					\draw[thick] (-1.2,3)--(1.2,0);
					\draw[thick] (-1.2,0)--(1.2,3);
					\draw (0,0)--(0,1);
					\draw (0,1)--(1,2);
					\draw (0,1)--(-1,2);
					\draw (0,1)--(0,2);
					\draw (0,2)--(0.75,2.75);
					\draw (0,2)--(-0.75,2.75);
					\node at (-1,2.25) {$1$};
					\node at (-0.75,3.1) {$2$};
					\node at (0.75,3.1) {$4$};
					\node at (1,2.25) {$3$};
		\end{tikzpicture}}}
		\]
		\caption{Graph $\Cyc_4$ (on the left side) and examples of $\Cyc_4$-admissible trees. The first three are $\Cyc_4$-admissible, but the fourth is not, since leaves 2,4 do not form a tube.}
		\label{roottrees}
	\end{figure}
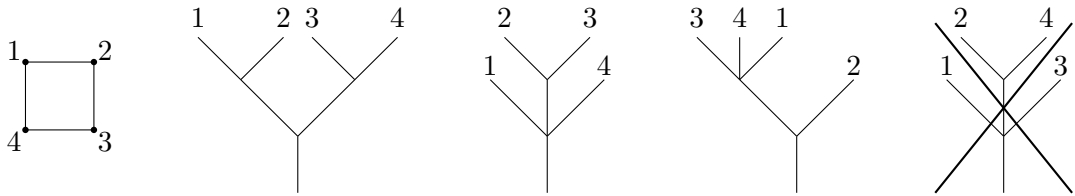
	
	We denote by $\Tree(\Gr)$ the set of all $\Gr$-admissible rooted trees. Note that the corolla with leaves labeled by the vertex set is always $\Gr$-admissible. Let us describe explicitly admissible trees for particular types of graphs.
	\begin{itemize}
		\label{grroottrees}
		\item For  paths, $\Path_n$-admissible trees are those that can be embedded in the plane such that leaves $[n]=\{1,2,3,...,n\}$ are arranged in increasing order. Indeed, this follows from the  fact  that tubes of $\Path_n$ are ordered intervals of $[n]$.
		\item For  cycles, $\Cyc_n$-admissible trees are those that can be embedded in the plane such that leaves are arranged in cyclic order, as in Figure~\ref{roottrees}.
		\item For complete graphs, $\K_n$-admissible trees are ordinary rooted trees since each vertex subset of a complete graph is a tube.
	\end{itemize}  
	Note that subtrees of admissible trees are also admissible. Indeed, for each edge $e$ of $\Gr$-admissible tree $T$, the subtree $T_e$ is a $\Gr|_{L_e}$-admissible, where $L_e$ is the set of leaves of $T_e$, and the subtree $T^e$ is a $\Gr/L_e$-admissible. For a partition $I$ of the graph $\Gr$, we define the grafting map
	\[
	\Tree(\Gr/I)\times\prod_{G \in I} \Tree(\Gr|_G)\to \Tree(\Gr),
	\] which joins roots of $\Gr|_G$-admissible trees to corresponding leaves of $\Gr/I$-admissible trees, as in Figure~\ref{substitution}.
	
	\begin{figure}[ht]
		\centering
		\[
		\vcenter{\hbox{\begin{tikzpicture}[scale=0.7]
					\fill (-0.63,1.075)  circle (2pt);
					\node at (-0.9,1.4) {1};
					\fill (0.63,1.075)  circle (2pt);
					\node at (0.9,1.4) {2};
					\fill (1.22,0) circle (2pt);
					\node at (1.6,0) {3};
					\fill (0.63,-1.075)  circle (2pt);
					\node at (0.9,-1.4) {4};
					\fill (-0.63,-1.075)  circle (2pt);
					\node at (-0.9,-1.4) {5};
					\fill (-1.22,0) circle (2pt);
					\node at (-1.6,0) {6};
					\draw (-1.22,0)--(-0.63,1.075)--(0.63,1.075)--(1.22,0)--(0.63,-1.075)--(-0.63,-1.075)--cycle;
					\draw[dashed] (-1.6,0)[rounded corners=15pt]--(-0.4,1.7)[rounded corners=15pt]--(-0.4,-1.7)[rounded corners=12pt]--cycle;
					\draw[dashed] (1.6,0)[rounded corners=15pt]--(0.4,1.7)[rounded corners=15pt]--(0.4,-1.7)[rounded corners=12pt]--cycle;
		\end{tikzpicture}}}
		\quad\quad
		\left(\vcenter{\hbox{\begin{tikzpicture}[scale=0.6]
					\draw (0,0)--(0,1);
					\draw (0,1)--(1,2);
					\draw (0,1)--(-1,2);
					\node at (-1.2,2.3) {\small $\{1,5,6\}$};
					\node at (1,2.3) {\small $\{2,3,4\}$};
		\end{tikzpicture}}};
		\vcenter{\hbox{\begin{tikzpicture}[scale=0.6]
					\draw (0,0)--(0,1);
					\draw (0,1)--(1,2);
					\draw (0,1)--(-1,2);
					\draw (-1,2)--(-1.75,2.75);
					\draw (-1,2)--(-0.25,2.75);
					\node at (-1.75,3.1) {$1$};
					\node at (-0.25,3.1) {$6$};
					\node at (1,2.3) {$5$};
		\end{tikzpicture}}},
		\vcenter{\hbox{\begin{tikzpicture}[scale=0.6]
					\draw (0,0)--(0,1);
					\draw (0,1)--(1,2);
					\draw (0,1)--(-1,2);
					\draw (1,2)--(1.75,2.75);
					\draw (1,2)--(0.25,2.75);
					\node at (-1,2.3) {$2$};
					\node at (0.25,3.1) {$3$};
					\node at (1.75,3.1) {$4$};
		\end{tikzpicture}}}\right)
		\quad \longrightarrow
		\vcenter{\hbox{\begin{tikzpicture}[scale=0.6]
					\draw (0,0)--(0,1);
					\draw (0,1)--(1,2);
					\draw (0,1)--(-1,2);
					\draw (1,2)--(1.75,2.75);
					\draw (1,2)--(0.25,2.75);
					\draw (1.75,2.75)--(2.5,3.5);
					\draw (1.75,2.75)--(1,3.5);
					\draw (-1,2)--(-1.75,2.75);
					\draw (-1,2)--(-0.25,2.75);
					\draw (-1.75,2.75)--(-2.5,3.5);
					\draw (-1.75,2.75)--(-1,3.5);
					\node at (-2.5,3.85) {\small$1$};
					\node at (-1,3.85) {\small$6$};
					\node at (-0.25,3.1) {\small$5$};
					\node at (0.25,3.1) {\small$2$};
					\node at (1,3.85) {\small$3$};
					\node at (2.5,3.85) {\small$4$};
					
		\end{tikzpicture}}}
		\]
		\caption{Example of substitution.}
		\label{substitution}
	\end{figure}
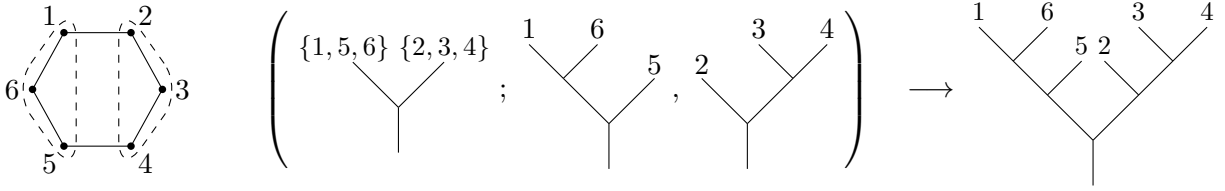
	
	For a $\Gr$-admissible tree $T$ and a vertex $v\in \Ver(T)$, we define the \textit{input graph} $\In(v)$ as follows: its vertices are input edges $e_1,\cdots,e_k$ of $v$ and two vertices $e_i,e_j$ are adjacent if the corresponding union of leaf sets $L_{e_i}\cup L_{e_j}$ forms a tube in $\Gr$, where $L_{e_i}$ is the set of leaves of subtree $T_{e_i}$.
	
	The grafting of trees allows us to give an explicit construction of free contractads. The free contractad generated by graphical collection $\E$ is the contractad $\T(\E)$, with components 
	\[
	\T(\E)(\Gr)=\bigoplus_{T \in \Tree(\Gr)} \E((T)), \text{ where } \E((T)):=\bigotimes_{v \in \Ver(T)} \E(\In(v)).
	\] It is helpful to think about an element of $\T(\E)(\Gr)$ as a sum of $\Gr$-admissible rooted tree where each vertex $v$ is decorated by an element of $\E(\In(v))$. The contractad structure $\T(\E)\circ\T(\E)\to\T(\E)$ is given by graftings of decorated rooted trees. 
	
	
	
	\begin{defi}
		The contractad presented by generators $\E$ and relations $\R \subset \T(\E)$ is the quotient contractad
		\[
		\T(\E)/\langle \R \rangle,
		\] where $\langle \R \rangle$ is the minimal ideal containing the subcollection $\R$. 
	\end{defi}
	\begin{notation*}
		\hfill\break
		\begin{itemize}
			\item[(i)]In practice we shall use 3 different notations while discussing presentations of contractads: tree notations, infinitesimal notations and algebraic notations. As an example, consider the infinitesimal product $\circ^{\Path_3}_{\{2,3\}}\colon \Pop(\Path_{\{1,\{2,3\}\}})\otimes \Pop(\Path_{\{2,3\}})\to \Pop(\Path_3)$ and elements $\alpha,\beta \in \Pop(\Path_2)$. The resulting element $\circ^{\Path_3}_{\{2,3\}}(\alpha,\beta)$ can be written in the 3 following ways
			\[
			\text{Tree Notation: } \vcenter{\hbox{\begin{tikzpicture}[
						scale=0.6,
						vert/.style={circle,  draw=black!30!black, thick, minimum size=1mm},
						leaf/.style={rectangle, thick, minimum size=1mm},
						edge/.style={-,black!30!black, thick},
						]
						\node[vert] (1) at (0,1) {\footnotesize$a$};
						\node[leaf] (l1) at (-0.75,2) {\footnotesize$1$};
						\node[vert] (2) at (0.75,2) {\footnotesize$b$};
						\node[leaf] (l2) at (1.5,3) {\footnotesize$3$};
						\node[leaf] (3) at (0,3) {\footnotesize$2$};
						\draw[edge] (0,0)--(1);
						\draw[edge] (1)--(2)--(3);
						\draw[edge] (1)--(l1);
						\draw[edge] (2)--(l2);
			\end{tikzpicture}}},\quad \text{Infinitesimal Notation: } a\circ^{\Path_3}_{\{2,3\}} b,\quad\text{Algebraic Notation: } a(x_1,b(x_2,x_3))
			\] When a contractad is generated in component $\Path_2$, we shall use the algebraic notation to give naive parallels with classical (ns)operads.
			\item[(ii)] We shall say that an element $\alpha\in \Pop(\Gr)$ is symmetric (resp. antisymmetric) if it is invariant (sign invariant) with respect to graph automorphism group $\alpha^\sigma=\alpha$ (resp. $\alpha^{\sigma}=\mathrm{sgn}(\sigma)\alpha$).
			\item[(iii)] If we have a relation in a contractad, we can produce new ones by graph automorphisms. For example, let $m$ be a symmetric element in component $\Path_2$ of a contractad satisfying the relation $m(x_1,m(x_2,x_3))-m(m(x_1,x_2),x_3)=0$ in $\K_3$. The action of transposition $(12)$ on this relation results in the following relation
			\[
			(m(x_1,m(x_2,x_3))-m(m(x_1,x_2),x_3))^{(12)}=m(m(x_1,x_3),x_2)-m(m(x_1,x_2),x_3)=0
			\]So, when we describe a presentation of contractads, we shall mention only a minimal set of generators/relations, omitting those that appear as a result of graph automorphisms.
		\end{itemize}
	\end{notation*}
	\subsection{Examples}
	\label{sec:sub::examples::contractads}
	In this subsection, we list several examples of contractads.  Most of them can be viewed as graphical counterparts of existing operads. 
	\subsubsection{Commutative and Lie contractads}
	\label{sec:sub::com::Lie}
	The simplest example of a contractad is the commutative contractad $\Com$ in the category of sets. This contractad has the components $\Com(\Gr)=\{*\}$ with the obvious infinitesimal compositions $\circ^{\Gr}_G\colon \Com(\Gr/G)\times \Com(\Gr|_G)\to\Com(\Gr)$ of the form $(*,*)\mapsto*$. Note that this contractad is a terminal object in the category of set-contractads.
	Let us recall basic facts about this contractad
	\begin{prop}
		\label{compres}
		The contractad $\Com$ is generated by a symmetric generator $m$ in component $\Path_2$, satisfying the relations
		\begin{gather}
			\text{In }\Path_3:\quad m(m(x_1,x_2),x_3))=m(x_1,m(x_2,x_3)),
			\\
			\text{In }\K_3:\quad m(m(x_1,x_2),x_3))=m(x_1,m(x_2,x_3)).
		\end{gather}
	\end{prop}
	
	From the presentation above, we see that this contractad is a graphical analogue of the commutative operad $\mathsf{Com}$. There is also an analogue of the Lie operad.
	\begin{defi} The Lie contractad $\Lie$ is the contractad generated by an anti-symmetric generator $b$ in component $\Path_2$, satisfying the relations
		\begin{align}
			\text{In }\Path_3:\quad & b(b(x_1,x_2),x_3))=b(x_1,b(x_2,x_3)),
			\\
			\text{In }\K_3:\quad &
			b(b(x_1,x_2),x_3)+b(b(x_3,x_1),x_2)+b(b(x_2,x_3),x_1)=0.
		\end{align}
	\end{defi}
	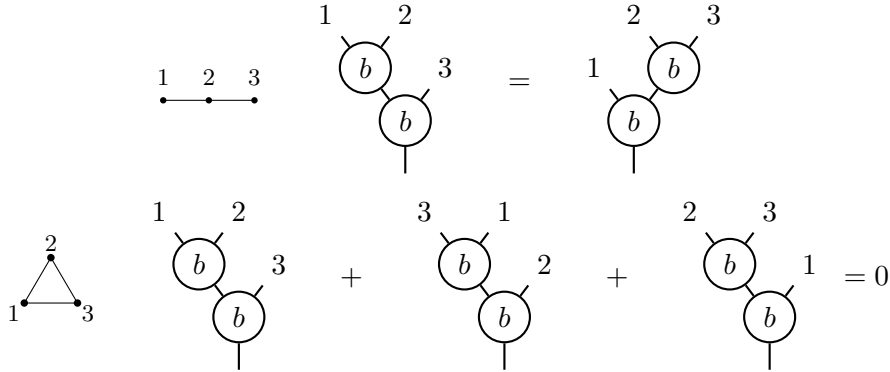
\begin{figure}[ht]
		\centering
		\caption{Relations in $\Lie$ in tree notations}
		\begin{gather*}
			\vcenter{\hbox{\begin{tikzpicture}[scale=0.6]
						\fill (0,0) circle (2pt);
						\node at (0,0.5) {\footnotesize$1$};
						\fill (1,0) circle (2pt);
						\node at (1,0.5) {\footnotesize$2$};
						\fill (2,0) circle (2pt);
						\node at (2,0.5) {\footnotesize$3$};
						\draw (0,0)--(1,0)--(2,0);    
			\end{tikzpicture}}}
			\quad
			\vcenter{\hbox{\begin{tikzpicture}[
						scale=0.7,
						vert/.style={circle,  draw=black!30!black, thick, minimum size=1mm},
						leaf/.style={circle, thick, minimum size=1mm},
						edge/.style={-,black!30!black, thick},
						]
						\node[leaf] (l1) at (-1.5,3) {$1$};
						\node[leaf] (l2) at (0,3) {$2$};
						\node[leaf] (l3) at (0.75,2) {$3$};
						\node[vert] (1) at (0,1) {$b$};
						\node[vert] (2) at (-0.75,2) {$b$};
						\draw[edge] (0,0)--(1);
						\draw[edge] (1)--(2);
						\draw[edge] (2)--(l1);
						\draw[edge] (2)--(l2);
						\draw[edge] (1)--(l3);
			\end{tikzpicture}}}
			\quad
			=
			\quad
			\vcenter{\hbox{\begin{tikzpicture}[
						scale=0.7,
						vert/.style={circle,  draw=black!30!black, thick, minimum size=1mm},
						leaf/.style={circle, thick, minimum size=1mm},
						edge/.style={-,black!30!black, thick},
						]
						\node[leaf] (l1) at (-0.75,2) {$1$};
						\node[leaf] (l2) at (0,3) {$2$};
						\node[leaf] (l3) at (1.5,3) {$3$};
						\node[vert] (1) at (0,1) {$b$};
						\node[vert] (2) at (0.75,2) {$b$};
						\draw[edge] (0,0)--(1);
						\draw[edge] (1)--(2);
						\draw[edge] (1)--(l1);
						\draw[edge] (2)--(l2);
						\draw[edge] (2)--(l3);
			\end{tikzpicture}}}
			\\
			\vcenter{\hbox{\begin{tikzpicture}[scale=0.7]
						\fill (-0.5,0) circle (2pt);
						\node at (-0.7,-0.2) {\footnotesize$1$};
						\fill (0.5,0) circle (2pt);
						\node at (0.7,-0.2) {\footnotesize$3$};
						\fill (0,0.86) circle (2pt);
						\node at (0,1.14) {\footnotesize$2$};
						\draw (-0.5,0)--(0,0.86)--(0.5,0)--cycle;    
			\end{tikzpicture}}}
			\quad
			\vcenter{\hbox{\begin{tikzpicture}[
						scale=0.7,
						vert/.style={circle,  draw=black!30!black, thick, minimum size=1mm},
						leaf/.style={circle, thick, minimum size=1mm},
						edge/.style={-,black!30!black, thick},
						]
						\node[leaf] (l1) at (-1.5,3) {$1$};
						\node[leaf] (l2) at (0,3) {$2$};
						\node[leaf] (l3) at (0.75,2) {$3$};
						\node[vert] (1) at (0,1) {$b$};
						\node[vert] (2) at (-0.75,2) {$b$};
						\draw[edge] (0,0)--(1);
						\draw[edge] (1)--(2);
						\draw[edge] (2)--(l1);
						\draw[edge] (2)--(l2);
						\draw[edge] (1)--(l3);
			\end{tikzpicture}}}
			\quad
			+
			\quad
			\vcenter{\hbox{\begin{tikzpicture}[
						scale=0.7,
						vert/.style={circle,  draw=black!30!black, thick, minimum size=1mm},
						leaf/.style={circle, thick, minimum size=1mm},
						edge/.style={-,black!30!black, thick},
						]
						\node[leaf] (l1) at (-1.5,3) {$3$};
						\node[leaf] (l2) at (0,3) {$1$};
						\node[leaf] (l3) at (0.75,2) {$2$};
						\node[vert] (1) at (0,1) {$b$};
						\node[vert] (2) at (-0.75,2) {$b$};
						\draw[edge] (0,0)--(1);
						\draw[edge] (1)--(2);
						\draw[edge] (2)--(l1);
						\draw[edge] (2)--(l2);
						\draw[edge] (1)--(l3);
			\end{tikzpicture}}}
			\quad
			+
			\quad
			\vcenter{\hbox{\begin{tikzpicture}[
						scale=0.7,
						vert/.style={circle,  draw=black!30!black, thick, minimum size=1mm},
						leaf/.style={circle, thick, minimum size=1mm},
						edge/.style={-,black!30!black, thick},
						]
						\node[leaf] (l1) at (-1.5,3) {$2$};
						\node[leaf] (l2) at (0,3) {$3$};
						\node[leaf] (l3) at (0.75,2) {$1$};
						\node[vert] (1) at (0,1) {$b$};
						\node[vert] (2) at (-0.75,2) {$b$};
						\draw[edge] (0,0)--(1);
						\draw[edge] (1)--(2);
						\draw[edge] (2)--(l1);
						\draw[edge] (2)--(l2);
						\draw[edge] (1)--(l3);
			\end{tikzpicture}}}=0
		\end{gather*}
	\end{figure}

	\subsubsection{Suspension of contractads}~\label{sec::susp} When we deal with homologically graded contractads, it is important to make a correct definition of "suspension". Similarly to operads, we define the suspension of contractads as follows. The Suspension contractad, denoted $\Susp$, is the contractad with components $\Susp(\Gr)=\Hom(\mathsf{k}s^{\otimes V_{\Gr}},\mathsf{k}s)$, where $\mathsf{k}s$ is a linear span of a homological degree 1 element $s$, with the contractad product given by the composition of functions,
	\begin{gather*}
		\gamma^{\Gr}_I\colon \Susp(\Gr/I)\otimes\bigotimes_{G\in I} \Susp(\Gr|_G) \to \Susp(\Gr)
		\\
		\gamma^{\Gr}_I(g;f_1,f_2,...,f_k)=g\circ(f_1\otimes f_2\otimes...\otimes f_k).
	\end{gather*} Note that each component $\Susp(\Gr)$ is a one-dimensional component concentrated in degree $(1-n)$ with $\Aut(\Gr)$-action by signs $(s^n\to s)^{\sigma}=(-1)^{|\sigma|}(s^n\to s)$ arising from Koszul sign rules. In general, for an integer $n$, we define the $n$-th suspension contractad $\Susp^{n}$ by replacing $s$ with $s^n$ in the definition. In particular, for $n=0$, we get the commutative contractad. For a contractad $\Pop$, we define its suspension $\Susp\Pop$ as the Hadamard product $\Susp\Pop:=\Susp\underset{\mathrm{H}}{\otimes}\Pop$ (take the usual tensor product componentwise) with an obvious contractad structure.
	\subsubsection{Little disks and Graphic configuration spaces}\label{sec:disks}
	
	The little $d$-disks contractad $\D_d$ is a topological contractad, generalising the little $d$-disks operad. Each component of this contractad $\D_d(\Gr)$ consists of configurations of $d$-dimensional disks in the unit disk labeled by the vertex set of a graph, such that interiors of disks corresponding to adjacent vertices do not intersect. So, an element of $\D_d(\Gr)$ is completely determined by a family of rectilinear\footnote{We only allow dilations and translations} embeddings $\{f_v\colon \mathbb{D}^d\to \mathbb{D}^d, v\in V_{\Gr}\}$ satisfying the edge-non-intersecting condition. The contractad structure is given by insertion of a disc in an interior disc, as in Figure~\ref{fig:disccomp}.
	
	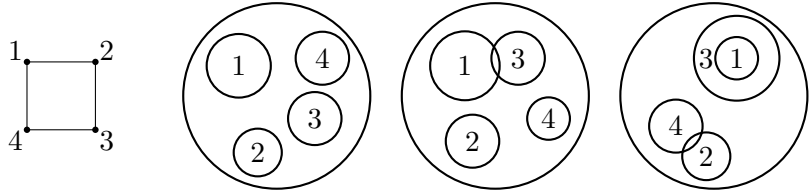
\begin{figure}[ht]
		\centering
		\[
		\vcenter{\hbox{\begin{tikzpicture}[scale=0.6]
					\fill (0,0) circle (2pt);
					\fill (0,1.5) circle (2pt);
					\fill (1.5,0) circle (2pt);
					\fill (1.5,1.5) circle (2pt);
					\draw (0,0)--(1.5,0)--(1.5,1.5)--(0,1.5)-- cycle;
					\node at (-0.25,1.75) {$1$};
					\node at (1.75,1.75) {$2$};
					\node at (1.75,-0.25) {$3$};
					\node at (-0.25,-0.25) {$4$};
		\end{tikzpicture}}}
		\quad
		\quad
		\vcenter{\hbox{\begin{tikzpicture}
					\draw[thick] (0,0) circle [radius=35pt];
					\draw[thick] (-0.5,0.4) circle [radius=12pt];
					\node at (-0.5,0.4) {1};
					\draw[thick] (-0.25,-0.75) circle [radius=9pt];
					\node at (-0.25,-0.75) {2};
					\draw[thick] (0.5,-0.3) circle [radius=10pt];
					\node at (0.5,-0.3) {3};
					\draw[thick] (0.6,0.5) circle [radius=10pt];
					\node at (0.6,0.5) {4};
		\end{tikzpicture}}}
		\quad
		\vcenter{\hbox{\begin{tikzpicture}
					\draw[thick] (0,0) circle [radius=35pt];
					\draw[thick] (-0.4,0.4) circle [radius=13pt];
					\node at (-0.4,0.4) {1};
					\draw[thick] (-0.3,-0.6) circle [radius=10pt];
					\node at (-0.3,-0.6) {2};
					\draw[thick] (0.3,0.5) circle [radius=10pt];
					\node at (0.3,0.5) {3};
					\draw[thick] (0.7,-0.3) circle [radius=8pt];
					\node at (0.7,-0.3) {4};    
		\end{tikzpicture}}}
		\quad
		\vcenter{\hbox{\begin{tikzpicture}
					\draw[thick] (0,0) circle [radius=35pt];
					\draw[thick] (0.3,0.5) circle [radius=16pt];
					\node at (-0.1,0.5) {3};
					\draw[thick] (-0.1,-0.8) circle [radius=9pt];
					\node at (-0.1,-0.8) {2};
					\draw[thick] (-0.5,-0.4) circle [radius=10pt];
					\node at (-0.5,-0.4) {4};
					\draw[thick] (0.3,0.5) circle [radius=8pt];
					\node at (0.3,0.5) {1};
		\end{tikzpicture}}}\]
		\caption{Example of configurations in $\D_2(\Cyc_4)$.}
	\end{figure}
	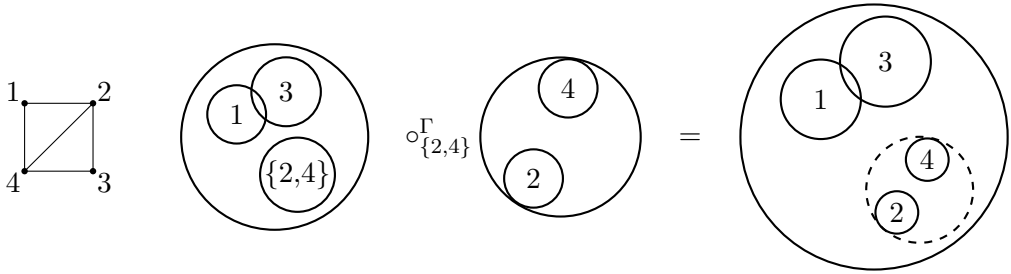
\begin{figure}[ht]
		\centering
		\[
		\vcenter{\hbox{\begin{tikzpicture}[scale=0.6]
					\fill (0,0) circle (2pt);
					\fill (0,1.5) circle (2pt);
					\fill (1.5,0) circle (2pt);
					\fill (1.5,1.5) circle (2pt);
					\draw (0,0)--(1.5,0)--(1.5,1.5)--(0,1.5)-- cycle;
					\draw (0,0)--(1.5,1.5);
					\node at (-0.25,1.75) {$1$};
					\node at (1.75,1.75) {$2$};
					\node at (1.75,-0.25) {$3$};
					\node at (-0.25,-0.25) {$4$};
		\end{tikzpicture}}}
		\quad
		\quad
		\vcenter{\hbox{\begin{tikzpicture}
					\draw[thick] (0,0) circle [radius=35pt];
					\draw[thick] (-0.5,0.3) circle [radius=11pt];
					\node at (-0.5,0.3) {1};
					\draw[thick] (0.3,-0.5) circle [radius=14pt];
					\node at (0.3,-0.5) {\{2,4\}};
					\draw[thick] (0.15,0.6) circle [radius=13pt];
					\node at (0.15,0.6) {3};
		\end{tikzpicture}}}
		\quad
		\circ^{\Gr}_{\{2,4\}}
		\vcenter{\hbox{\begin{tikzpicture}
					\draw[thick] (0,0) circle [radius=30pt];
					\draw[thick] (-0.355,-0.55) circle [radius=11pt];
					\node at (-0.355,-0.55) {2};
					\draw[thick] (0.1,0.645) circle [radius=11pt];
					\node at (0.1,0.645) {4};    
		\end{tikzpicture}}}
		\quad
		=
		\quad
		\vcenter{\hbox{\begin{tikzpicture}
					\draw[thick] (0,0) circle [radius=50pt];
					\draw[thick] (-0.7,0.5) circle [radius=15pt];
					\node at (-0.7,0.5) {1};
					\draw[dashed, thick] (0.6,-0.7) circle [radius=20pt];
					\draw[thick] (0.7,-0.3) circle [radius=8pt];
					\node at (0.7,-0.3) {4};
					\draw[thick] (0.3,-1) circle [radius=8pt];
					\node at (0.3,-1) {2};
					\draw[thick] (0.15,1) circle [radius=17pt];
					\node at (0.15,1) {3};
		\end{tikzpicture}}}
		\]
		\caption{Disks composition in $\D_2$}
		\label{fig:disccomp}
	\end{figure}
	
	There is a connection between little disks contractad and, so-called, \textit{graphic configuration spaces}. For a graph $\Gr$, the graphic configuration space of the real $d$-dimensional space $\mathbb{R}^d$ is the topological space, denoted $\Conf_{\Gr}(\mathbb{R}^d)$, consisting of functions $f\colon V_{\Gr}\to \mathbb{R}^d$ such that images of adjacent vertices do not coincide. There is the map from the component $\D_d(\Gr)$ of little disks contractad to the graphic configuration space, 
	\[
	r\colon \D_d(\Gr)\to \Conf_{\Gr}(\mathbb{R}^d), \quad (r_v\mathbb{D}^d+x_v)_{v\in V_{\Gr}}\mapsto (x_v)_{v\in V_{\Gr}},
	\] that maps configurations of disks to configurations of their centers. These maps are component wise homotopy equivalences.
	\subsubsection{Gerstenhaber contractad}\label{sec::gerst}
	If we replace each component of the contractad $\D_d$ with its homology groups, we obtain a
	contractad $H_{\bullet}(\D_d)$ in the category of $\mathbb{Z}$-modules. In the case $d=2$, The resulting contractad is called the Gerstenhaber contractad and denoted $\Gerst$.
	Let us recall some results about this contractad.
	\begin{theorem}\cite[Th.~5.2.1]{lyskov2023contractads}\label{thm:gerstpres}
		\begin{itemize}
			\item[(i)] The Gerstenhaber contractad is generated by two symmetric generators $m$ and $b$ in component $\mathsf{P}_2$ of degree $0$ and $1$, respectively, satisfying the relations
			\begin{equation}
				\label{eq::gerst::relations}
				\begin{array}{c}
					\text{In }\Path_3:\quad 
					\begin{cases}
						m(m(x_1,x_2),x_3) = m(x_1,m(x_2,x_3)),
						\\
						b(b(x_1,x_2),x_3)+b(x_1,b(x_2,x_3)) = 0,
						\\
						b(m(x_1,x_2),x_3) = m(x_1,b(x_2,x_3))
					\end{cases},   
					\\
					\text{In }\K_3:\quad 
					\begin{cases}
						m(m(x_1,x_2),x_3) = m(x_1,m(x_2,x_3))
						\\
						b(b(x_1,x_2),x_3)+b(b(x_3,x_1),x_2)+b(b(x_2,x_3),x_1) = 0
						\\
						b(m(x_1,x_2),x_3) = m(x_1,b(x_2,x_3)) + m(b(x_1,x_3),x_2)
					\end{cases}
				\end{array}
			\end{equation}
			\item[(ii)] The underlying graphical collection of $\Gerst$ is isomorphic to the product
			\[
			\Gerst\cong \Com\circ \Susp^{-1}\Lie,
			\] where $\Susp^{-1}\Lie=\Susp^{-1}\underset{\mathrm{H}}{\otimes}\Lie$ is a desuspuension.
		\end{itemize}
	\end{theorem}
	The homology contractads $H_{\bullet}(\D_d)$ for $d\geq 2$ are almost the same and differ only by the homological degree of the generator $b$ ($\deg(b)=d-1$).
	
	\subsubsection{Associative and Poisson contractads}\label{sec:Ass}
	As will be explained in the forthcoming paper of the second author~\cite{lyskov2024Hamilton} there are several different contractads that generalize the operad of associative algebras $\mathsf{Ass}$. 
	In this paper, we are interested in the one coming from the homology of the little balls operad.
	
	In contrast with the case $d\geq 2$, the components of little intervals contractad $\D_1(\Gr)\simeq \Conf_{\Gr}(\mathbb{R})$ are the disjoint union of contractible spaces. Indeed, the graphic configuration space $\Conf_{\Gr}(\mathbb{R})$ is the complement to a real hyperplane arrangement, so it is a disjoint union of convex hulls, one for each connected component. In particular, the homology of this contractad is concentrated in zero degree.
	
	The corresponding homology contractad $H_{\bullet}(\D_1)=H_{0}(\D_1)$ is called the associative contractad and denoted $\Ass$. Combinatorially, the component $\Ass(\Gr)$ is a linear span of vertex orderings $(v_1,v_2,\cdots,v_n)$ subject to the equivalence $(\cdots,v_i,v_{i+1},\cdots)\sim (\cdots,v_{i+1},v_{i},\cdots)$ for non-adjacent $v_i,v_{i+1}$. Indeed, the ordering of interval labeling from left to right in a configuration in $\D_1(\Gr)$ defines a vertex-ordering up to the equivalence described before (the equivalence coming from the condition that disks labeled by non-adjacent vertices could intersect) and this vertex-ordering depends only on connected components. 
	
	Let us state the presentation of this contractad. The space $\D_1(\Path_2)$ consists of two connected components $\{z_1>z_2\}$, $\{z_1<z_2\}$, where $z_1,z_2$ are centers of labeled 1-disks. Let $\nu \in H_0(\D_2(\Path_2))$ be a class of the point belonging to the first connected component. The action of transposition $(12)\in \Aut(\Path_2)$ on this generator results in the remaining generator $\nu^{\op}$.
	
	\begin{theorem}~\cite[Th.~5.1.1]{lyskov2023contractads}
		The associative contractad $\Ass$ is generated by an element $\nu$ and its opposite $\nu^{\op}:=\nu^{(12)}$ in the component $\Path_2$, satisfying the relations
		\begin{gather*}
			\text{In }\Path_3:\quad 
			\begin{cases}
				\nu(\nu(x_1,x_2),x_3) = \nu(x_1,\nu(x_2,x_3)), \\ 
				\nu^{\op}(\nu^{\op}(x_1,x_2),x_3) = \nu^{\op}(x_1,\nu^{\op}(x_2,x_3)),  
				\\
				\nu(\nu^{\op}(x_1,x_2),x_3) = \nu^{\op}(x_1,\nu(x_2,x_3)), \\
				\nu^{\op}(\nu(x_1,x_2),x_3) = \nu(x_1,\nu^{\op}(x_2,x_3)). 
			\end{cases} 
			\\
			\text{In }\K_3:\quad 
			\nu(\nu(x_1,x_2),x_3) = \nu(x_1,\nu(x_2,x_3)) \ \& \text{ permutations of this relation. }    
		\end{gather*}
	\end{theorem}
	\begin{example}
		The "strange" relation $\nu^{\op}(\nu(x_1,x_2),x_3) = \nu(x_1,\nu^{\op}(x_2,x_3))$  in $\Path_3$ illustrates the fact that two configurations below are in the same connected components since vertices $1$ and $3$ are non-adjacent.
		\begin{gather*}
			\vcenter{\hbox{\begin{tikzpicture}[
						scale=0.5,
						]
						\draw (0,0)--(6,0);
						\node (bl) at (0,0) {[};
						\node (i1) at (1.5,0.5) {$3$};
						\node (i1l) at (1,0) {[};
						\node (i1r) at (2,0) {]};
						\node (i2) at (4,0.5) {$\{1,2\}$};
						\node (i2l) at (3,0) {[};
						\node (i2r) at (5,0) {]};
						\node (br) at (6,0) {]};
			\end{tikzpicture}}} \circ^{\Path_3}_{\{1,2\}}
			\vcenter{\hbox{\begin{tikzpicture}[
						scale=0.5,
						]
						\draw (0,0)--(5,0);
						\node (bl) at (0,0) {[};
						\node (i1) at (1.5,0.5) {$1$};
						\node (i1l) at (1,0) {[};
						\node (i1r) at (2,0) {]};
						\node (i2) at (3.5,0.5) {$2$};
						\node (i2l) at (3,0) {[};
						\node (i2r) at (4,0) {]};
						\node (br) at (5,0) {]};
			\end{tikzpicture}}}=
			\vcenter{\hbox{\begin{tikzpicture}[
						scale=0.5,
						]
						\draw (0,0)--(7,0);
						\node (bl) at (0,0) {[};
						\node (i1) at (1.5,0.5) {$3$};
						\node (i1l) at (1,0) {[};
						\node (i1r) at (2,0) {]};
						\node (i2) at (3.5,0.5) {$1$};
						\node (i2l) at (3,0) {[};
						\node (i2r) at (4,0) {]};
						\node (i3) at (5.5,0.5) {$2$};
						\node (i3l) at (5,0) {[};
						\node (i3r) at (6,0) {]};
						\node (br) at (7,0) {]};
			\end{tikzpicture}}}
			\\
			\vcenter{\hbox{\begin{tikzpicture}[
						scale=0.5,
						]
						\draw (0,0)--(6,0);
						\node (bl) at (0,0) {[};
						\node (i1) at (1.5,0.5) {$1$};
						\node (i1l) at (1,0) {[};
						\node (i1r) at (2,0) {]};
						\node (i2) at (4,0.5) {$\{2,3\}$};
						\node (i2l) at (3,0) {[};
						\node (i2r) at (5,0) {]};
						\node (br) at (6,0) {]};
			\end{tikzpicture}}} \circ^{\Path_3}_{\{2,3\}}
			\vcenter{\hbox{\begin{tikzpicture}[
						scale=0.5,
						]
						\draw (0,0)--(5,0);
						\node (bl) at (0,0) {[};
						\node (i1) at (1.5,0.5) {$3$};
						\node (i1l) at (1,0) {[};
						\node (i1r) at (2,0) {]};
						\node (i2) at (3.5,0.5) {$2$};
						\node (i2l) at (3,0) {[};
						\node (i2r) at (4,0) {]};
						\node (br) at (5,0) {]};
			\end{tikzpicture}}}=
			\vcenter{\hbox{\begin{tikzpicture}[
						scale=0.5,
						]
						\draw (0,0)--(7,0);
						\node (bl) at (0,0) {[};
						\node (i1) at (1.5,0.5) {$1$};
						\node (i1l) at (1,0) {[};
						\node (i1r) at (2,0) {]};
						\node (i2) at (3.5,0.5) {$3$};
						\node (i2l) at (3,0) {[};
						\node (i2r) at (4,0) {]};
						\node (i3) at (5.5,0.5) {$2$};
						\node (i3l) at (5,0) {[};
						\node (i3r) at (6,0) {]};
						\node (br) at (7,0) {]};
			\end{tikzpicture}}}
		\end{gather*} Equivalently, in terms of vertex-orderings, we have $$\nu^{\op}(\nu(x_1,x_2),x_3)= (3,1,2)\sim (1,3,2)=\nu(x_1,\nu^{\op}(x_2,x_3)).$$
	\end{example}
	To make a more straightforward comparison of $\Ass$ and $\Gerst$ we may use generators $m:=\nu+\nu^{\op}$ and $b:=\nu-\nu^{\op}$ to be the "odd" and "even" part of $\nu$. In the operad case, these generators correspond to Jordan and Lie brackets, respectively. By direct computations, we see that the relations in new generators have the following form
	\begin{gather}\label{eq::jordanlie_pres}
		\text{In }\Path_3:\quad 
		\begin{cases}
			m(m(x_1,x_2),x_3) = m(x_1,m(x_2,x_3)), \\
			b(b(x_1,x_2),x_3) + b(x_1,b(x_2,x_3))=0, \\
			m(b(x_1,x_2),x_3) = b(x_1,m(x_2,x_3))
		\end{cases}, 
		\\
		\label{eq::Ass::rel}
		\text{In }\K_3:\quad 
		\begin{cases}
			m(m(x_1,x_2),x_3) - m(x_1,m(x_2,x_3)) = b(b(x_1,x_3),x_2) \\
			b(m(x_1,x_2),x_3) = m(x_1,b(x_2,x_3))+m(b(x_1,x_3),x_2),\\
			b(b(x_1,x_2),x_3) + b(b(x_2,x_3),x_1)+b(b(x_3,x_1),x_2)=0.
		\end{cases}
	\end{gather}
	The map $b\mapsto b:=\nu-\nu^{op}$ defines an embedding of the contractads $\Lie\hookrightarrow\Ass$ generalizing the one known for operads. We can assign a filtration $\calF$ of the contractad $\Ass$ by the number of the operations $m$, which is compatible with the contractad structure.
	\begin{prop}
		\label{prop::Ass::Pois}
		The associated graded contractad with respect to the filtration $\calF$  is isomorphic to the Poisson contractad:
		\[
		\mathrm{gr}_{\calF} \Ass \cong \Pois.
		\]
	\end{prop}
	\begin{proof}
		Poisson contractad 
		$\Pois$ is generated by symmetric binary operation $m\in\Path_2$ and skewsymmetric binary operation $b\in\Path_2$ of homological degree $0$ yielding Relations~\eqref{eq::gerst::relations} of the Gerstenhaber contractad.
		These relations differ from~\eqref{eq::Ass::rel} in one term which tells that the Jordan multiplication is associative modulo terms with Lie brackets.
	\end{proof}
	
	\begin{remark}\label{rem::associahedron}
		It is worth mentioning, that the little disks operads $\D_d$ admit a nice topological realization given by Fulton-MacPherson compactifications of the configuration spaces $\mathsf{FM}_d$ (see e.g.~\cite{lambrechts2014formality} for detailed accurate description). In the case $d=1$ the corresponding topological spaces are homeomorphic to the disjoint union of the famous Stasheff polytopes~\cite{stasheff1963homotopy} whose faces are labelled by planar trees. There exists a straightforward generalization of the Fulton-MacPherson compactifications to the world of contractad.
		The compactification comes with natural stratification that is compatible with the contractad structure that we delay to a future paper. In particular, for the case, $d=1$, we expect that the corresponding semialgebraic manifolds $\FM_1(\Gamma)$ are homeomorphic to the disjoint union of convex polytope generalising Stasheff polytopes.
	\end{remark}

	\subsection{Koszul duality for contractads}
	\label{sec::Koszul::contractads}
	We shall briefly recall a Bar-Cobar adjunction for contractads and Koszul duality theory for quadratic contractads. For details, see~\cite[\S 3]{lyskov2023contractads}. Let us mention, that most constructions are completely analoguesly to ones in the (ns)operadic case.

	For a pair of graphical collections $\Pop,\Q$, we define its \textit{infinitesimal product} by the rule
	\[
	(\Pop\circ'\Q)(\Gr)=\bigoplus_{G}\Pop(\Gr/G)\otimes\Q(\Gr|_G),
	\] where the sum ranges over all tubes.

	\subsubsection{Bar-Cobar construction} Given an augmented dg contractad $(\Pop,d_{\Pop})$, there is a quasi-free dg contractad $\mathsf{B}\Pop$, called Bar construction. This is the free contractad generated by the suspension of the augmentation ideal $s\overline{\Pop}$, $\overline{\Pop}=\ker(\epsilon\colon \Pop\to \mathbb{1})$,
	\[
	\mathsf{B}\Pop=(\T^c(s\overline{\Pop}),d),
	\] with the differential $d=d_1+d_2$, where $d_1$ and $d_2$ are unique coderivations of the free cocontractad induced from the differential $d_{\Pop}$ and contractad structure respectively
	\begin{gather*}
		\T^c(s\overline{\Pop})\twoheadrightarrow s\overline{\Pop}\overset{d_{\Pop}}{\rightarrow} s\overline{\Pop}\quad \rightsquigarrow \quad d_1\colon \T^c(s\overline{\Pop})\to \T^c(s\overline{\Pop})
		\\
		\T^c(s\overline{\Pop})\twoheadrightarrow \T^{(2)}(s\overline{\Pop})\cong s\overline{\Pop}\circ's\overline{\Pop}\overset{\gamma'}{\longrightarrow} s\overline{\Pop} \quad \rightsquigarrow \quad d_2\colon \T^c(s\overline{\Pop})\to \T^c(s\overline{\Pop})
	\end{gather*}  
	When a contractad $\Pop$ is a weight-graded contractad, the Bar construction $\mathsf{B}\Pop$ admits an additional cohomology grading given by \textit{syzygy degrees}:  $\omega(s\alpha_1\otimes s\alpha_2\otimes\cdots\otimes s\alpha_k)=\omega(\alpha_1)+\cdots+\omega(\alpha_k)-k$, where $\omega(\alpha_i)$ are weights of elements from $\Pop$.
	
	In a similar way, for a dg coaugmented cocontractad $\Q$ we define a quasi-free dg cocontrcatad $\mathsf{\Omega}\Q=(\T(s^{-1}\overline{\Q}),d_1+d_2)$, called Cobar construction. The bar and cobar constructions form an adjoint pair
	\[
	\Hom_{\mathsf{dgCon}}(\mathsf{\Omega}\Q,\Pop)\cong\Hom_{\mathsf{dgCon}}(\Q,\mathsf{B}\Pop).
	\] Furthermore, the corresponding (co)unit
	\[
	\mathsf{\Omega}\mathsf{B}\Pop\overset{\simeq}{\longrightarrow} \Pop, \quad  \Q\overset{\simeq}{\longrightarrow}\mathsf{B}\mathsf{\Omega}\Q
	\] is a quasi-isomorphism of (co)contractads.

	\subsubsection{Koszul contractads} Let $\Pop=\Pop(\E,\R)$ be a quadratic contractad with generator $\E$ and relations $\R\subset \E\circ'\E$. The Koszul dual cocontractad $\Pop^{\cokoszul}$ is a cocontractad $\Pop^{\cokoszul}=\Q(s\E,s^2\R)$ with a cogenerators $s\E$ and corelations $s^2\R$. The Koszul dual contractad $\Pop^!$ is defined by the formula
	\begin{equation}
		\label{eq::Koszul::dual}    
		\Pop^!:=\Susp^{-1}(\Pop^{\cokoszul})^*,
	\end{equation} where $\Susp^{-1}$ is a desuspension (see \S\ref{sec::susp})
	When the space of generators $\E$ is finite-dimensional component wise, $\Pop^!$ admits the following quadratic presentation
	\[
	\Pop^!=\Pop(s^{-1}\Susp^{-1}\E^*,\R^{\bot})
	\] 
	where the space of relation $\R^{\bot}$ is the orthogonal with respect to the following pairing 
	$$\langle-,-\rangle\colon (\E\circ'\E)\otimes (s^{-1}\Susp^{-1}\E^*\circ' s^{-1}\Susp^{-1}\E^*)\rightarrow \mathsf{k}.$$ 
	Note that, the homological shift is hidden in the pairing due to appropriate suspensions.
	\begin{remark}
		The suspension in the dualization~\eqref{eq::Koszul::dual} was initially suggested by Ginzburg and Kapranov in~\cite{ginzburg1994koszul}. 
		This choice of suspensions is good enough for operads generated by binary operations. Indeed, if $\Pop(\E,\R)$ is a quadratic operad that does not have extra homological grading ($\E$ is an ordinary vector space), the generators $s^{-1}\Susp^{-1}\E^*$ of the Koszul dual operad $\Pop^{!}$ also belongs to an ordinary vector space and all homological shifts are canceled.
		
		While working with operads/contractads whose generators are of different degrees as for a gravity operad/contractad the more comfortable choice of suspensions is different.
		These phenomena were first noticed by Getzler in~\cite{getzler1994two} and later on clarified in~\cite{loday2012algebraic} whose notations we follow in this exposition.
	\end{remark}
	
	\begin{example*} For the commutative contractad $\Com$, its Koszul dual $\Com^!\cong \Lie$ is the Lie contractad. The contractad $\Gerst$ is self-dual up to suspension $\Gerst^!\cong \Susp\Gerst$.
	\end{example*}
	
	Similarly to contractads, there is so-called the \textit{Koszul complex} $\mathcal{K}(\Pop):=(\Pop^{\cokoszul}\circ \Pop,d_{\kappa})$, where the differential  has the form
	\[
	d_{\kappa}\colon \Pop^{\cokoszul}\circ\Pop\overset{\triangle'}{\rightarrow} (\Pop^{\cokoszul}\circ' s\E)\circ \Pop\overset{s^{-1}}{\rightarrow} (\Pop^{\cokoszul}\circ'\E) \circ\Pop\overset{\gamma}{\rightarrow} \Pop^{\cokoszul}\circ \Pop.
	\] Also, there are Koszul inclusion and projection $\Pop^{\cokoszul}\to\mathsf{B}^{0}\Pop$, $\mathsf{\Omega}_{0}\Pop^{\cokoszul}\to \Pop$ defined analogously to the operadic case.
	\begin{defi}
		A quadratic contractad $\Pop$ is Koszul if the Koszul complex $\mathcal{K}(\Pop)$ is acyclic (in the unital sense). Equivalently, the Koszul projection/inclusion is a quasi-isomorphism of (co)contractads.
	\end{defi}
	\begin{prop}\cite{lyskov2023contractads}
		The contractads $\Com$, $\Lie$, $\Gerst$, $\Ass$ and $\Pois$ are Koszul.
	\end{prop}
	\begin{proof}
		It is shown in~\cite{lyskov2023contractads} that for all these contractads the quadratic relations we presented assemble a Gr\"obner basis of relations. We briefly recall the description of Gr\"obner bases for contractads in \S\ref{sec::Grobner::contractad} and refer the reader for more details to~\cite{lyskov2023contractads}.
	\end{proof}
	Let us record the following Corollary, which will be useful later.
	\begin{corollary}
		\label{cor::koszul::quadratic}
		Let $\Pop$ be a weight graded contractad. If the cohomology of the bar construction $\mathsf{B}^{\bullet}\Pop$ is concentrated in syzygy degree $0$, then the contractad $\Pop$ has a quadratic presentation and is Koszul.
	\end{corollary}
	
	\subsubsection{Distributive law} It is possible to adapt the formalism of distributive laws~\cite{markl1996distributive} to the case of contractads.
	
	\begin{defi}[distributive law]
		Let $\Pop$ and $\Q$ be two contractads. We say that a morphism of graphical collections
		\[
		\Lambda\colon \Q\circ\Pop\to\Pop\circ\Q
		\] defines a distributive law between $\Pop$ and $\Q$ if the composite 
		\[
		\Pop\circ\Q\circ\Pop\circ\Q\stackrel{\Id\circ\Lambda\circ\Id}{\longrightarrow}  
		\Pop\circ\Pop\circ\Q\circ\Q
		\stackrel{\gamma_{\Pop}\circ\mu_{\Q}}{\longrightarrow}\Pop\circ\Q
		\]
		defines a contractad structure on the graphical collection $\Pop\circ\Q$.
	\end{defi} For a pair of quadratic contractads $\Pop=\Pop(\E,\R)$ and $\Q=\Pop(\mathcal{V},\mathcal{H})$ and a morphism of graphical collections $\lambda\colon \mathcal{V}\circ'\E\to \E\circ'\mathcal{V}$, called a \textit{rewriting rule}, we could define a new contractad $\Pop\vee_{\lambda}\Q$ with generators $\E\oplus \mathcal{V}$ and relations $\R\oplus\mathcal{H}\oplus \langle \alpha-\gamma(\alpha)|\alpha\in \E\circ'\mathcal{V}\rangle$. Note that its Koszul dual contractad is also a contractad obtained from the rewriting rule $\lambda^!$
	\begin{gather*}
		(\Pop\vee_{\lambda}\Q)^!\cong \Q^!\vee_{\lambda^!}\Pop^!,
		\\
		\lambda^!:=s^{-2}\Susp^{-1}\gamma^*\colon s^{-1}\Susp^{-1}\E^*\circ' s^{-1}\Susp^{-1}\mathcal{V}^*\to s^{-1}\Susp^{-1}\mathcal{V}^*\circ' s^{-1}\Susp^{-1}\E^*.   
	\end{gather*}
	\begin{prop}\label{prop::distributive_law} We have
		\begin{itemize}
			\item[(i)] For every choice of rewriting rule $\lambda$ there is a surjective morphism $\Pop\circ\Q\twoheadrightarrow \Pop\vee_{\lambda}\Q$. If that map is an isomorphism, the composite
			\[
			\Q\circ\Pop\hookrightarrow (\Pop\vee_{\lambda}\Q)\circ(\Pop\vee_{\lambda}\Q)\to \Pop\vee_{\lambda}\Q\cong \Pop\circ\Q,
			\] is a distributive law.
			\item[(ii)] If $\gamma$ defines a distributive law, then its Koszul dual $\gamma^!$ also defines it. 
			\item[(iii)] If both $\Pop$ and $\Q$ are Koszul and $\lambda$ defines a distributive law, then $\Pop\vee_{\lambda}\Q$ is also Koszul.
		\end{itemize}
	\end{prop}
	\begin{example}
		\label{ex::distib::law}
		For example, Relations~\eqref{eq::gerst::relations} of the  contractad $\Gerst$ defines a rewriting rule for the contractads $\Com$ and $\Susp^{-1}\Lie$. The decomposition $\Gerst\cong \Com\circ\Susp^{-1}\Lie$ explains why the rewriting rules produces a distributive law. In particular, thanks to $\Com$ and $\Lie$ are Koszul contractads, so $\Gerst$ is also Koszul. 
		We have the same distributive law for the Poisson contractad $\Pois$, so $\Pois$ is also Koszul.   
	\end{example}
	
	\subsection{Gr\"obner bases}\label{sec::Grobner::contractad}
	
	The efficient tool for construction of monomial basis in algebras or operads is a technique of Gr\"obner bases. We shall briefly recall the constructions of Gr\"obner basis in the case of contractads, for more details see~\cite[\S 4]{lyskov2023contractads}. Let us note to the reader that most of the constructions are completely analogous to those in the operadic case~ \cite{dotsenko2010grobner}.
	
	Similarly to operads~\cite{dotsenko2010grobner}, to introduce monomial orders for contractads, we need to consider a "shuffle" version of contractads that ignores actions of graph-automorphisms. The right way is to replace graphs with vertex-ordered ones. We define a \textit{non-symmetric graphical collection} as a contravariant functor $\Pop\colon \mathsf{OCGr}^{\mathrm{op}}\to \C$, where  $\mathsf{OCGr}$ is the category of vertex-ordered connected simple graphs with order-preserving isomorphisms. We define a \textit{shuffle contractad} as a monoid in the category of ns graphical collections equipped with the \textit{Shuffle contraction product}:
	\begin{equation*}
		(\Pop\circ_{\Sha}\Q)(\Gr,<)=\bigoplus_{I \in \parti(\Gr)} \Pop(\Gr/I,<^{\ind})\otimes\bigotimes_{\substack{I=\{G_1,G_2,...,G_k\}\\ \min G_1<\cdots<\min G_k}} \Q(\Gr|_{G_i},<_{\res}),
	\end{equation*} where $<_{\res}$ is the restriction of $<$ on the induced subgraph $\Gr|_G$, and $<^{\ind}$ is the order on the contracted graph $\Gr/I$ defined by the rule $ \{G_i\}<^{\ind} \{G_j\}:= \min_{v\in G_i} v<\min_{v\in G_j} v$. 
	
	The free shuffle contractads $\T_{\Sha}(\E)$ are defined analogously to usual contractads. Monomials in free shuffle contractads are called shufle monomials/trees. Due to the ordering of vertices in graphs, each shuffle tree admits a canonical planar embedding, \textit{shuffle presentation}.
	
	A monomial order on the free shuffle contractad is a total order on the set of shuffle monomials that is compatible with the contractad structure (contractad compositions are monotone functions). Given an ideal $\I\subset \T_{\Sha}(\E)$ in the free shuffle contractad with monomial order, its \textit{Gr\"obner basis} is a collection of elements $\G=\{f_i\}\subset \I$ such that shuffle monomials that are not divisible by leading terms $\LT(f_i)$ of elements from $\G$ form a basis of the quotient contractad $\T_{\Sha}(\E)/\I$.
	
	Given a usual contractad $\Pop$, we define its shuffle version $\Pop^{\forget}$ by the rule $\Pop^{\forget}(\Gr,<):=\Pop(\Gr)$. Let us note that this correspondence preserves presentations of contractads $\T(\E)/(\R)^{\forget}=\T_{\Sha}(\E^{\forget})/(\R^{\forget})$. We say that a quadratic contractad $\Pop=\Pop(\E,\R)$ admits a quadratic Gr\"obner basis if the shuffle ideal $(\R^{\forget})\subset \T(\E^{\forget})$ admits a Gr\"obner basis $\G$ with $\G\subset \R^{\forget}$. Analogously to the operad case, a contractad with quadratic Gr\"obner basis is Koszul~\cite[Th.~4.3.1]{lyskov2023contractads}.
	
	\begin{example}\label{ex::grobner_for_classical_contractads}
		We consider the $\grpermlex$-order on shuffle monomials, that is defined in the same way as for operads~\cite[\S 3.2.1]{dotsenko2010grobner}. \begin{itemize}
			\item 
			The shuffle contractad $\Lie^{\forget}$ has a quadratic Gr\"obner basis with respect to $\grpermlex$-order~\cite[Cor 4.3.1]{lyskov2023contractads}. The leading terms of relations $\R^{\forget}_{\Lie}$ with respect to this order have the form
			\[\label{leadtermlie}
			\vcenter
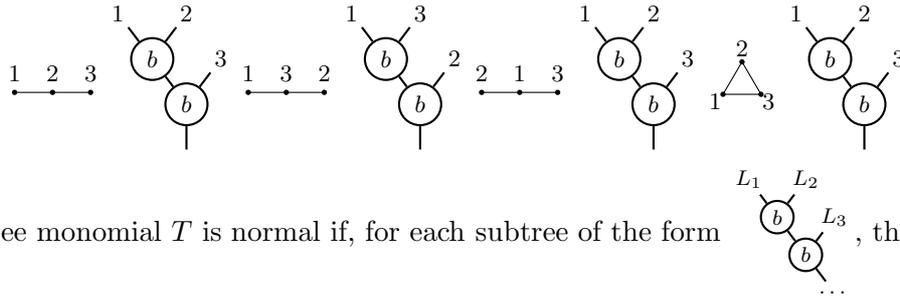
{\hbox{\begin{tikzpicture}[scale=0.5]
						\fill (0,0) circle (2pt);
						\node at (0,0.5) {\footnotesize$1$};
						\fill (1,0) circle (2pt);
						\node at (1,0.5) {\footnotesize$2$};
						\fill (2,0) circle (2pt);
						\node at (2,0.5) {\footnotesize$3$};
						\draw (0,0)--(1,0)--(2,0);    
			\end{tikzpicture}}}
			\vcenter{\hbox{\begin{tikzpicture}[
						scale=0.6,
						vert/.style={inner sep=3pt, circle,  draw, thick},
						leaf/.style={inner sep=2pt,rectangle},
						edge/.style={-,black!30!black, thick},
						]
						\node[vert] (1) at (0,1) {\footnotesize$b$};
						\node[leaf] (l1) at (0.75,2) {\footnotesize$3$};
						\node[vert] (2) at (-0.75,2) {\footnotesize$b$};
						\node[leaf] (l2) at (0,3) {\footnotesize$2$};
						\node[leaf] (3) at (-1.5,3) {\footnotesize$1$};
						\draw[edge] (0,0)--(1);
						\draw[edge] (1)--(2)--(3);
						\draw[edge] (1)--(l1);
						\draw[edge] (2)--(l2);
			\end{tikzpicture}}}
			\vcenter{\hbox{\begin{tikzpicture}[scale=0.5]
						\fill (0,0) circle (2pt);
						\node at (0,0.5) {\footnotesize$1$};
						\fill (1,0) circle (2pt);
						\node at (1,0.5) {\footnotesize$3$};
						\fill (2,0) circle (2pt);
						\node at (2,0.5) {\footnotesize$2$};
						\draw (0,0)--(1,0)--(2,0);    
			\end{tikzpicture}}}
			\vcenter{\hbox{\begin{tikzpicture}[
						scale=0.6,
						vert/.style={inner sep=3pt, circle,  draw, thick},
						leaf/.style={inner sep=2pt,rectangle},
						edge/.style={-,black!30!black, thick},
						]
						\node[vert] (1) at (0,1) {\footnotesize$b$};
						\node[leaf] (l1) at (0.75,2) {\footnotesize$2$};
						\node[vert] (2) at (-0.75,2) {\footnotesize$b$};
						\node[leaf] (l2) at (0,3) {\footnotesize$3$};
						\node[leaf] (3) at (-1.5,3) {\footnotesize$1$};
						\draw[edge] (0,0)--(1);
						\draw[edge] (1)--(2)--(3);
						\draw[edge] (1)--(l1);
						\draw[edge] (2)--(l2);
			\end{tikzpicture}}}
			\vcenter{\hbox{\begin{tikzpicture}[scale=0.5]
						\fill (0,0) circle (2pt);
						\node at (0,0.5) {\footnotesize$2$};
						\fill (1,0) circle (2pt);
						\node at (1,0.5) {\footnotesize$1$};
						\fill (2,0) circle (2pt);
						\node at (2,0.5) {\footnotesize$3$};
						\draw (0,0)--(1,0)--(2,0);    
			\end{tikzpicture}}}
			\vcenter{\hbox{\begin{tikzpicture}[
						scale=0.6,
						vert/.style={inner sep=3pt, circle,  draw, thick},
						leaf/.style={inner sep=2pt,rectangle},
						edge/.style={-,black!30!black, thick},
						]
						\node[vert] (1) at (0,1) {\footnotesize$b$};
						\node[leaf] (l1) at (0.75,2) {\footnotesize$3$};
						\node[vert] (2) at (-0.75,2) {\footnotesize$b$};
						\node[leaf] (l2) at (0,3) {\footnotesize$2$};
						\node[leaf] (3) at (-1.5,3) {\footnotesize$1$};
						\draw[edge] (0,0)--(1);
						\draw[edge] (1)--(2)--(3);
						\draw[edge] (1)--(l1);
						\draw[edge] (2)--(l2);
			\end{tikzpicture}}}
			\vcenter{\hbox{\begin{tikzpicture}[scale=0.5]
						\fill (-0.5,0) circle (2pt);
						\node at (-0.7,-0.2) {\footnotesize$1$};
						\fill (0.5,0) circle (2pt);
						\node at (0.7,-0.2) {\footnotesize$3$};
						\fill (0,0.86) circle (2pt);
						\node at (0,1.2) {\footnotesize$2$};
						\draw (-0.5,0)--(0,0.86)--(0.5,0)--cycle;    
			\end{tikzpicture}}}
			\vcenter{\hbox{\begin{tikzpicture}[
						scale=0.6,
						vert/.style={inner sep=3pt, circle,  draw, thick},
						leaf/.style={inner sep=2pt,rectangle},
						edge/.style={-,black!30!black, thick},
						]
						\node[vert] (1) at (0,1) {\footnotesize$b$};
						\node[leaf] (l1) at (0.75,2) {\footnotesize$3$};
						\node[vert] (2) at (-0.75,2) {\footnotesize$b$};
						\node[leaf] (l2) at (0,3) {\footnotesize$2$};
						\node[leaf] (3) at (-1.5,3) {\footnotesize$1$};
						\draw[edge] (0,0)--(1);
						\draw[edge] (1)--(2)--(3);
						\draw[edge] (1)--(l1);
						\draw[edge] (2)--(l2);
			\end{tikzpicture}}}
			\]
			So, a tree monomial $T$ is normal if, for each  subtree of the form $\vcenter
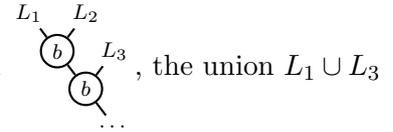
{\hbox{\begin{tikzpicture}[
						scale=0.5,
						vert/.style={inner sep=2pt, circle,draw, thick},
						leaf/.style={inner sep=2pt, rectangle, thick},
						edge/.style={-,black!30!black, thick},
						]
						\node[vert] (1) at (0,1) {\scriptsize$b$};
						\node[leaf] (l1) at (0.75,2) {\scriptsize$L_3$};
						\node[vert] (2) at (-0.75,2) {\scriptsize$b$};
						\node[leaf] (l2) at (0,3) {\scriptsize$L_2$};
						\node[leaf] (3) at (-1.5,3) {\scriptsize$L_1$};
						\node[leaf] (dd) at (0.75,0) {\scriptsize$\cdots$};
						\draw[edge] (dd)--(1);
						\draw[edge] (1)--(2)--(3);
						\draw[edge] (1)--(l1);
						\draw[edge] (2)--(l2);
			\end{tikzpicture}}}$, the union $L_1 \cup L_3$ is a tube and $\min L_2 >\min L_3$.
			
			\item   Dually, the shuffle contractad $\Com^{\forget}$ admits a quadratic Gr\"obner basis with respect to reverse $\grpermlex$-order~\cite[Prop 4.3.2]{lyskov2023contractads}. For an ordered graph $\Gr$, the unique normal monomial, denoted $m_{\Gr}$, has the form
			\[
			\vcenter
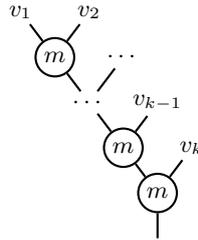
{\hbox{\begin{tikzpicture}[
						scale=0.6,
						vert/.style={inner sep=2pt, circle,  draw, thick},
						leaf/.style={inner sep=2pt,rectangle},
						edge/.style={-,black!30!black, thick},
						]
						\node[vert] (k) at (0.75,0) {\footnotesize$m$};
						\node[leaf] (lk) at (1.5,1) {\footnotesize$v_k$};
						\node[vert] (1) at (0,1) {\footnotesize$m$};
						\node[leaf] (l1) at (0.75,2) {\footnotesize$v_{k-1}$};
						\node[leaf] (2) at (-0.75,2) {\footnotesize$\cdots$};
						\node[leaf] (l2) at (0,3) {\footnotesize$\cdots$};
						\node[vert] (3) at (-1.5,3) {\footnotesize$m$};
						\node[leaf] (l3) at (-0.75,4) {\footnotesize$v_2$};
						\node[leaf] (l4) at (-2.25,4) {\footnotesize$v_1$};
						\draw[edge] (0.75,-1)--(k);
						\draw[edge] (k)--(1)--(2)--(3);
						\draw[edge] (k)--(lk);
						\draw[edge] (1)--(l1);
						\draw[edge] (2)--(l2);
						\draw[edge] (3)--(l3);
						\draw[edge] (3)--(l4);
			\end{tikzpicture}}}
			\] where $v_1$ is the minimal vertex of $\Gr$, $v_2$ is the minimal vertex adjacent to $\{v_1\}$, $v_3$ is the minimal vertex adjacent to tube $\{v_1,v_2\}$, and so on. Inductively, this monomial is defined by the rule $m_{\Gr}=m_{\Gr/e}\circ^{\Gr}_e m$, where $e$ is an edge consisting of minimal vertex and its minimal adjacent vertex with respect to vertex ordering. 
			
			\item Also, it was shown in~\cite[Thm 5.2.1]{lyskov2023contractads}, that the shuffle contractads $\Gerst$ and $\Ass$ (in presentation~\eqref{eq::jordanlie_pres}) admits a quadratic Gr\"obner basis with respect to "quantum" version of $\grpermlex$-order, such that the normal monomials are of the form 
			\[
			\vcenter
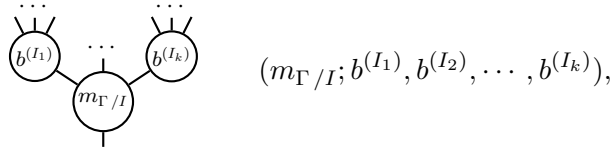
{\hbox{\begin{tikzpicture}[
						scale=0.6,
						vert/.style={inner sep=1pt, circle,  draw, thick},
						leaf/.style={inner sep=2pt,rectangle},
						edge/.style={-,black!30!black, thick},
						]
						\node[vert] (1) at (0,1) {\scriptsize$m_{\scalebox{0.7}{$\Gr/I$}}$};
						\node[vert] (l) at (-1.5,2) {\scriptsize$b^{\scalebox{0.7}{$(I_1)$}}$};
						\node[leaf] (ll) at (-2,3) {\space};
						\node[leaf] (lm) at (-1.5,3.1) {\footnotesize$\cdots$};
						\node[leaf] (lr) at (-1,3) {\space};
						\node[leaf] (mid) at (0,2.2) {\footnotesize$\cdots$};
						\node[vert] (r) at (1.5,2) {\scriptsize$b^{\scalebox{0.7}{$(I_k)$}}$};
						\node[leaf] (rl) at (1,3) {\space};
						\node[leaf] (rm) at (1.5,3.1) {\footnotesize$\cdots$};
						\node[leaf] (rr) at (2,3) {\space};
						\draw[edge] (0,0)--(1);
						\draw[edge] (1)--(l);
						\draw[edge] (l)--(lm);
						\draw[edge] (l)--(ll);
						\draw[edge] (l)--(lr);
						\draw[edge] (1)--(mid);
						\draw[edge] (1)--(r);
						\draw[edge] (r)--(rm);
						\draw[edge] (r)--(rl);
						\draw[edge] (r)--(rr);
			\end{tikzpicture}}} \quad \quad (m_{\Gr/I};b^{(I_1)},b^{(I_2)},\cdots,b^{(I_k)}),
			\]where $b^{(I_i)}$ are normal $\Lie$-monomials. In particular, this monomial basis  implies the decompostions $\Gerst\cong \Com\circ\Susp^{-1}\Lie,\quad \Ass\cong \Com\circ\Lie$.
		\end{itemize}
	\end{example}
	\subsection{Complete Multipartite Graphs and Colored operads}\label{sec::multipartite}
	Recall that a \textit{complete multipartite graph} is a graph whose vertex set admits a partition, called \textit{independent}, such that each two vertices from different blocks are adjacent and vertices in one block are not adjacent. Let us note that, by the definition, such a partition is unique. Moreover, a complete multipartite graph is determined up to isomorphism by the sizes of the corresponding partition blocks. So, we shall use the following notation for complete multipartite graphs.
	
	\begin{defi}\label{def:younggraph}
		For a partition $\lambda=(\lambda_1\geq \lambda_2\geq\cdots,\geq\lambda_k)$, we let $\K_{\lambda}$ be the complete multipartite graph with independent partition with blocks of sizes $\lambda_1,\lambda_2,\cdots,\lambda_k$.
	\end{defi}
	\begin{figure}[ht]
		\caption{List of complete multipartite graphs on 4 vertices.}
		\[
		\vcenter{\hbox{\begin{tikzpicture}[scale=0.6]
					\fill (0,0) circle (2pt);
					\fill (0,1.5) circle (2pt);
					\fill (1.5,0) circle (2pt);
					\fill (1.5,1.5) circle (2pt);
					\node at (0.75,-0.6) {$\K_{(4)}$};
		\end{tikzpicture}}}
		\quad
		\vcenter{\hbox{\begin{tikzpicture}[scale=0.6]
					\fill (0,0) circle (2pt);
					\fill (-1,1.5) circle (2pt);
					\fill (0,1.5) circle (2pt);
					\fill (1,1.5) circle (2pt);
					\draw (0,0)--(-1,1.5);
					\draw (0,0)--(0,1.5);
					\draw (0,0)--(1,1.5);
					\node at (0,-0.6) {$\K_{(3,1)}\cong \St_3$};
		\end{tikzpicture}}}
		\quad
		\vcenter{\hbox{\begin{tikzpicture}[scale=0.6]
					\fill (0,0) circle (2pt);
					\fill (0,1.5) circle (2pt);
					\fill (1.5,0) circle (2pt);
					\fill (1.5,1.5) circle (2pt);
					\draw (0,0)--(1.5,0)--(1.5,1.5)--(0,1.5)-- cycle;
					\node at (0.75,-0.6) {$\K_{(2^2)}\cong \Cyc_4$};
		\end{tikzpicture}}}
		\quad
		\vcenter{\hbox{\begin{tikzpicture}[scale=0.6]
					\fill (0,0) circle (2pt);
					\fill (0,1.5) circle (2pt);
					\fill (1.5,0) circle (2pt);
					\fill (1.5,1.5) circle (2pt);
					\draw (0,0)--(1.5,0)--(1.5,1.5)--(0,1.5)-- cycle;
					\draw (0,0)--(1.5,1.5);
					\node at (0.75,-0.6) {$\K_{(2,1^2)}$};
		\end{tikzpicture}}}
		\quad
		\vcenter{\hbox{\begin{tikzpicture}[scale=0.6]
					\fill (0,0) circle (2pt);
					\fill (0,1.5) circle (2pt);
					\fill (1.5,0) circle (2pt);
					\fill (1.5,1.5) circle (2pt);
					\draw (0,0)--(1.5,0)--(1.5,1.5)--(0,1.5)-- cycle;
					\draw (0,0)--(1.5,1.5);
					\draw (1.5,0)--(0,1.5);
					\node at (0.75,-0.6) {$\K_{(1^4)}\cong \K_4$};
		\end{tikzpicture}}}
		\]
	\end{figure}
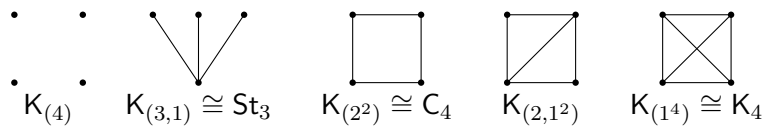
	Note that, for each partition $\lambda$ of length at least $2$, the graph $\K_{\lambda}$ is connected. Also, a subset of vertices $G\subset \K_{\lambda}$ is a tube if it is a singleton or contains a pair of vertices from different blocks of independent partition.
	\begin{lemma}
		\label{lem::multipartite_contraction}
		If $\Gamma$ is a complete multipartite graph and $G$ is a tube in $\Gamma$, then both $\Gamma|_{G}$ and $\Gamma/G$ are complete multipartite graphs.    
	\end{lemma}
	\begin{proof}
		If $\Gamma = \K_{(1^m,k_1,\ldots,k_i)}$ is the complete multipartite graph on $m+k_1+\ldots+k_i$ vertices that are divided into $m+i$ groups with first $m$ groups consisting from $1$ element. Let $G$ be the tube consisting of exactly $r$ vertices belonging to the subset of first $m$ vertices, $s_1$ vertices belonging to $m+1$-st group of $k_1$ vertices, $s_2$ from $k_2$ and so on.
		Then the induced subgraph $\K_{(1^m;k_1,\ldots,k_i)}|_{G}$ is the complete multipartite graph $\K_{(1^r;s_1,\ldots,s_i)}$, respectively the contracted graph $\K_{(1^m;k_1,\ldots,k_i)}/{G}$ coincides with the complete multipartite graph $\K_{(1^{m-r+1};k_1-s_1,\ldots,k_i-s_i)}$ where the tube $G$ corresponds to the new apex.
	\end{proof}
	The following lemma shows how complete multipartite graphs stand out from the rest graphs.
	\begin{lemma}
		\label{lemma:yungcone}
		A graph $\Gr$ is a complete multipartite if and only if, for each non-trivial tube $G$, the vertex $\{G\}$ in the contracted graph $\Gr/G$ is adjacent to all vertices.
	\end{lemma}
	\begin{proof}
		$(\Rightarrow)$: Follows from the proof of Lemma~\ref{lem::multipartite_contraction}
		
		$(\Leftarrow)$: Let $\Gr$ be a graph with the condition above. Define the relation on the vertex set by the rule: $v\sim w$ if $v=w$ or $v$ and $w$ are not adjacent. It suffices to check that this relation is an equivalence relation. Indeed, in this case, the graph $\Gr$ is the complete multipartite graph associated with the partition into equivalence classes. This relation is reflexive and symmetric by the definition. Let us examine transitivity. Suppose we have a triple of vertices such that $v\sim w$, $w\sim u$, but $v\not\sim u$. We get a contradiction since the contracted vertex $\{v,u\}$ in $\Gr/\{v,u\}$ is not adjacent to $w$.
	\end{proof}
	
	\subsubsection{From contractad on multipartite graphs to coloured operads}
	
	Colored operads are operads whose inputs and outputs are colored by a given set of colors, such that compositions should respect the colours.
	This is a very general object and, in particular,
	it is known that any Feynman category~\cite{ward2017feynman} admits a description as algebras over an appropriate colored operad. 
	One may look at contractades as coloured operads with an infinite collection of colours, indexed by graphs. This description does not give anything new in the general case and seems very complicated. 
	However, we suggest the description of contractads on multipartite graphs based on coloured operads and \emph{multi-symmetric} collections and show how one can use it for the computations of the Hilbert series.
	\begin{defi}
		\label{def::P_M_Collection}
		\begin{itemize}
			\item
			\emph{A pointed partition} of a set $S$ is a partition such that one of the subsets in a partition is marked (\emph{pointed}). (Note that we allow to point an empty subset.)
			\\
			By $\Pi_{*}$ we denote the corresponding groupoid, whose objects are pointed partitions of finite sets and morphisms are isomorphisms of sets that preserve partitions and the marked subsets.
			\item 
			\emph{A pointed multi-symmetric collection} in $\C$ is a contravariant functor  $\Pi_{*}\to \C$.
		\end{itemize}
	\end{defi}
	For example, in~\ref{eq::part::example} examples of two pointed partitions of the set on $5$ elements, where we underline the pointed subset 
	\begin{equation}
		\label{eq::part::example}
		\{\underline{\{1,2\}},\{3,5\},\{4\}\}, \ \{\underline{\emptyset},\{1,3,4\},\{2,5\}\} \ \in \Pi_{*}(5).
	\end{equation}
	In the second example, the cardinality of a pointed subset is $0$ and we underline the empty subset.
	
	A class of isomorphic pointed partitions is uniquely defined by the cardinalities of its subsets, therefore, a multi-symmetric collection $\eP$ is a collection of objects $\eP(m;\lambda)$ where $m\in\bZ_{\geq 0}$ is the cardinality of the pointed subset and $\lambda$ is the Young diagram of the cardinalities of the remaining subsets in a partition.
	
	With each pointed partition 
	$J_*:=\underline{J_0}\sqcup J_1\sqcup \ldots\sqcup J_m \in \Pi_{*}(T)$ and a given subset $S\subset T$ of cardinality greater than $1$ we associate an \emph{induced} pointed partition $S\cap J_{*}$ of $S$ given by the intersection:
	$$
	J_{*}\cap S:= \underline{{J_{0}}\cap S}\sqcup (J_1\cap S)\sqcup \ldots \sqcup (J_k\cap S).
	$$
	and the contracted pointed partition $J_{*}/S$ of $T\setminus S\sqcup \{*\}$:
	$$
	J_{*}/S:= \underline{\{*\}\sqcup J_0/(S\cap J_0)} \sqcup J_1/(S\cap J_1) \sqcup \ldots \sqcup J_k/(S\cap J_k)
	$$
	If the subset $S$ consists of one element we say that $J_{*}/S$ coincide with $S$.
	Note that if $S_1$ and $S_2$ are two subsets of $T$ with empty intersection ($S_1\cap S_2 = \emptyset$) then the consecutive contractions coincide, that is $(J_{*}/S_1)/S_2 = (J_{*}/S_2)/S_1$. 
	The iterated contractions can be summarised in the following operation.
	To each (\emph{non-pointed}) partition $S_{\bullet}:=S_1\sqcup\ldots\sqcup S_k$ of the set $T$ and a pointed partition $J_{*}:=\underline{J_0}\sqcup J_1\sqcup\ldots\sqcup J_l$ of the same set $T$ we assign the contracted pointed partition $J_*/S_{\bullet}:=\underline{(J/S)_{0}}\sqcup (J/S)_{1}\sqcup\ldots\sqcup (J/S)_{m}$ of the set with $k$ elements (whose elements are the indexing set of the partition $S_{\bullet}$) by the following rule: 
	$$
	\begin{array}{c}
		j\geq 1 \ (J/S)_{j}:= \{s \colon \ |S_s|=1\ \& \ S_s\subset J_{j}\}, \\
		\underline{(J/S)_{0}} := \{s \colon \ |S_s|=1\ \& \ I_s\subset \underline{J_{0}}\} \cup \{s \colon |S_s|>1 \}
	\end{array}
	$$
	In other words, 
	if the cardinality of the subset $S_s$ is greater than one then all elements in $S_s$ are collapsed to an extra element of the pointed subset $(J/S)_0$, and if $S_s$ consists of one element then it remains to belong to the same subset as it was in the pointed partition $J_*$.
	For example:
	$$
	\begin{array}{c}
		J_*:=\underline{\{1,3,7\}}\sqcup\{4,5,6\}\sqcup\{2,8,9,10\}, \\
		S_{\bullet}:=\{1,4,5\}\sqcup\{2,3,8\}\sqcup \{6\} \sqcup\{7\}\sqcup\{9\}\sqcup\{10\}, \\
		J_*/S_{\bullet}:= \underline{\left\{\{1,4,5\},\{2,3,8\},\{\stackrel{\phantom{1}}7\}\right\}}\sqcup 
		\left\{\{\stackrel{\phantom{1}}6\}\right\} \sqcup
		\left\{\{\stackrel{\phantom{1}}9\},\{10\}\right\}.
	\end{array}
	$$
	
	The contraction of partitions defines a monoidal structure in the category of pointed multi-symmetric collections, which we also denote by $\circ$ and call \emph{the contraction product}:
	\begin{equation}
		\label{eq::multisym:prod}
		\eP\circ\eQ( J_*) := \bigoplus_{I\in \Pi_{|J_*|}} \eP(J_*/I_{\bullet})\otimes \bigotimes_{I_s\in I_{\bullet}}\eQ(J_*\cap I_s)
	\end{equation}
	
	\begin{defi}
		\emph{A pointed multi-symmetric operad} is a monoid in the category of pointed multi-symmetric collections equipped with the contraction product $\circ$ defined in~\eqref{eq::multisym:prod}.
	\end{defi}
	
	\begin{prop}
		\label{prp::multipart}
		The forgetful functor $\Psi$ from the category of graphical collections to the category of multi-symmetric collections:
		\begin{multline*}
			\Psi(\eP)(m;\lambda) =\Psi(\eP)(\{\underline{\{1,\ldots,m\}},\{m+1,\ldots,m+\lambda_1\},\ldots,\{m+\lambda_1+\ldots+\lambda_{r-1}+1,\ldots,m+\lambda_1+\ldots+\lambda_r\}\}) 
			\\
			:= \eP(\K_{1^m;\lambda_1,\ldots,\lambda_r}).  
		\end{multline*}
		defined by complete multipartite graphs $\K_{1^m;\lambda_1,\ldots,\lambda_r}$ is compatible with compositions:
		$$
		\Psi(\eP\circ \eQ) = \Psi(\eP)\circ\Psi(\eQ).
		$$
		In particular, $\Psi$ maps a contractad to a multi-symmetric operad.
	\end{prop}
	\begin{proof}
		The statement follows from the direct inspection of the combinatorics of the contraction of multipartite graphs.
		From Corollary~\ref{lem::multipartite_contraction} we know that contracting a tube in a multipartite graph produces a new apex vertex.
	\end{proof}
	
	We defined the notion of a pointed multi-symmetric operad for the future use of generating series, however, the subsequent Remark~\ref{rem::multisym} should explain why this notion might be more familiar to the wide audience.
	\begin{remark}
		\label{rem::multisym}
		The category of pointed multi-symmetric operads in $\C$ is isomorphic to the subcategory of coloured operads whose colours
		are indexed by natural numbers 
		$\{\underline{0},\underline{1},\underline{2},\ldots\}$ and such that operations and compositions satisfy the following properties:
		\begin{itemize}
			\item
			The set of operations
			$\eP\binom{\underline{i}}{m;k_1,k_2\ldots}$   
			with $m$ inputs of colour $\underline{0}$, $k_1$ inputs of colour $\underline{1}$, $k_2$  inputs of colour $\underline{2}$ and $k_j$ inputs of colour $\underline{j}$ and outgoing color $\underline{i}$ is empty if either $i\neq 0$, or all except one $k_i$ with $i\geq 1$ equals zero\footnote{There is one exception: we keep the identity element in each colour  $\eP\left(\begin{smallmatrix} \underline{i} \\ 0;0,\ldots,1_{i},0,\ldots
				\end{smallmatrix}\right) = 1_{\C}$ 
				to be an operation in $\eP$.}:
			$$
			\forall i\geq 1 \ 
			\eP\left(\begin{smallmatrix}
				\underline{i} \\
				m;k_1,k_2,\ldots
			\end{smallmatrix}
			\right) = 0_{\C}.
			$$
			\item 
			The sets of operations are "multi-symmetric" for renumbering of the colours with positive indices:
			\begin{equation}
				\label{eq::sym::oper}
				\forall \sigma\in S_{\infty} \ 
				\eP\left(\begin{smallmatrix}
					\underline{0} \\
					m;k_1,k_2,\ldots
				\end{smallmatrix}\right)= 
				\eP\left(\begin{smallmatrix}
					\underline{0} \\
					m;k_{\sigma(1)},k_{\sigma(2)},\ldots
				\end{smallmatrix}\right) 
			\end{equation}
			\item 
			The operadic compositions are compatible with the symmetries~\eqref{eq::sym::oper} coming from the renumbering of colors indexed by positive numbers.
		\end{itemize}
		If, in addition, the multi-symmetric operad $\eP$ coincides with the image of a connected contractad $\Pop$, then we have vanishing the following operations:
		$$	\eP\left(\begin{smallmatrix}
			\underline{0} \\
			0;0,\ldots,0,k_i,0,\ldots
		\end{smallmatrix}
		\right) = \Pop(\K_{k_i}) =0
		$$
		because the graph $\K_{k_i}$ is a disconnected union of $k_i$ vertices.
	\end{remark}
	Remark~\ref{rem::multisym} explains, that the category of pointed multi-symmetric operads admits all standard constructions that we know for coloured operads. In particular, we do not need to explain the meaning of the following notions, which are well known for (coloured) operads:
	\begin{itemize}
		\item 
		A pointed multi-symmetric operad described in terms of generators and relations, in particular, a quadratic operad;
		\item 
		Bar construction of a pointed multi-symmetric operad and
		CoBar construction of a pointed multi-symmetric cooperad; 
		\item 
		Quadratic dual multi-symmetric operad and the Koszul duality for multi-symmetric operads.
	\end{itemize}
	\begin{corollary}
		The functor $\Psi$ from connected\footnote{Connectedness means that there are no unary operations except the identity} graphical collections to pointed multi-symmetric collections preserves operadic-type structures.
		In particular,
		\begin{itemize}
			\item $\Psi$ defines a functor from contractads to pointed multi-symmetric operads;
			\item $\Psi$ maps the free contractad generated by a graphical collection $\Upsilon$, such that $\Upsilon(\K_1)=0$ to the free pointed multisymmetric operad generated by multi-symmetric collection $\Psi(\Upsilon)$. 
			\item If the contractad $\eP$ was defined by generators $\Upsilon$ and relations $\mathcal{R}$ then the multi-symmetric operad is generated by $\Psi(\Upsilon)$ subject to relations $\Psi(\mathcal{R})$.
			\item If the quadratic contractad $\eP$ is Koszul, then the quadratic pointed multi-symmetric operad $\Psi(\eP)$ is also Koszul.
		\end{itemize}
	\end{corollary}
	\begin{proof}
		All statements follow from the observation that $\Psi$ is an exact functor that maps the composition of graphical collections to the composition of the corresponding pointed multi-symmetric collections.
	\end{proof}

	\section{Hilbert series}\label{sec::hilbertseries}
	In this section, we discuss the restrictions on Hilbert series of contractad. As is often the case with operadic-like structures, one can use Koszul duality to calculate the dimensions and Hilbert series of Koszul contractads.  Unfortunately, we typically can not say something concrete for all graphs, because they are too many and we do not know a good generating series for them. But we can give some concrete conclusions for different families of graphs.
	
	\subsection{Graphic functions}
	\label{sec::grafic::func}
	\begin{defi}
		\label{def::graphic::func}
		\emph{A graphic function} is a function on the groupoid of connected graphs $f\colon \mathsf{CGr}\to \mathsf{k}$ with values in a commutative ring $\mathsf{k}$ which is constant on isomorphism classes.
	\end{defi} 
	Roughly speaking, a \emph{graphic function} is a function that assigns to each graph a number (an element of a ring).
	The graphic functions (that we are mostly interested in) are functions $\chi(\Pop)$ given by dimensions of a graphical collection coming from the set of operations of a given contractad $\Pop$:
	$$
	\chi(\Pop)(\Gamma):= \dim \Pop(\Gamma).
	$$
	As always with questions of this type, the additional grading may help for computations. In other words, whenever $\Pop$ admits an additional grading compatible with the contractad structure, the corresponding graphic function
	\[
	\chi_q(\Orb)(\Gr)=\sum_i \dim \Orb_i(\Gr)q^i.
	\]
	will take values in the ring of polynomials (formal series in extra variable $q$).
	\begin{prop-def}
		The contraction product~\eqref{eq::contract::product} of graphical collections defines an associative (linear on the left argument) product of graphic functions:
		\[
		(\phi * \psi)(\Gr):= \sum_{I \vdash \Gr} \phi(\Gr/I)\prod_{G\in I}\psi(\Gr|_G), 
		\]
		such that for any two graphical collections $\Pop$ and $\Orb$
		\[
		\chi(\Pop \circ \Orb) = \chi(\Pop) * \chi(\Orb),
		\]
		and the graphic function 
		$$\varepsilon:=\left( \varepsilon(\Gr)\stackrel{\phantom{1}}=\delta_{\Gr, \Path_1} \right)$$ 
		is the unit for this product.
	\end{prop-def}
	\begin{proof}
		The associative product $*$ was introduced by Schmitt in~\cite{schmitt1994incidence}. The definition and all mentioned properties follow from the observation that $\circ$ defines a monoidal structure on graphical collections with collection $\mathbb{1}$ as a unit for this operation.    
	\end{proof}
	
	The product above of graphic function is well-behaved for Koszul dual contractads, meaning that up to some sign shifts, the graphic functions of Koszul dual contractads are inverse to each other. 
	One of the simplest ways to state the correct signs can be formulated using the weighted graphic function.
	Indeed, a quadratic contractad $\Pop=\Pop(\E,\R)$ admits a natural weight-grading with respect to the set of generators $\E$, where $\E$ is of weight 1. In this case, we denote by $\chi^{\mathrm{w}}_q(\Pop)$ the Hilbert series of $\Pop$ for the weight grading. For example, if the set of generators were binary ($\E= \E(\Path_2)$), then $\Pop(\Gamma)$ is concentrated in the weight space of degree $|V(\Gamma)|-1$.
	
	\begin{example}
		Consider the commutative contractad $\Com$. Since it is one-dimensional in each component and generated in the component $\Path_2$, we have 
		\[
		\chi(\Com)=1, \quad \chi_{q}(\Com)=q^{|V_{\Gr}|-1}.
		\] We shall use notations $\mathbb{1}_{q}:=\chi^{\mathrm{w}}_{q}(\Com)$ and $\mathbb{1}:=\mathbb{1}_{1}=\chi(\Com)$  for these graphic functions. Let us record some useful properties of this function that will be useful later:
		\begin{gather}
			\mathbb{1}\cdot f=f,
			\\
			\mathbb{1}_x\cdot\mathbb{1}_y=\mathbb{1}_{xy},
			\\
			\mathbb{1}_q\cdot(f*g)=(\mathbb{1}_q\cdot f)*(\mathbb{1}_q\cdot g),\label{eq::techformula1}
			\\
			f*(q\cdot g)=q\cdot((\mathbb{1}_q\cdot f)*g),\label{eq::techformula2}
		\end{gather} where $f\cdot g$ is a componentwise multiplication, $(f\cdot g)(\Gr):=f(\Gr)g(\Gr)$.
		The first two identities are obvious. The third identity is checked component-wise
		\begin{multline*}
			[\mathbb{1}_q\cdot(f*g)](\Gr)=q^{|V_{\Gr}|-1}\sum_{I\vdash \Gr}f(\Gr/I)\prod_{G\in I}g(\Gr|_G)
			=\\
			=\sum_{I\vdash \Gr}q^{|I|-1}f(\Gr/I)\prod_{G\in I}(q^{|G|-1}g(\Gr|_G))=[(\mathbb{1}_q\cdot f)*(\mathbb{1}_q\cdot g)](\Gr).
		\end{multline*} 
		The last one is checked in a similar way
		\[
		[f*(q\cdot g)](\Gr)=\sum_{I\vdash \Gr}f(\Gr/I)\prod_{G\in I}qg(\Gr|_G)=q\sum_{I\vdash \Gr}q^{|I|-1}f(\Gr/I)\prod_{G\in I}g(\Gr|_G)=[q\cdot((\mathbb{1}_q\cdot f)*g)](\Gr).
		\]
	\end{example}
	
	\begin{theorem}\label{thm::hilbertkoszul}
		Let $\Pop=\Pop(\E,\R)$ be a quadratic contractad that is weight-graded with respect to generators. If $\Pop$ is Koszul, we have
		\[
		\chi^{\mathrm{w}}_{-q}(\Pop^!)*\chi^{\mathrm{w}}_{q}(\Pop)=\chi(\mathbb{1})=\varepsilon.
		\]
		In particular, if $\Pop$ is generated by binary operations ($\E=\E(\Path_2)$), we have 
		\[
		\chi_{-1}(\Pop^{!})*\chi(\Pop) = \varepsilon,
		\] where $\chi_{-1}(\Pop^!)(\Gr)=(-1)^{|V_{\Gr}|-1}\dim\Pop^!(\Gr)$.
	\end{theorem}
	\begin{proof}
		The proof repeats the one known for operads, see e.g.~\cite{loday2012algebraic}.
		The contractad $\Pop$ is Koszul if the Koszul complex $\mathcal{K}(\Pop):=(\Pop^{\cokoszul}\circ \Pop,d_{\kappa})$ is acyclic. 
		
		From the definition of $\Pop^!$ out of $\Pop^{\cokoszul}$, it follows that these two weight graded graphical collections have the same Hilbert functions: $\chi_q(\Pop^!)=\chi_{q}(\Pop^{\cokoszul})$. Note that the Koszul complex $\mathcal{K}(\Pop)=\Pop^{\text{!`}}\circ^{\kappa}\Pop$ decomposes into the direct sum of complexes with respect to weight gradings
		\[
		\mathcal{K}(\Pop)=\bigoplus_{r\geq 0} \mathcal{K}(\Pop)^{(r)},
		\] The Euler characteristic of each component is given by the formula $\chi(\mathcal{K}(\Pop)^{(r)})=[q^r](\chi_{-q}(\Pop^{\cokoszul})*\chi_{q}(\Pop))$. Since $\Pop$ is Koszul, we have the quasi-isomorphism $\mathcal{K}(\Pop)\overset{\cong}{\rightarrow}\mathbb{1}$, so we obtain the desired identity
		\[
		\chi_{-q}(\Pop^!)*\chi_{q}(\Pop)=\sum_{r\geq 0}q^r\chi(\mathcal{K}(\Pop)^{(r)})=\chi(\mathbb{1})=\epsilon.
		\]
	\end{proof}
	\begin{example}
		Recall that the contractad $\Com$ is Koszul, with Koszul dual $\Lie$. So, by Theorem~\ref{thm::hilbertkoszul}, we have
		\[
		\chi^{\mathrm{w}}_{-q}(\Lie)*\chi^{\mathrm{w}}_{q}(\Com)=\epsilon.
		\] If we specialise $q=1$, we get the following recurrence
		\[
		\sum_{I\vdash \Gr}(-1)^{|I|-1}\dim \Lie(\Gr/I)=\begin{cases}
			1,\text{if }\Gr=\Path_1
			\\
			0,\text{ otherwise}.
		\end{cases}.
		\] It was shown in~\cite[Corrolary 3.2.1]{lyskov2023contractads}, that the recurrence above implies the formula $\dim \Lie(\Gr)=(-1)^{|V_{\Gr}|-1}\mu_{\parti(\Gr)}(\hat{0},\hat{1})$, where $\mu_{\parti(\Gr)}$ stand for the M\"obius function of the poset of graph-partitions (for refinement order). We shall use notation $\mu(\Gr)=\mu_{\parti(\Gr)}(\hat{0},\hat{1})$, so that $\mathbb{1}*\mu=\mu*\mathbb{1}=\epsilon$ and $\chi^{\mathrm{w}}_{-q}(\Lie)=\mathbb{1}_q\cdot\mu$.
	\end{example}
	\subsection{Hilbert series for special families of graphs}
	\label{sec::Hilb::ser}
	\subsubsection{One-parameter families}
	\begin{defi}
		For a graphic function $f$, we define series
		\begin{gather*}
			F_{\Path}(f)(t)=\sum_{n\geq 1} f(\Path_n)t^n
			\\
			F_{\Cyc}(f)(t)=\sum_{n\geq 1} f(\Cyc_n)\frac{t^n}{n}
			\\
			F_{\K}(f)(t)=\sum_{n\geq 1} f(\K_n)\frac{t^n}{n!}
			\\
			F_{\St}(f)(t)=\sum_{n\geq 0} f(\St_n)\frac{t^n}{n!}
		\end{gather*}
	\end{defi}
	We have
	\begin{prop}[Composition rules]\label{prop::one_parameter_comp}
		For a pair of graphic functions $f,g$, we have
		\begin{itemize}
			\item[(i)] $F_{\Path}(f*g)(t)=F_{\Path}(f)(F_{\Path}(g)(t))$
			\item[(ii)] $F_{\K}(f*g)(t)=F_{\K}(f)(F_{\K}(g)(t))$
			\item[(iii)] $F_{\Cyc}(f*g)(t)=F_{\Cyc}(f)(F_{\Path}(g)(t))-f(\Path_1)F_{\Path}(g)(t)+f(\Path_1)F_{\Cyc}(g)(t)$
		\end{itemize} Additionally, if g is connected, $g(\Path_1)=1$, we have
		\begin{itemize}
			\item[(iv)] $F_{\St}(f*g)(t)=F_{\St}(f)(t)\cdot F_{\St}(g)(t)$
		\end{itemize}
	\end{prop}
	\begin{proof}
		(i) Note that the family of paths $\{\Path_n\}_{n\geq 1}$ is closed under induced and contracted graphs. Moreover, the partition set of $\Path_n$ coincides with the set $\Pi_{\mathsf{int}}([n])$ of interval partitions of $[n]$. Hence the product $(f*g)(\Path_n)$ is given by the formula
		\[
		(f*g)(\Path_n)=\sum_{I\in\Pi_{\mathsf{int}}([n])}f(\Path_{|I|})\prod_{G\in I}g(\Path_{|G|})
		\] So, we conclude the desired identity
		\begin{multline*}
			F_{\Path}(f)(F_{\Path}(g)(t))=f(\Path_1)(g(\Path_1)t+g(\Path_2)t^2+\cdots)+f(\Path_2)(g(\Path_1)t+g(\Path_2)t^2+\cdots)^2+\cdots=
			\\=\sum_{n\geq 1} \left[\sum_{I\in \Pi_{\mathsf{int}}([n])} f(\Path_{|I|})\prod_{G\in I}g(\Path_{|G|})\right]t^n=\sum_{n\geq 1} \left[\sum_{I\vdash\Path_n} f(\Path_{|I|})\prod_{G\in I}g(\Path_{|G|})\right]t^n
			=\\
			=\sum_{n\geq 1} (f*g)(\Path_n)t^n=F_{\Path}(f*g)(t).
		\end{multline*}
		(ii) For complete graph $\K_n$, the corresponding partition set $\parti(\K_n)$ coincides with the usual partition set $\Pi([n])$, hence
		\[
		F_{\K}(f)(F_{\K}(g)(t))=f(\K_1)F_{\K}(g)+f(\K_2)\frac{F_{\K}(g)^2}{2!}+\cdots=\sum_{n\geq 1}(\sum_{I\in \Pi([n])} f(\K_{|I|})\prod_{G\in I}g(\K_{|G|}))\frac{t^n}{n!}=F_{\K}(f*g)(t).
		\]
		(iii)  For cycle graphs, the product $(f*g)(\Cyc_n)$ is given by the formula
		
		\[(f*g)(\Cyc_n)=f(\Path_1)g(\Cyc_n)+\sum_{I \vdash\Cyc_n, I\neq \{[n]\}} f(\Cyc_{|I|})\prod_{G\in I} g(\Path_{|G|})
		\]
		Let us rewrite the summation on the right hand side as the summation by partitions of $\Path_n$ as follows. The graph inclusion $\Path_n\subset \Cyc_n$ defines an inclusion of posets $\parti(\Path_n)\subset\parti(\Cyc_n)$. Also, the action of the cyclic group $\mathbb{Z}_n$ on $\Cyc_n$ produces the action on the partition set $\parti(\Cyc_n)$. In terms of this action, we have
		\[(f*g)(\Cyc_n)=f(\Path_1)g(\Cyc_n)+\sum_{J\vdash  \Path_n, J\neq \{[n]\}} \frac{|\mathsf{Orb}_{\mathbb{Z}_n}(J)|}{|\mathsf{Orb}_{\mathbb{Z}_n}(J)\cap\parti(\Path_n)|}f(\Cyc_{|J|})\prod_{G\in I} g(\Path_{|G|})
		\]
		Let us determine the fraction $\frac{|\mathsf{Orb}_{\mathbb{Z}_n}(J)|}{|\mathsf{Orb}_{\mathbb{Z}_n}(J)\cap\parti(\Path_n)|}$. For an edge $e\in E_{\Cyc_n}$, let $\mathsf{Orb}^{e}_{\mathbb{Z}_n}(J)$ be the subset of the orbit consisting of partitions such that the edge $e$ is out of partition blocks. In particular, we have $\mathsf{Orb}^{(1,n)}_{\mathbb{Z}_n}(J)=\mathsf{Orb}_{\mathbb{Z}_n}(J)\cap\parti(\Path_n)$. Since each partition $I$ of $\Cyc_n$ has exactly $|I|$ edges out of the partition, we have
		\[
		n|\mathsf{Orb}_{\mathbb{Z}_n}(J)\cap\parti(\Path_n)|=\sum_{e\in E_{\Cyc_n}} |\mathsf{Orb}^{e}_{\mathbb{Z}_n}(J)|=|J||\mathsf{Orb}_{\mathbb{Z}_n}(J)|.
		\] All in all, we get
		\[(f*g)(\Cyc_n)=f(\Path_1)g(\Cyc_n)+\sum_{J\vdash  \Path_n, J\neq \{[n]\}} \frac{n}{|J|}f(\Cyc_{|J|})\prod_{G\in J} g(\Path_{|G|})
		\]
		So, for a pair of graphic functions $f,g$, we have
		\begin{multline*}
			F_{\Cyc}(f*g)(t)=\sum_{n\geq 1} \left[f(\Path_1)g(\Cyc_n)-nf(\Path_1)g(\Path_n)+\sum_{J\vdash \Path_n}\frac{n}{|J|}f(\Cyc_{|J|})\prod_{G\in I} g(\Path_{|G|})\right]\frac{t^n}{n}=\\=f(\Path_1)F_{\Cyc}(g)(t)-f(\Path_1)F_{\Path}(g)(t)+\sum_{n\geq 1} \left[\sum_{J\vdash \Path_n}\frac{f(\Cyc_{|J|})}{|J|}\prod_{G\in I} g(\Path_{|G|})\right]t^n=\\=f(\Path_1)F_{\Cyc}(g)(t)-f(\Path_1)F_{\Path}(g)(t)+F_{\Cyc}(f)(F_{\Path}(g)(t)).\end{multline*}
		(iv) Tubes of the stellar graph $\St_n$ are singletons or vertex subsets containing the core $0$. So, a partition of the stellar graph is a partition of the vertex set that is made up of one non-trivial block $J\cup \{0\}$ containing core. Moreover, the resulting contracted graph $\St_{n}/_{J\cup \{0\}} \cong \St_{n-|J|}$ and induced subgraph $\St_{n}|_{J\cup\{0\}}\cong \St_{|J|}$ are also stellar graphs. Hence, for $g$ connected, the product $(f*g)(\St_n)$ is given by the formula
		\[
		(f*g)(\St_n)=\sum_{I\bigsqcup J=[n]}f(\St_{|I|})g(\St_{|J|})=\sum^n_{i=0} \binom{n}{i}f(\St_i)g(\St_{n-i})
		\] Hence
		\begin{multline*}
			F_{\St}(f)\cdot F_{\St}(g)=\left(\sum_{n\geq 0} f(\St_n)\frac{t^n}{n!}\right)\left(\sum_{m\geq 0} f(\St_m)\frac{t^m}{m!}\right)=
			\\ =
			\sum_{k\geq 0} \left[\sum^n_{i=0} \binom{n}{i}f(\St_i)g(\St_{n-i})\right]\frac{t^n}{n!}=F_{\St}(f*g).
		\end{multline*}
	\end{proof}
	\begin{example}
		We have
		\begin{gather}
			F_{\Path}(\mathbb{1})=\frac{t}{1-t} \Rightarrow F_{\Path}(\mu)=\frac{t}{1+t} \Rightarrow \mu(\Path_n)=(-1)^{n-1}
			\\
			F_{\Cyc}(\mathbb{1})=-\log(1-t)\Rightarrow F_{\Cyc}(\mu)=t+\frac{t}{1+t}-\log(1+t) \Rightarrow \mu(\Cyc_n)=(-1)^{n-1}(n-1)
			\\
			F_{\K}(\mathbb{1})=e^t-1 \Rightarrow F_{\K}(\mu)=\log(1+t) \Rightarrow \mu(\K_n)=(-1)^{n-1}(n-1)!
			\\
			F_{\St}(\mathbb{1})=e^t \Rightarrow F_{\K}(\mu)=e^{-t} \Rightarrow \mu(\St_n)=(-1)^{n}.
		\end{gather}
	\end{example}
	\subsubsection{Complete Multipartite graphs and Symmetric functions}
	Consider the ring of symmetric functions $\Lambda_{\mathbb{Q}}=\underset{\to}{\lim}\mathbb{Q}[x_1,\cdots,x_n]^{\Sigma_n}$ with rational coefficients. 
	This ring is graded and admits many famous bases indexed by Young diagrams (or partitions). 
	The first basis we want to use is called monomial symmetric functions $m_{\lambda}=\mathsf{Sym}(x^{\lambda})$ given by symmetrizations of monomials. For a Young diagram $\lambda=(\lambda_1\geq\lambda_2\geq\ldots)$, we denote by $|\lambda|$ its weight $\sum \lambda_i$, by $l(\lambda)$ its length (number of nonzero $\lambda_i$), and by $\lambda!=\lambda_1!\cdot\lambda_2!\cdots\cdot\lambda_k!$ its factorial.
	\begin{defi}
		Let $\eP$ be a pointed multi-symmetric collection (see Definition~\ref{def::P_M_Collection}), then its \emph{Young generating series} is given by the following symmetric function:
		\[
		F_{\mathsf{Y}}(\eP)(z;x_1,x_2,\ldots):= 
		\sum_{n\geq 0,\lambda} \dim\eP(n;\lambda) \frac{z^{n}}{n!}\frac{m_\lambda}{\lambda!} \in \Lambda_{\bQ}[[z]].
		\]    
	\end{defi}
	Note that the linear term of $F_{\mathsf{Y}}(\eP)$ is equal to $z$ if a pointed multi-symmetric collection is connected ($\eP(1;0)=\mathbb{1}$). 
	\begin{theorem}
		\label{thm::Ser::MS}
		The Young generating series of the composition of connected pointed multi-symmetric partitions $\eP$ and $\eQ$ coincides with the composition 
		with respect to the $z$ variable of the Young generating series of $\eP$ and $\eQ$ correspondingly:
		\[ 
		F_{\mathsf{Y}}(\eP\circ\eQ)(z;x_1,x_2,\ldots) = 
		F_{\mathsf{Y}}(\eP)(F_{\mathsf{Y}}(\eQ)(z;x_1,x_2,\ldots)
		;x_1,x_2,\ldots).
		\]
	\end{theorem}
	\begin{proof}
		Let us examine more closely the typical generating series of pointed multi-symmetric operads. Recall that, for a coloured collection $\eP$ with a colour set $S$, the ordinary generating series is given by a collection of series $\{F^{\underline{s}}(\eP)(x_t) | \underline{s},\underline{t}\in S\}$ with 
		\[
		F^{\underline{s}}(x_1,x_2,\ldots):= \sum \dim\eP\left(\begin{smallmatrix}
			\underline{s} \\
			m_1,m_2,\ldots
		\end{smallmatrix}\right) \frac{x_1^{m_1}}{m_1!}\frac{x_2^{m_2}}{m_2!}\ldots
		\]
		which defines a Taylor expansion of a formal endomorphism $F(\eP)=(F^{\underline{1}},\ldots,F^{\underline{s}},\ldots)$ of the space $\mathbb{C}^{|S|}$ whose $s$'th coordinate coincides with $F^{\underline{s}}$. 
		It is well known (see e.g.~\cite{kharitonov2022grobner}) that the generating series of the composition of coloured collections is equal to the composition of the aforementioned endomorphism:
		\[
		F(\eP\circ\eQ) = F(\eP)\circ F(\eQ).
		\]
		We mentioned in Remark~\ref{rem::multisym} that pointed multi-symmetric operads coincide with the subcategory of a category of coloured operads on the infinite number of colours.
		The symmetry for inputs of colours with positive indices implies that the corresponding generating series are symmetric. To emphasize the symmetry, we denote the variable that corresponds to the colour $\underline{0}$ by $z$ and the remaining variables in the generating series by $x_i$
		\begin{multline*}
			F^{\underline{0}}(\eP)(z;x_1,x_2,\ldots) :=  \sum_{n,k_1,k_2,\ldots\geq 0} \dim\eP\left(\begin{smallmatrix}
				\underline{0} \\
				n; k_1,k_2,\ldots
			\end{smallmatrix}\right) \frac{z^n}{n!}\frac{x_1^{k_1}}{k_1!}\frac{x_2^{k_2}}{k_2!}\ldots = 
			\\
			=
			\sum_{n\geq 0,\lambda} \dim\eP(n;\lambda)\frac{z^n}{n!} \mathsf{Sym}\left(\frac{x_1^{\lambda_1}}{\lambda_1!}\frac{x_2^{\lambda_2}}{\lambda_2!}\ldots\right) =
			\sum_{n,\lambda} \dim\eP(n;\lambda)\frac{z^n}{n!} \frac{m_\lambda}{\lambda_!} =
			\\
			=
			z + \sum_{n,\lambda \colon \ n+|\lambda|\geq 2} \dim\eP(n;\lambda)\frac{z^n}{n!} \frac{m_\lambda}{\lambda!}.
		\end{multline*}
		Since we expect our multi-symmetric collection $\eP$ to be connected, meaning that $\eP$ of a one-element set is identity, the $\underline{i}$'th component of the generating series is trivial:
		$$\forall i>0 \ F^{\underline{i}}(\eP)= x_i.$$
		Therefore, the only nontrivial coordinate where one has to take the composition is the $z$-variable that corresponds to the pointed colour $\underline{0}$. 
	\end{proof}
	
	\begin{defi}
		For a graphic function $f$ with values in $\mathsf{k}$, we define its Young generating function $F_{\mathsf{Y}}(f)\in \mathsf{k}\otimes \Lambda_{\mathbb{Q}}$, by the rule
		\begin{gather*}
			F_{\mathsf{Y}}^{(0)}(f)=\sum_{l(\lambda)\geq 2} f(\K_{\lambda})\frac{m_{\lambda}}{\lambda!}; \quad
			F_{\mathsf{Y}}^{(n)}(f)=\sum_{|\lambda|\geq 0} f(\K_{(1^n)\cup\lambda})\frac{m_{\lambda}}{\lambda!}\text{, for }n\geq 1
			\\
			F_{\mathsf{Y}}(f)(z)=\sum_{n\geq 0} F_{\mathsf{Y}}^{(n)}(f)\frac{z^n}{n!}=\sum_{l(\lambda)\geq 2} f(\K_{\lambda})\frac{m_{\lambda}}{\lambda!}+\sum_{n\geq 1,|\lambda|\geq 0} f(\K_{(1^n)\cup\lambda})\frac{m_{\lambda}}{\lambda!}\frac{z^n}{n!}.
		\end{gather*}
	\end{defi}  
	
	Recall that with respect to the standard scalar product on the space of symmetric functions (for which Schur basis is orthonormal) the families of monomial and complete symmetric functions $h_\lambda:=h_{\lambda_1}h_{\lambda_2}\ldots$ are orthogonal to each other $\langle h_{\lambda},m_{\mu}\rangle=\delta_{\lambda,\mu}$. So, for each partition $\lambda$, we have
	\[
	\lambda!\langle h_{\lambda},F_{\mathsf{Y}}(f)\rangle=\sum_{n\geq 0} f(\K_{(1^n)\cup \lambda})\frac{z^n}{n!}. 
	\] In particular, for $l(\lambda)\geq 2$, we have  $f(\K_{\lambda})=\lambda!\langle h_{\lambda},F^{(0)}_{\mathsf{Y}}(f)\rangle$.

	\begin{theorem}
		\label{thm::ser::multipartie}
		For connected graphic functions $f,g$, the Young symmetric function of their composition is the composition of these functions with respect to the variable $z$:
		\begin{equation}
			\label{eq::Young::composition}
			F_{\mathsf{Y}}(f*g)(z)=F_{\mathsf{Y}}(f)(F_{\mathsf{Y}}(g)(z))
		\end{equation}    
	\end{theorem}
	\begin{proof}
		Note that for a connected graphical collection $\Pop$ the Young symmetric function $F_{\mathsf{Y}}(\dim\Pop)$ of the graphic function of dimensions of $\Pop$ coincides with the Young symmetric function of the corresponding pointed multi-symmetric collection $\Psi(\Pop)$. 
		In particular, the absence of the terms $f( \mathsf{K}_{(m)} )\mathsf{Sym}\frac{x^{m}}{m!}$ is necessary because the corresponding complete multipartite graph $\mathsf{K}_{(m)}$ is not connected.
		
		Since the functor $\Psi$ preserves compositions, we end up with the desired identity of generating series:
		\begin{multline*}
			F_{\mathsf{Y}}(\dim\Pop*\dim\cQ) = F_{\mathsf{Y}}(\dim \Pop\circ\cQ) = 
			F_{\mathsf{Y}}(\Psi(\Pop\circ\cQ)) \stackrel{\text{Proposition\ref{prp::multipart}}}{=} \\
			= F_{\mathsf{Y}}(\Psi(\Pop)\circ\Psi(\cQ)) \stackrel{\text{Theorem~\ref{thm::Ser::MS}}}{=} F_{\mathsf{Y}}(\Psi(\Pop))\circ F_{\mathsf{Y}}(\Psi(\cQ)).
		\end{multline*}
		This implies that Equation~\eqref{eq::Young::composition} holds for graphic functions with non-negative integer values. Using linearity with respect to the left argument and appropriate homogeneity with respect to the second argument Equation~\eqref{eq::Young::composition} makes sense for arbitrary connected graphic functions.
	\end{proof}
	
	The proof of Theorem~\ref{thm::ser::multipartie} is the only place where we need the notion of pointed multi-symmetric operads defined in \S\ref{sec::multipartite}. We believe that this theorem might have a direct combinatorial proof, however, we think that the point of view with coloured operads should help the reader to understand the mystery of the Equation~\eqref{eq::Young::composition}.

	\begin{example}\label{ex::completemultipartite_for_com_lie}
		Let us compute the Young generating function for $\mathbb{1}=\chi(\Com)$. Let $p_n=\sum_{i\geq 1} x_i^n$ be the power symmetric function. The zero term $F_{\mathsf{Y}}^{(0)}(\mathbb{1})$ has the form
		\begin{gather*}
			F_{\mathsf{Y}}^{(0)}(\mathbb{1})=\sum_{l(\lambda)\geq 2} \frac{m_{\lambda}}{\lambda!}=\sum_{|\lambda|\geq 0} \frac{m_{\lambda}}{\lambda!}-\left(1+\sum_{n\geq 1}\frac{p_n}{n!}\right)=\sum_{n\geq 0} \frac{p_1^n}{n!}-\left(1+\sum_{n\geq 1}\frac{p_n}{n!}\right)=e^{p_1}-\left(1+\sum_{n\geq 1}\frac{p_n}{n!}\right),
		\end{gather*}
		where the second identity follows from the formula $\frac{p_1^n}{n!}=\sum_{\lambda\vdash n}\frac{m_{\lambda}}{\lambda!}$. Similarly, for higher terms, we have $F_{\mathsf{Y}}^{(n)}(\mathbb{1})=e^{p_1}$. All in all, we get
		\begin{equation}
			F_{\mathsf{Y}}(\mathbb{1})=e^{z+p_1}-\left(1+\sum_{n\geq 1}\frac{p_n}{n!}\right).
		\end{equation}
		For the graded version $\mathbb{1}_q=\chi^{\mathrm{w}}_q(\Com)$, for similar reasons, we have
		\begin{equation}\label{eq::com_multipartite}
			F_{\mathsf{Y}}(\mathbb{1}_q)=\frac{1}{q}\left[e^{q(z+p_1)}-1-\sum_{n\geq 1}\frac{p_nq^n}{n!}\right].
		\end{equation}
		For the inverse graphic function $\mu=\chi^{\mathrm{w}}_{-1}(\Lie)$, thanks to Theorem~\ref{thm::ser::multipartie}, we have
		\begin{equation}\label{eq::lie_multipartite}
			F_{\mathsf{Y}}(\mathbb{1})(F_{\mathsf{Y}}(\mu))=F_{\mathsf{Y}}(\mathbb{1}*\mu)=F_{\mathsf{Y}}(\epsilon)=z \quad \Rightarrow\quad F_{\mathsf{Y}}(\mu)=\log\left(1+z+\sum_{n\geq 1} \frac{p_n}{n!}\right)-p_1.  
		\end{equation}
	\end{example}
	\subsection{Chromatic polynomials}
	\label{sec::chromatic}
	Consider the Gerstenhaber contractad $\Gerst$ (see \S\ref{sec:disks}). Thanks to Theorem~\ref{thm:gerstpres}, the contractad $\Gerst$ is isomorphic to the product $\Com\circ\Susp^{-1}\Lie$ on the level of graphical collections. Hence its graded graphic function (with respect to homological grading) is given by the formula
	\begin{equation}\label{eq::hilbertgerst}
		\chi_{-q}(\Gerst)=\chi(\Com)*\chi^{\mathrm{w}}_{-q}(\Lie):=\mathbb{1}*(\mathbb{1}_q\cdot\mu).
	\end{equation}
	It was shown in~\cite[Prop.~5.2.1]{lyskov2023contractads}, that the Hilbert series of the components of the Gerstenhaber contractad are equal to re-normalised chromatic polynomials
	\begin{equation}\label{eq::chrom_gerst}
		\chi_{-q}(\Gerst)(\Gr)=q^{|V_{\Gr}|}\chi_{\Gr}(\frac{1}{q}),   
	\end{equation}
	where $\chi_{\Gr}(q)$ is the chromatic polynomial of a graph $\Gr$.
	
	\begin{defi}\label{def::chromgrfun}
		The Chromatic graphic function, denoted $\mathfrak{X}(q)$, is
		\[
		\mathfrak{X}(q)(\Gr):=\chi_{\Gr}(q).
		\]
	\end{defi} Let us express the Chromatic graphic function in terms of simple graphic functions.
	\begin{lemma}\label{lemma::formula_for_chrom}
		We have
		\[
		\mathfrak{X}(q)=q\cdot(\mathbb{1}_q*\mu).
		\]
	\end{lemma}
	\begin{proof}
		From equation~\eqref{eq::chrom_gerst}, we have
		\[
		\mathfrak{X}(q)(\Gr)=q^{|V_{\Gr}|}\chi_{-q^{-1}}(\Gerst)(\Gr)\Rightarrow \mathfrak{X}(q)=q\cdot\mathbb{1}_q\cdot\chi_{-q^{-1}}(\Gerst).
		\] By equation~\eqref{eq::hilbertgerst}, the right hand side is
		\[
		q\mathbb{1}_q\cdot(\mathbb{1}*(\mathbb{1}_{q^{-1}}\cdot\mu))\overset{\text{Eq}~\eqref{eq::techformula1}}{=}q(\mathbb{1}_q*(\mathbb{1}_q\cdot\mathbb{1}_{q^{-1}}\cdot\mu))=q\cdot(\mathbb{1}_q*\mu).
		\]
	\end{proof}
	As a first application, we reprove the following identity for chromatic polynomials, proved by Foissy in~\cite{foissy2016chromatic}.
	\begin{sled}[Multiplicative identity for chromatic polynomials]
		For graph $\Gr$, we have
		\begin{equation}
			\chi_{\Gr}(xy)=\sum_{I\vdash \Gr}\chi_{\Gr/I}(x)\prod_{G\in I}\chi_{\Gr|_G}(y).
		\end{equation} In terms of graphic functions, we have
		\begin{equation}
			\mathfrak{X}(xy)=\mathfrak{X}(x)*\mathfrak{X}(y).   
		\end{equation}
	\end{sled}
	\begin{proof}
		Thanks to Lemma~\ref{lemma::formula_for_chrom}, we have
		\begin{multline*}
			\mathfrak{X}(x)*\mathfrak{X}(y)=(x\cdot(\mathbb{1}_x*\mu))*(y\cdot(\mathbb{1}_y*\mu))\overset{\text{Eq}~\eqref{eq::techformula2}}{=}xy\cdot ((\mathbb{1}_{y}\cdot(\mathbb{1}_x*\mu))*(\mathbb{1}_y*\mu))=
			\\
			\overset{\text{Eq}~\eqref{eq::techformula1}}{=}xy\cdot ((\mathbb{1}_{y}\cdot(\mathbb{1}_x*\mu*\mathbb{1}))*\mu))\overset{\mathbb{1}*\mu=\epsilon}{=}xy\cdot ((\mathbb{1}_{y}\cdot\mathbb{1}_x)*\mu)=xy\cdot (\mathbb{1}_{xy}*\mu)=\mathfrak{X}(xy).
		\end{multline*}
	\end{proof}
	As another application, we compute the generating function of chromatic polynomials for complete multipartite graphs.
	\begin{sled}\label{cor::yungchrom}
		We have
		\begin{equation}\label{eq::chrommultipartitefull}
			F_{\mathsf{Y}}(\mathfrak{X}(q))(z)=\sum_{l(\lambda)\geq 2} \chi_{\K_{\lambda}}(q) \frac{m_{\lambda}}{\lambda!}+\sum_{n\geq 1, |\lambda|\geq 0} \chi_{\K_{(1^n)\cup \lambda}}(q) \frac{m_{\lambda}}{\lambda!}\cdot \frac{z^n}{n!}=\left(1+z+\sum_{n\geq 1}\frac{p_n}{n!}\right)^q-1-\sum_{n\geq 1}\frac{p_nq^n}{n!}.     
		\end{equation}
		In particular, we have\footnote{We use convention $\chi_{\K_{(0)}}(q)=1$ for the empty graph $\K_{(0)}=\varnothing$}
		\begin{equation}\label{eq::chromaticseries}
			\sum_{\lambda} \chi_{\K_{\lambda}}(q) \frac{m_{\lambda}}{\lambda!}=\left(1+\sum_{n\geq 1}\frac{p_n}{n!}\right)^q.   
		\end{equation}
	\end{sled}
	\begin{proof}
		Thanks to Theorem~\ref{thm::ser::multipartie} and equations~\eqref{eq::com_multipartite},\eqref{eq::lie_multipartite}, we have
		\begin{multline*}
			F_{\mathsf{Y}}(\mathfrak{X})(z)=F_{\mathsf{Y}}(q\cdot(\mathbb{1}_q*\mu))(z)=qF_{\mathsf{Y}}(\mathbb{1}_q)(F_{\mathsf{Y}}(\mu)(z))=\\ 
			=\exp\left(qp_1+q\left[\log\left(1+z+\sum_{n\geq 1}\frac{p_n}{n!}\right)-p_1\right]\right)-1-\sum_{n\geq 1}\frac{p_nq^n}{n!}=\left(1+z+\sum_{n\geq 1}\frac{p_n}{n!}\right)^q-1-\sum_{n\geq 1}\frac{p_nq^n}{n!}. 
		\end{multline*}
		Formula~\eqref{eq::chromaticseries} follows from formula~\eqref{eq::chrommultipartitefull} by substitution $z=0$ and the fact that, for disconnected complete mutipartite graphs $\K_{(n)}$, we have $\chi_{\K_{(n)}}(q)=q^n$.
	\end{proof}
	The series expansion of the right hand side of~\eqref{eq::chromaticseries} in the $p$-basis is given by
	\[
	\left(1+\sum_{n\geq 1}\frac{p_n}{n!}\right)^q=\sum_{\mu}\binom{q}{l(\mu)}\binom{l(\mu)}{m(\lambda)}\frac{p_{\mu}}{\mu!},
	\] where $\binom{l(\mu)}{m(\mu)}=\frac{l(\mu)!}{m_1(\mu)!\cdots m_n(\mu)!}$ and $m_i(\mu)=|\{j|\lambda_j=i\}|$-  cycle types. Recall that the transition matrix from the $p$-basis to the $m$-basis is lower triangular
	\[
	p_{\mu}=\sum_{\lambda \geq \mu} L_{\mu\lambda}m_{\lambda},
	\] where the sum ranges over all partitions $\lambda$ of the size $|\mu|$ that \textit{dominate} $\mu$: $\lambda_1\geq \mu_1, \lambda_1+\lambda_2\geq \mu_1+\mu_2$, and so on. Combinatorially, the coefficient $L_{\mu \lambda}$ is the number of $l(\lambda)\times l(\mu)$-matrices with entries in $\mathbb{N}$ whose column sum is $\mu$ and row sums is $\lambda$ and there is exactly one non-zero entry for each column.
	\begin{sled}
		For a partition $\lambda$, the chromatic polynomial of $\K_{\lambda}$ is given by the formula
		\begin{equation}
			\chi_{\K_{\lambda}}(q)=\lambda!\sum_{\mu\leq \lambda}\binom{q}{l(\mu)}\binom{l(\mu)}{m(\mu)}\frac{L_{\mu\lambda}}{\mu!}.   
		\end{equation}
	\end{sled}
	
	\begin{example}[Chromatic symmetric functions and trees]
		In~\cite{stanley1995symmetric}, Stanley introduced \textit{chromatic symmetric function}, a generalisation of chromatic polynomials, given by 
		\[
		\mathsf{X}(\Gr)=\sum_{\kappa\colon V_{\Gr}\to \mathbb{N}} \prod_{v\in V_{\Gr}}x_{\kappa(v)},
		\] where the sum ranges over all proper vertex colorings. The chromatic symmetric function can be expanded in the power-sum symmetric functions via the following formula
		\[
		\mathsf{X}(\Gr)=\sum_{S\subset E_{\Gr}} (-1)^{|S|}p_{\lambda(S)},
		\] where $\lambda(S)$ is the partition whose parts are the vertex sizes of the connected components of the edge-induced subgraph of $\Gr$ specified by $S$. When $\Gr=T$ is a tree, the correspondence $S\mapsto \lambda(S)$ defines a bijection between subsets of $E_{T}$ and graph partitions of $T$.
		So, in the class of trees, the symmetric graphic function $\mathsf{X}$ is expressed as a product of two simple graphic functions 
		\[
		\mathsf{X}(T)=(\mathbb{1}*(\mathbb{1}_{-1}\cdot p))(T),\quad T\text{ is a tree.}
		\] where $p(\Gr)=p_{|V_{\Gr}|}\in \Lambda$. In particular, for paths, we have
		\[
		\sum_{n\geq 1} \mathsf{X}(\Path_n)t^n=F_{\Path}(\mathsf{X})=F_{\Path}(\mathbb{1})(F_{\Path}(\mathbb{1}_{-1}\cdot p))=\frac{\sum_{n\geq 1} (-1)^{n-1}p_nt^n}{1-\sum_{n\geq 1} (-1)^{n-1}p_nt^n}.
		\] Hence we have
		\begin{equation}
			\mathsf{X}(\Path_n)=\sum_{\lambda\vdash n} (-1)^{n-l(\lambda)}\binom{l(\lambda)}{m(\lambda)}p_{\lambda}. 
		\end{equation}   The class of trees is significant due to Stanley’s conjecture that a symmetric chromatic function is a full invariant for trees. We expect that the methods of generating functions described in this Section could help in resolving this problem.
	\end{example}

	\section{Wonderful Contractad \texorpdfstring{$\bM$}{M} 
		and Modular Compactifications \texorpdfstring{$\beM_{0,\K_\lambda}$}{MK}}
	\label{sec::wonderful_modular}
	
	In this section, we first recall the definition of the wonderful contractad $\bM$ which is given by wonderful compactifications of graphical arrangments (\S\ref{sec::Wonderful}). Second, for complete multipartite graphs $\K_{\lambda}$, we define and describe moduli spaces $\beM_{0,\K_{\lambda}}$ as an explicit quotient of the Deligne-Mumford compactification $\beM_{0,n+1}$, where $\lambda\vdash n$.
	Third, we recall the combinatorial blowdowns of $\beM_{0,n+1}$ introduced by Smyth in~\cite{smyth2009towards} under the name \emph{modular compactification}. Finally, we show that, for multipartite graphs, all these smooth algebraic varieties are isomorphic. It is worth mentioning that the wonderful contractad is well defined over any field $\Bbbk$.
	However, later on, we are focused only on the cases  $\Bbbk=\mathbb{C}$ and $\Bbbk=\mathbb{R}$ because, in the subsequent sections, we extract some data on the homology of the corresponding smooth manifolds.
	
	\subsection{Wonderful contractad}
	\label{sec::Wonderful}
	In this subsection, we recall the construction of the Wonderful contractad, for details see~\cite[Sec 2.6]{lyskov2023contractads}.

	For a connected graph $\Gr$ with at least two vertices, consider the \textit{graphic arrangement} $\B(\Gr)=\{H_{(v,w)};=\{x_v=x_w\}| (v,w)\in E_{\Gr}\}$ in $\mathbb{k}^{V_{\Gr}}$. Note that the complete intersection of hyperplanes $\cap_{e\in E_{\Gr}} H_e=\langle \sum_{v\in V_{\Gr}} x_v\rangle$ is a one-dimensional space, so we replace the vector space $\mathbb{k}^{V_{\Gr}}$ with its quotient $W_{\Gr}=\mathbb{k}^{V_{\Gr}}/\langle \sum_{v\in V_{\Gr}} x_v\rangle$ to make the arrangement $\B(\Gr)$ essential. Moreover, for a non-trivial tube $G\subset V_{\Gr}$, we consider the subspace $H_G:=\cap_{e \in E_{\Gr|_G}} H_{e}$. 
	
	We define a map on $\M(\Gr):=\Pro(W_{\Gr})\setminus \cup_{e\in E_{\Gr}} \Pro(H_e)$, the projective complement of the arrangement,
	\[
	\rho\colon \M(\Gr)\rightarrow \prod_{G}\Pro(W_{\Gr}/H_G),
	\] where $\rho=\prod_G \rho_G$ is the product of projections $\rho_G\colon \M(\Gr)\to \Pro(W_{\Gr}/H_G)$. The map $\rho$ defines an embedding of $\M(\Gr)$ in the right hand side space.
	\begin{defi} The graphic compactification, denoted $\bM(\Gr)$, is the closure of the image of the embedding $\rho$. For the graph $\Path_1$, we let $\bM(\Path_1)=\mathrm{pt}$.  
	\end{defi} 
	
	The space $\bM(\Gr)$ is a smooth projective variety containing $\M(\Gr)$ as an open set. For a tube $G\subset V_{\Gr}$, there is the embedding of wonderful compactifications
	\begin{equation}\label{eq::rainsmaps}
		\iota\colon \bM(\Gr/G)\times \bM(\Gr|_{G})\to \bM(\Gr),  
	\end{equation} defined as follows. Consider the linear map $j\colon \mathbb{k}^{V_{\Gr/G}}\rightarrow \mathbb{k}^{V_{\Gr}}$ defined by the rule: $i(x_v)=x_v$ for $v\not\in G$, and $i(x_{\{G\}})=\sum_{v\in G}x_v$. For each tube $K$ with $K\not\subset G$, this map defines an inclusion of the quotient spaces $j^{K}_G\colon W_{\Gr/G}/H_{K/G}\hookrightarrow W_{\Gr}/H_{K}$. Next, consider the coordinate inclusion $\mathbb{k}^{G}\hookrightarrow \mathbb{k}^{V_{\Gr}}$. For each tube $K\subset G$, this map induces an isomorphism $i^K_G\colon W_{\Gr|_G}/H_{K}\overset{\cong}{\rightarrow}W_{\Gr}/H_{K}$. So, we define the inclusion $\iota$ by determining $\rho_K\circ \iota$ for each tube $K$
	\[
	\rho_K\circ \iota=\begin{cases}
		\Pro(i^K_G)\circ\rho_{K}\circ \pi_{\Gr|_G}\text{, for }K\subset G,
		\\
		\Pro(j^K_G)\circ\rho_{K/G}\circ \pi_{\Gr/G}\text{, for }K\not\subset G,
	\end{cases}
	\] where $\pi_{\Gr|_G}$ and $\pi_{\Gr/G}$ are projections on the corresponding factors. The image of this map, denoted $D_G$ forms an irreducible divisor. The complement $\partial\bM(\Gr):=\bM(\Gr)\setminus\M(\Gr)$ is a divisor with normal crossings with irreducible components $D_G$.
	\begin{remark}\label{rem::wondcomp}
		The construction of graphic compactifications $\bM(\Gr)$ is a special case of the wonderful compactifications, introduced by De Concini and Processi in~\cite{de1995wonderful}. In their language, the graphic compactification $\bM(\Gr)$ is the wonderful compactification associated with so-called a graphic building set $\G(\Gr)=\{H_G\}$ consisting of subspaces $H_G$ labeled by tubes of a graph. So, all algebraic-geometric properties of $\bM(\Gr)$ follow from the general theory of wonderful compactifications. The inclusions~\eqref{eq::rainsmaps} were first described in the operadic language by Rains in~\cite{rains2010homology} and formalized by Coron in~\cite{coron2022matroids}.
	\end{remark}
	
	\begin{defi}The Wonderful contractad is the graphical collection $\bM$ of graphic compactifications, with infinitesimal compositions~\eqref{eq::rainsmaps}.
	\end{defi}
	
	Let us consider some concrete examples of  graphic compactifications explaining the motivation of the Wonderful contractad.
	\begin{example}~\label{ex::comp_wonderful} For particular types of graphs, we have
		\begin{itemize}
			\item For complete graph $\K_n$, the resulting graphic compactification $\bM(\K_n)$ coincides with the moduli space $\beM_{0,n+1}$ of $(n+1)$-marked stable curves (with one marked point singled out $\infty$ as a special). The infinitesimal compositions \[
			\circ_i: \beM_{0,n+1} \times \beM_{0,m+1} \rightarrow \beM_{0,n+m},
			\] are usual gluing maps
			obtained by gluing the special point $\infty$  of the curve  from the right factor to the $i$-th point of the curve from the left one.
			
			\item For stellar graph $\St_n$, the resulting graphic compactification $\bM(\St_n)$ coincides with the Losev-Manin moduli space $\overline{\EuScript{L}}_{0,n}$ \cite{losev2000new}. This moduli space parametrizes chains of projective lines with two poles labeled by $\infty$ and $0$ and $n$ marked points that may coincide. More generally, for a complete multipartite graph $\K_{(1^n,m)}$, the resulting graphic compactification coincides with the extended Losev-Manin moduli space $\overline{\EuScript{L}}_{0;n,m}$ of stable curves with $n$ white and $m$ black marked points~\cite{losev2004extended}.
			\item For paths $\Path_n$, the graphic compactification $\bM(\Path_n)$ is an $(n-2)$-dimensional toric projective variety called Brick manifold. Its dual polytope is the Stasheff polytope $\mathcal{K}_{n-2}$. More generally, when $\Gr$ is a tree on $n$ vertices, the resulting graphic compactification $\bM(\Gr)$ is an $(n-2)$-dimensional toric projective variety with its dual polytope $\mathcal{K}_{\La(\Gr)}$, where $\La(\Gr)$ is a line graph (we replace vertices with edges) and $\mathcal{K}_{\La(T)}$ is the associated graph associahedron, for details see~\cite{dotsenko2024reconnectads}.
		\end{itemize}
	\end{example}
	The contractad language provides a simple combinatorial description of intersections of canonical divisors $D_{G}$ in $\bM(\Gr)$.\label{page::intersectionsofdivisors} From the general theory of wonderful compactifications, for a set $\Susp$ of tubes, the intersection $D_{\Susp}=\cap_{G\in \Susp} D_G$ is non-empty if and only if $\Susp$ is \textit{nested}, i.e., every two tubes $G,H$ from $\Susp$ are comparable by inclusion or do not intersect at all. Moreover, in this case, the intersection $D_{\Susp}=\cap_{G\in \Susp} D_G$ is transversal. Since $\bM$ is a contractad, we have a morphism of contractads
	\[
	\T(\bM)\to \bM,
	\] where $\T(\bM)$ is the free contractad on $\bM$. For a stable\footnote{A rooted tree is called stable if each internal vertex has at least 2 incoming edges} $\Gr$-admissible tree $T$, the image of $\bM((T))=\prod_{v\in \Ver(T)}\bM(\In(v))$ in $\bM(\Gr)$ is a closed subvariety and coincides with the intersection of divisors of the form
	\[
	\bM((T))=\bigcap_{e\in \Edge(T)}D_{L_e},
	\] where $L_e=\Leav(T_e)$ is the leaf set of subtree $T_e\subset T$. Moreover, each nested set has the form $\Susp(T)=\{L_e|e\in \Edge(T)\}$, for a unique $\Gr$-admissible stable rooted tree. So, a non-empty intersection of divisors $D_G$ is a product of graphic compactifications for smaller graphs. 
	
	Similarly, we could describe the stratification of $\bM(\Gr)$ in terms of admissible trees as follows. The component-wise inclusion $\M(\Gr)\hookrightarrow\bM(\Gr)$ produces a morphism of contractads $\T(\M)\to \bM$, where $\T(\M)$ is the free set-contractad on $\M$\footnote{we use convention $\M(\Path_1)=\varnothing$}. This morphism is an isomorphism of set-contractads. Moreover, for a stable $\Gr$-admissible tree $T$, the image of $\M((T))$ in $\bM(\Gr)$ is a locally open subset, whose closure coincides with $\bM((T))$. All in all, there is a locally open stratification of $\bM(\Gr)$
	\begin{equation}\label{eq::stratification}
		\bM(\Gr)=\bigsqcup_{T\in \Tree_{\mathrm{st}}(\Gr)}\M((T)), \quad \M((T))\cong \prod_{v\in \Ver(T)}\M(\In(v)).   
	\end{equation}
	\begin{example}
		For complete graphs, through the isomorphism $\bM(\K_n)\cong\beM_{0,n+1}$, the strata $\M((T))$ corresponds to the strata $\eM_{0,n+1}((T))$ consisting of stable curves with dual graph $T$. Similarly, for stellar graphs, through the isomorphism $\bM(\St_n)\cong\overline{\EuScript{L}}_{0,n}$, the strata $\M(T))$ corresponds to the strata $\EuScript{L}(((T))$ consisting of stable chains with dual graph $T$.
		\begin{figure}[ht]
			\caption{Dual graphs of stable chains are $\St$-admissible trees.}
			\[
			\vcenter{\hbox{\begin{tikzpicture}[scale=0.7]
						\draw[name path=1] (0,0)--(3,1);
						\draw[name path=2] (1.5,1)--(5,-0.25);
						\fill[name intersections={of=1 and 2}] (intersection-1) circle (2pt);
						\draw[name path=3] (4,-0.3)--(6,0.8);
						\fill[name intersections={of=2 and 3}] (intersection-1) circle (2pt);
						\draw[thick, fill=white] ($(0,0)!.1!(3,1)$) circle [radius=2pt];
						\node at (0.15,0.35) {\small$\infty$};
						\fill ($(0,0)!.5!(3,1)$) circle (2pt);
						\node at (1.6,0.25) {\small$2$};
						\fill ($(1.5,1)!.4!(5,-0.25)$) circle (2pt);
						\node at (3,0.2) {\small$1$};
						\fill ($(1.5,1)!.65!(5,-0.25)$) circle (2pt);
						\node at (4,0.4) {\small$3$};
						\fill ($(4,-0.3)!.5!(6,0.8)$) circle (2pt);
						\node at (4.9,0.55) {\small$4$};
						\draw[thick, fill=white] ($(4,-0.3)!.8!(6,0.8)$) circle [radius=2pt];
						\node at (5.9,0.35) {\small$0$};
						
			\end{tikzpicture}}}
			\quad
			\longrightarrow
			\quad
			\vcenter{\hbox{\begin{tikzpicture}[scale=0.7]
						\draw (0,0)--(0,0.75);
						\draw (0,0.75)--(-0.75,1.5);
						\draw (0,0.75)--(0.75,1.5);
						\draw (0.75,1.5)--(0.1,2.15);
						\draw (0.75,1.5)--(0.75,2.15);
						\draw (0.75,1.5)--(1.6,2.15);
						\draw (1.6,2.15)--(1.1,2.75);
						\draw (1.6,2.15)--(2.1,2.75);
						\node at (0,-0.25) {\small$\infty$};
						\node at (-0.75,1.75) {\small$2$};
						\node at (0.1,2.4) {\small$1$};
						\node at (0.75,2.4) {\small$3$};
						\node at (1.1,3) {\small$4$};
						\node at (2.1,3) {\small$0$};
			\end{tikzpicture}}}
			\]
		\end{figure}
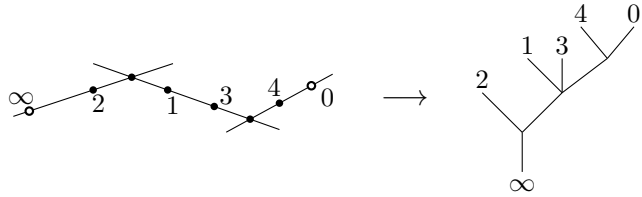 According to Example~\ref{ex::comp_wonderful}, when graph $\Gr$ is a tree, the component $\bM(\Gr)$ is a toric variety. In this case, the elements of strata coincide with the orbits of torus action.
	\end{example}
	
	\subsection{Complete multipartite graphs and Modular compactification}
	\label{sec::modular}
	\subsubsection{$\K_\lambda$-stable curves}
	\label{sec::sub::stable}
	In Example~\ref{ex::comp_wonderful}, we see that, for particular types of graphs, the corresponding component of the wonderful contractad $\bM$ can be described as the moduli space of stable curves with certain conditions on marked points. In this subsection, we show how this description extends to the whole class of complete multipartite graphs.
	
	For $n\geq 2$, let $\eM_{0,n+1}$ be the moduli space of $n$-pointed smooth curves of genus zero. Specifically, this is a space of complex projective lines $\mathbb{P}^1$ with $n+1$-marked pairwise distinct points (with one marked point singled out $\infty$ as a special) with two pointed projective lines isomorphic if there is a marked point preserving isomorphism between them. Recall that this moduli space admits a smooth compactification $\eM_{0,n+1}\subset \beM_{0,n+1}$, called a Deligne-Mumford compactification. This moduli space parametrizes $(n+1)$-pointed stable curves of genus zero, complex curves of genus zero with only singularities are double points with $(n+1)$ marked pairwise distinct smooth points.
	Let us define a moduli space, where marked points may coincide.

	\begin{defi}
		\label{def::multipartite::modular}
		Let $\mathbb{k}=\mathbb{C}$ or $\mathbb{R}$. For a connected complete multipartite graph $\K_{\lambda}$, a $\K_\lambda$-pointed stable curve is a pointed $\mathbb{k}$-curve $(\Sigma, x_{\infty},\{x_v\}_{v\in V_{\Gr}})$ of genus zero with only singularities are double points, satisfying the following conditions: 
		\begin{itemize}
			\item The marked point $x_{\infty}$ is separated from all other marked points.
			\item For adjacent vertices $v,w\in V_{\K_{\lambda}}$, the marked points are separated $x_v\neq x_w$.
			\item Each irreducible component $C\subset \Sigma$ not containing $x_{\infty}$ with a single node contains a pair of marked points $x_v,x_w\in C$ labeled by adjacent vertices $(v,w)\in E(\K_{\lambda})$.
		\end{itemize}
		The corresponding moduli stack is denoted $\beM_{0,\K_\lambda}$.
	\end{defi}
	We can define a stabilization map 
	\begin{equation}
		\label{eq::stabilization::map}
		\pi:\beM_{0,n+1}\to \beM_{0,\K_\lambda}
	\end{equation}
	that sends a stable rational curve to a $\K_\lambda$-stable rational curve by iterative contraction of all irreducible components that do not contain either $x_\infty$ or a pair of double points, or a pair of marked points labeled by adjacent vertices. Apriori the quotient curve may depend on the order in which we choose the components that we contract. However, in the case of multipartite graphs, the quotient curve is well-defined. 
	\begin{figure}[ht]
		\caption{Stabilization map $\pi\colon \beM_{0,5}\to \beM_{0,\Cyc_4}$ for $4$-cycle $\Cyc_4\cong \K_{(2^2)}$ with classical ordering. Note that the first curve remains the same. For the second curve, we collapse the component out of $\infty$, since vertices $2,4$ are not adjacent in $\Cyc_4$}
		\[
		\vcenter{\hbox{\begin{tikzpicture}[scale=0.7]
					\draw[name path=1] (0,0)--(4,1);
					\draw[name path=2] (3,1)--(5,-0.25);
					\fill[name intersections={of=1 and 2}] (intersection-1) circle (2pt);
					\draw[thick, fill=white] ($(0,0)!.1!(4,1)$) circle [radius=2pt];
					\node at (0.15,0.35) {\small$\infty$};
					\fill ($(0,0)!.3!(4,1)$) circle (2pt);
					\node at (1.4,0) {\small$1$};
					\fill ($(0,0)!.6!(4,1)$) circle (2pt);
					\node at (2.2,0.85) {\small$2$};
					\fill ($(3,1)!.4!(5,-0.25)$) circle (2pt);
					\node at (3.6,0.2) {\small$3$};
					\fill ($(3,1)!.8!(5,-0.25)$) circle (2pt);
					\node at (4.9,0.1) {\small$4$};
		\end{tikzpicture}}}
		\overset{\pi}{\mapsto}
		\vcenter{\hbox{\begin{tikzpicture}[scale=0.7]
					\draw[name path=1] (0,0)--(4,1);
					\draw[name path=2] (3,1)--(5,-0.25);
					\fill[name intersections={of=1 and 2}] (intersection-1) circle (2pt);
					\draw[thick, fill=white] ($(0,0)!.1!(4,1)$) circle [radius=2pt];
					\node at (0.15,0.35) {\small$\infty$};
					\fill ($(0,0)!.3!(4,1)$) circle (2pt);
					\node at (1.4,0) {\small$1$};
					\fill ($(0,0)!.6!(4,1)$) circle (2pt);
					\node at (2.2,0.85) {\small$2$};
					\fill ($(3,1)!.4!(5,-0.25)$) circle (2pt);
					\node at (3.6,0.2) {\small$3$};
					\fill ($(3,1)!.8!(5,-0.25)$) circle (2pt);
					\node at (4.9,0.1) {\small$4$};
		\end{tikzpicture}}}
		\text{, but }
		\vcenter{\hbox{\begin{tikzpicture}[scale=0.7]
					\draw[name path=1] (0,0)--(4,1);
					\draw[name path=2] (3,1)--(5,-0.25);
					\fill[name intersections={of=1 and 2}] (intersection-1) circle (2pt);
					\draw[thick, fill=white] ($(0,0)!.1!(4,1)$) circle [radius=2pt];
					\node at (0.15,0.35) {\small$\infty$};
					\fill ($(0,0)!.3!(4,1)$) circle (2pt);
					\node at (1.4,0) {\small$1$};
					\fill ($(0,0)!.6!(4,1)$) circle (2pt);
					\node at (2.2,0.85) {\small$3$};
					\fill ($(3,1)!.4!(5,-0.25)$) circle (2pt);
					\node at (3.6,0.2) {\small$2$};
					\fill ($(3,1)!.8!(5,-0.25)$) circle (2pt);
					\node at (4.9,0.1) {\small$4$};
		\end{tikzpicture}}}
		\overset{\pi}{\mapsto}
		\vcenter{\hbox{\begin{tikzpicture}[scale=0.7]
					\draw[name path=1] (0,0)--(4,1);
					\draw[thick, fill=white] ($(0,0)!.1!(4,1)$) circle [radius=2pt];
					\node at (0.15,0.35) {\small$\infty$};
					\fill ($(0,0)!.3!(4,1)$) circle (2pt);
					\node at (1.4,0) {\small$1$};
					\fill ($(0,0)!.6!(4,1)$) circle (2pt);
					\node at (2.2,0.85) {\small$3$};
					\fill ($(0,0)!.9!(4,1)$) circle (2pt);
					\node at (3.5,0.6) {\small$2,4$};
		\end{tikzpicture}}}
		\]
	\end{figure}
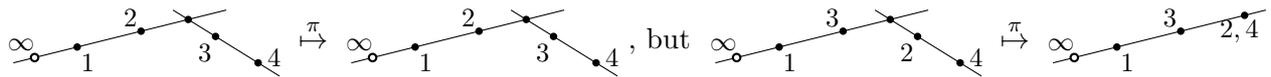
	
	Different quotients of $\beM_{0,n+1}$ of this kind were considered by different people. In particular, we recall the \emph{modular compactifications} due to Smyth that generalizes blow-down construction described in Definition~\ref{def::multipartite::modular}.
	Consequently, we conclude that $\beM_{0,\K_\lambda}$ is a smooth algebraic variety and the stabilization map $\pi$ is dominant.

	\subsubsection{Modular Compactifications}
	\label{sec::sub::modular}
	
	
	In addition to the Deligne-Mumford compactification, there are alternative compactifications, for example, the Hasset compactifications that parametrize weighted pointed stable curves.  In~\cite{smyth2009towards}, Smyth classified all modular compactifications   of $\eM_{0,n+1}$. A \textit{modular compactification} is an open proper substack $\beM'_{0,n+1}$ of the algebraic stack $\mathcal{U}_{0,n+1}$ of $(n+1)$-pointed smoothable curves of genus zero containing $\eM_{0,n+1}$. Smyth described the combinatorial data that indexed these compactifications and called them \textit{extremal assignments} which we recall below. 
	
	A stable rational $(n+1)$-pointed curve  $\Sigma$ determines the dual graph $T_{\Sigma}$: the vertices are components of the subvariety of smooth points of $\Sigma$, the edges are double points, and leaves are marked points $[n]\cup\{\infty\}$. This graph $T_{\Sigma}$ is a rooted tree with the root $\infty$ since $\Sigma$ of genus zero.
	Denote by $\Tree(n)$ the set of stable rooted trees with the root $\infty$ and $n$ labeled leaves.
	\begin{defi}
		An extremal assignment $\mathcal{Z}$ is a rule assigning with each tree $T\in \Tree(n)$, a subset $\mathcal{Z}(T)\subset \Ver(T)$, satisfying the following properties:
		\begin{itemize}
			\item[(i)] $\mathcal{Z}(T)\neq \Ver(T)$
			\item[(ii)] For any contraction of rooted trees $\pi\colon T\rightsquigarrow T'$, we have
			\[
			v'\in \mathcal{Z}(T')\Leftrightarrow \pi^{-1}(v')\subset \mathcal{Z}(T).
			\]
		\end{itemize}
	\end{defi}
	Each extremal assignment $\mathcal{Z}$ defines the following compactification of $\eM_{0,n+1}$ called \emph{modular compactification} after Smyth.
	\begin{defi}
		A smoothable $(n+1)$-pointed curve $(\Sigma,x_1,\cdots,x_n,x_{\infty})$ of genus-zero is $\mathcal{Z}$-stable if there exists a stable curve $(\Sigma',x'_1,\cdots,x'_n,x'_{\infty})$ of genus zero with dual graph $T_{\Sigma'}$ and a morphism of pointed curves $\pi\colon (\Sigma',x'_1,\cdots,x'_n,x'_{\infty})\to(\Sigma,x_1,\cdots,x_n,x_{\infty})$, satisfying
		\begin{itemize}
			\item[(i)] $\pi$ is surjective with connected fibers 
			\item[(ii)] $\pi|_{\Sigma'\setminus \mathcal{Z}(\Sigma')}$ is an isomorphism
			\item[(iii)] the image of each connected component from $\mathcal{Z}(\Sigma')$ is a single point,
		\end{itemize} where $\mathcal{Z}(\Sigma')\subset \Sigma$ is a component obtained from the irreducible components labeled by vertices of dual graphs from $\mathcal{Z}(T_{\Sigma})$.
	\end{defi}
	Let $\beM_{0,n+1}(\mathcal{Z})$ be the resulting moduli stack of $\mathcal{Z}$-stable curves.
	
	\begin{theorem}[\cite{smyth2009towards}]
		Let $\mathcal{Z}$ be an extremal assignment. Then, the moduli stack $\beM_{0,n+1}(\mathcal{Z})$ of $\mathcal{Z}$-stable curves is a modular compactification. Conversely, for any modular compactification $\beM'$ there is an extremal assignment $\mathcal{Z}$, such that $\beM'\cong\beM_{0,n+1}(\mathcal{Z})$. Furthermore, the compactification $\beM_{0,n+1}(\mathcal{Z})$ is represented by an algebraic space.
	\end{theorem}
	
	For particular types of assignments, the resulting modular compactification is represented by a smooth proper variety. Let $\Tree_2(n)\subset \Tree(n)$ be a subset of rooted trees with 2 internal vertices. An extremal assignment $\mathcal{Z}$ is called \textit{smooth}, if, for each pair of rooted tree $T\in\Tree(n)$ and a vertex $v\in \mathcal{Z}_{\Gr}(T)$, there exists a tree $T'\in\Tree_2(n)$ and a contraction $\pi\colon T\rightsquigarrow T'$ with  $\pi(v)\in \mathcal{Z}_{\Gr}(T')$.
	\begin{theorem}~\cite{moon2018birational}
		Let $\mathcal{Z}$ be an extremal smooth assignment. Then the moduli stack $\beM_{0,n+1}(\mathcal{Z})$ is represented by a smooth proper algebraic variety.
	\end{theorem}
	\subsubsection{Graphic Modular compactifications} 
	Let $\K_\lambda$ be a complete multipartite graph on the vertex set $[n]$. For a rooted tree $T$ on $n$ leaves and vertex $v\in \Ver(T)$, let $T(v)\subset [n]$ be a set of leaves of the subtree $T_v\subset T$ outgoing from the vertex set $v$.  We consider an assignment $\mathcal{Z}_{\K_\lambda}$ on the rooted trees $\Tree(n)$, given by the rule
	\[
	\mathcal{Z}_{\K_\lambda}(T)=\{v\in \Ver(T)| T(v)\subset [n]\text{- is not a tube of} \K_\lambda\}
	\]
	\begin{prop}
		For a connected complete multipartite graph $\K_\lambda$ on the vertex set $[n]$, the assignment $\mathcal{Z}_{\K_\lambda}$ is extremal and smooth.
	\end{prop}
	\begin{proof}
		Firstly, this assignment is extremal. Indeed, since $\K_\lambda$ is connected, for the root vertex $T(r)=[n]$ is a tube, so $\mathcal{Z}_{\K_\lambda}(T)\neq \Ver(T)$. Next, consider a contraction $\pi\colon T\rightsquigarrow T'$. For a vertex $w'\in \Ver(T')$, we have $T(v)\subset T(w)$ for all $v\in \pi^{-1}(w)$.  Recall that, for a complete multipartite graph, a vertex subset forms a tube if it contains at least one edge. So, if $T(w)$ is not a tube, it implies that $T(v)$ is also not a tube for all $v\in \pi^{-1}(w')$.  The second condition of extremal assignment follows from this observation.  
		
		Secondly, the assignment $\mathcal{Z}_{\K_\lambda}$ is smooth. Indeed, given a vertex $v\in \mathcal{Z}_{\K_\lambda}(T)$, we consider a rooted 2-tree $T'$  with root vertex $v_{root}$ and non-root vertex $v'$ with $T(v')=T(v)$. In particular, we have $\mathcal{Z}_{\K_\lambda}(T')=v'$. We consider a contraction $\pi\colon T\rightsquigarrow T'$ that maps all vertices $v\leq w$ to $v'$, and all other vertices to the root vertex.
	\end{proof}
	
	By direct inspections, we get the following
	\begin{prop}
		For a complete multipartite graph $\K_\lambda$
		the definitions of $\mathcal{Z}_{\K_{\lambda}}$-stable curves and $\K_{\lambda}$-pointed stable curves coincide.
		In particular, $\beM_{0,\K_\lambda}$ is a smooth modular compactification associated with extremal assignment $\mathcal{Z}_{\K_\lambda}$.
		Furthermore, this moduli stack is represented by a smooth projective algebraic variety.
	\end{prop}
	As one of the consequences, we see that the stabilisation map~\eqref{eq::stabilization::map} that sends each stable curve $[(\Sigma,x_{\infty},x_{1},x_{2},\cdots,x_n)]$ to the $\K_\lambda$-stable curve by collapsing component from $\mathcal{Z}(T_{\Sigma})$ is a morphism of the underlying algebraic varieties.  It was shown in~\cite{moon2018birational}, that a stabilising morphism is a regular dominant map. Later we show that this morphism is a sequence of blow-downs.
	
	\subsubsection{Isomorphism between wonderful contractad and modular compactifications}
	
	Consider the open subvariety $\eM_{0,\K_\lambda}\subset \beM_{0,\K_\lambda}$ parametrising smooth $\K_\lambda$-pointed curves. By definition, a smooth $\K_\lambda$-pointed curve is a projective line $\Pro^1$ with a function $f\colon V_{\K_\lambda}\cup \{\infty\}\to \Pro^1$ such that the image of $\infty$ does not belong to the image of vertices and images of adjacent vertices do not coincide. In other words, it is an element of the graphic configuration space $\Conf_{\Cyc_{\infty}\K_\lambda}(\Pro^1)/\PGL_2(\mathbb{k})$, where $\Cyc_{\infty}\K_\lambda:=\K_{(1)\cup \lambda}$ is the graph obtained from $\K_\lambda$ by adding a vertex $\infty$ adjacent to all vertices. Moreover, two $\K_\lambda$-pointed smooth curves are equivalent if and only if there exists a marked point preserving automorphism of $\Pro^1$ between them. So, the moduli space $\eM_{0,\K_\lambda}$ coincides with the orbit space $\Conf_{\Cyc_{\infty}\K_\lambda}(\Pro^1)/\PGL_2(\mathbb{k})$.
	\begin{prop}\label{prop:openpart}
		The moduli space of $\K_\lambda$-pointed smooth curves $\EuScript{M}_{0,\K_\lambda}$ is isomorphic to the projective complement of graphic arrangement $\M(\K_\lambda)$.
	\end{prop}
	\begin{proof}
		By transitivity of $\PGL_2(\mathbb{k})$-action on $\Pro^1$, in each class $[f]$ from $\eM_{0,\K_\lambda}$ we can choose a representative that sends the special point to $\infty\in \Pro^1$. In other words, we have the isomorphism $\eM_{0,\K_\lambda}\cong \Conf_{\K_\lambda}(\Pro^1\setminus \{\infty\})/\PGL_2(\mathbb{k})_{\infty} \cong \Conf_{\K_\lambda}(\mathbb{k})/\Aff_1(\mathbb{k})$. By the definition, this graphic configuration space $\Conf_{\K_\lambda}(\mathbb{k}):=\mathbb{k}^{V_{\K_\lambda}}\setminus \B(\K_\lambda)$ is the complement to the graphic braid arrangement. The quotient of this space by the translation subgroup $\mathbb{k}^+\subset \Aff_1(\mathbb{k})$ is exactly the quotient by the diagonal subspace $\langle \sum_v x_v\rangle$, then the quotient by scalars $\mathbb{k}^{\times} \subset \Aff_1(\mathbb{k})$ is exactly the projectivization of the latter space. So, in the end, we get the projective complement $\M(\K_\lambda):=\Pro(W_{\K_\lambda})\setminus \cup_{e \in E_{\K_\lambda}} \Pro(H_e)$.
	\end{proof}
	
	Note that the morphism $\rho\colon \EuScript{M}_{0,\K_\lambda}\to \M(\K_\lambda)$ constructed above lifts to the morphism $\overline{\rho}\colon \beM_{0,\K_\lambda}\to \Pro(V)$. Indeed, the lifting is defined as follows: for an arbitrary $\K_\lambda$-stable curve we collapse all components not containing special point $\infty$, hence we obtain a map $f\colon V_{\K_\lambda}\cup \{\infty\}\to \Pro^1$. In contrast with smooth curves, while collapsing components some marked points labelled by adjacent vertices can "stick together" at the end of process, so the resulting map $f$ does not necessarily belong to the graphic configuration space $\Conf_{\Cyc_{\infty}\K_\lambda}(\Pro^1)$.
	
	Also, for a tube $G$, we associate the collapsing map $\pi_G\colon\beM_{0,\K_\lambda}\to \beM_{0,\K_\lambda|_G}$. This map is induced by ignoring all marked points out of $G$ and collapsing components as necessary to preserve stability.
	
	Let $G$ be a tube of the connected complete multipartite graph $\K_\lambda$. For $\K_\lambda/G$-stable curve $\Sigma$ and $\K_\lambda|_G$-stable curve $\Sigma'$, we produce the $\K_\lambda$-stable curve $\Sigma''$ by joining the special point of the curve $\Sigma'$ to the point labeled by $\{G\}$ of the curve $\Sigma$. By Lemma~\ref{lemma:yungcone}, this procedure is well-defined since both points are separated from other marked points. Moreover, this procedure induces the gluing map
	\[
	\circ^{\K_\lambda}_G\colon \beM_{0,\K_\lambda/G}\times\beM_{0,\K_\lambda|_G}\to \beM_{0,\K_\lambda}.
	\] 
	Note that the moduli space $\beM_{0,\K_\lambda}$ admits a locally open stratification by dual graphs of stable curves
	\[
	\beM_{0,\K_\lambda}=\coprod_{T\in \Tree_{\mathrm{st}}(\K_\lambda)} \EuScript{M}_{0,\K_\lambda}((T)).
	\] Moreover, the substitution maps define isomorphisms of the form
	\begin{equation}\label{eq:3}
		\circ^{\K_\lambda}_G\colon \EuScript{M}_{0,\K_\lambda/G}((T))\times\EuScript{M}_{0,\K_\lambda|_G}((T'))\overset{\cong}{\rightarrow}\EuScript{M}_{0,\K_\lambda}((T\circ^{\K_\lambda}_G T')).
	\end{equation}
	Hence elements of strata are direct products of moduli spaces of graph-stable smooth curves
	\begin{equation}\label{eq:4}
		\EuScript{M}_{0,\K_\lambda}((T))\cong \prod_{v\in \Ver(T)}\EuScript{M}_{0,\In(v)}. 
	\end{equation}
	Let us state the main result of this subsection.
	\begin{theorem}\label{thm::modular=wonderful}
		For a connected complete multipartite graph $\K_{\lambda}$, there is an isomorphism of algebraic varieties
		between modular compactification $\beM_{0,\K_\lambda}$ and graphic compactification $\bM(\K_\lambda)$.
		Moreover, these isomorphisms preserve contractad structures.
	\end{theorem}
	\begin{proof}
		First, we construct the morphism $\phi\colon\beM_{0,\K_\lambda}\to \prod_{G\in \G(\K_\lambda)}\Pro(V/G)$ as follows. For each tube $G\in \G(\K_\lambda)$, we take the composition $\beM_{0,\K_\lambda}\overset{\pi_G}{\rightarrow}\beM_{0,\K_\lambda|_G}\overset{\overline{\rho}_G}{\rightarrow}\Pro(V/G)$. By local conditions, we see that $\phi$ maps $\beM_{0,\K_\lambda}$ to the graphic compactification $\bM(\K_\lambda)$. Hence we obtain a well-defined morphism of smooth proper varieties $\phi\colon\beM_{0,\K_\lambda}\to \bM(\K_\lambda)$. By comparing gluing maps and Rain's maps, we have a commutative diagram:
		\[\begin{tikzcd}
			\beM_{0,\K_\lambda/G}\times\beM_{0,\K_\lambda|_G} && \beM_{0,\K_\lambda} \\
			\\
			\bM(\K_\lambda/G)\times\bM(\K_\lambda|_G) && \bM(\K_\lambda)
			\arrow["{\circ^{\K_\lambda}_G}", from=1-1, to=1-3]
			\arrow["\phi",from=1-3, to=3-3]
			\arrow["\phi\times\phi"',from=1-1, to=3-1]
			\arrow["{\circ^{\K_\lambda}_G}",from=3-1, to=3-3]
		\end{tikzcd}\] Hence this morphism is compatible with contractad structures on both sides. By Proposition~\ref{prop:openpart}, the restriction to smooth curves defines the isomorphism $\phi|_{\EuScript{M}_{0,\K_\lambda}}\colon \EuScript{M}_{0,\K_\lambda}\to \M(\K_\lambda)$. Hence the morphism $\phi$ is birational. Combining this observation with the compatibility of contractad and gluing maps, we conclude that $\phi$ defines an isomorphism on each element of strata $\phi|_T\colon \EuScript{M}_{0,\K_\lambda}((T))\to \M((T))$. So, the morphism $\phi$ is birational and bijective, hence it is an isomorphism.
	\end{proof}
	\noindent As an immediate corollary, we have
	\begin{sled}
		\begin{itemize}
			\item[(i)] The variety $\beM_{0,\K_\lambda}$ is projective. \item[(ii)] The complement $\beM_{0,\K_\lambda}\setminus \eM_{0,\K_\lambda}$ is a divisor with normal crossings with irreducible components $\beM((T))$ for $T\in \Tree_2(\K_\lambda)$.
			\item[(iii)] The stabilising morphism $\pi\colon \beM_{0,n+1}\to\beM_{0,\K_\lambda}$ is a sequence of blow-downs.
		\end{itemize}
	\end{sled}
	\begin{proof}
		The first two statements follow from the properties of wonderful compactifications. Note that the stabilising map $\pi\colon \beM_{0,n+1}\cong \bM(\K_n)\rightarrow \bM(\K_\lambda)\cong\beM_{0,\K_\lambda}$ is the morphism of wonderful compactifications induced from the inclusion of building sets $\G(\K_\lambda)\hookrightarrow \G(\K_n)$ on the same vector space, that is automatically a sequence of blow-downs~\cite{rains2010homology}.
	\end{proof}
	
	\section{(Co)Homology of the Complex Wonderful contractad \texorpdfstring{$\bM_{\mathbb{C}}$}{MC}}
	\label{sec::H::wonderful::C}
	
	In this section, we demonstrate that the (co)homology of the wonderful contractad $\bM_{\mathbb{C}}$ forms a quadratic Koszul contractad. We detail its generators and relations and present computations of the Hilbert series. The proof of Koszulness extends the strategy discovered by E. Getzler in \cite{getzler1994two,getzler1995operads}. To determine the generators and relations, we employ the theory of Gr\"obner bases and generalize the computations suggested in \cite{dotsenko2013quillen}. One of the main corollaries of our work is the description of the Hilbert series for modular compactifications.

	\subsection{Cohomology rings for complex variety \texorpdfstring{$\bM_{\mathbb{C}}$}{MC}}
	\label{sec::H::MC::ring}
	In this subsection, we describe cohomology rings of $\Conf_{\Gr}(\mathbb{C})$, $\M_{\mathbb{C}}(\Gr)$ and its compactification $\bM_{\mathbb{C}}(\Gr)$.

	We recall the graphic configuration spaces $\Conf_{\Gr}(\mathbb{R}^2)$, see \S\ref{sec:disks}. We shall identify the real plane $\mathbb{R}^2$ with complex line $\mathbb{C}$. Note that the space $\Conf_{\Gr}(\mathbb{C})$ coincides with the affine complement to the graphic arrangement $\B(\Gr)$ 
	\[
	\Conf_{\Gr}(\mathbb{C})=\{(x_v)\in \mathbb{C}^{V_{\Gr}}| x_v\neq x_w\text{ for }(v,w)\in E_{\Gr}\}=\mathbb{C}^{V_{\Gr}}\setminus \B(\Gr).
	\]
	
	Recall that the cohomology ring of a complement to a complex hyperplane arrangement is isomorphic to the respective Orlik-Solomon algebra~\cite{yuzvinsky2001orlik}. In the case of graphic arrangements, we obtain the following presentation of $H^{\bullet}(\Conf_{\Gr}(\mathbb{C}))$. For an edge $e=(v,w)\in E_{\Gr}$, let $\omega_e$ be the logarithmic differential form on $\Conf_{\Gr}(\mathbb{C})$ given by the formula
	\[
	\omega_e=\frac{1}{2\pi i}d\log(x_v-x_w)=\frac{1}{2\pi i}\frac{d(x_v-x_w)}{x_v-x_w}
	\] Note that the cohomology class of $\omega_e$ is integral.
	
	\begin{prop}\cite{lyskov2023contractads}\label{prop::cohomologyConf}
		The cohomology ring $H^{\bullet}(\Conf_{\Gr}(\mathbb{C}))$ is the $\mathbb{Z}$-ring generated by logarithmic forms $\omega_e$, satisfying the relations
		\[
		\sum_{i=1}^n (-1)^{i-1}\omega_{e_1}\omega_{e_2}...\hat{\omega}_{e_i}...\omega_{e_n}=0, \quad \text{if}\quad \{e_1,e_2,\cdots,e_n\} \subset E_{\Gr}\quad \text{contains a cycle.}
		\]
	\end{prop}
	There is a natural action of the affine group $\Aff_1(\mathbb{C})$ on the graphic configuration space $\Conf_{\Gr}(\mathbb{C})$. Specifically, for a configuration $f\colon V_{\Gr}\to \mathbb{C}$ and a transformation $\phi\in \Aff_1(\mathbb{C})$, we have $\varphi.f=\phi\circ f$.
	\begin{lemma}\label{lemma::confandM}
		The action of $\Aff_1(\mathbb{C})$ on $\Conf_{\Gr}(\mathbb{C})$ is free and there is the isomorphism of varieties
		\[
		\Conf_{\Gr}(\mathbb{C})/\Aff_1(\mathbb{C})\cong\M_{\mathbb{C}}(\Gr).
		\]
	\end{lemma}
	\begin{proof}
		The proof repeats the ones of Proposition~\ref{prop:openpart}.
	\end{proof} Similarly to the case of an affine complement to hyperplane arrangement, there is the projective analogue of the Orlik-Solomon algebra~\cite{yuzvinsky2001orlik} in the case of a projective complement. The action of scalar subgroup $\mathbb{C}^{\times}$ produces the differential $\partial$ on the cohomology ring $H^{\bullet}(\Conf_{\Gr}(\mathbb{C}))$ given on generators by the rule $\partial\omega_e=1$.
	\begin{prop}\label{prop::cohomologyM}
		The cohomology ring $H^{\bullet}(\M_{\mathbb{C}}(\Gr))$ may be identified with the kernel of the differential $\partial$ on $H^{\bullet}(\Conf_{\Gr}(\mathbb{C}))$ whose action on the generators is $\partial\omega_e=1$.
	\end{prop}
	\begin{proof} Follows from the description of cohomology rings of projective complement to hyperplane arrangements~\cite{yuzvinsky2001orlik} and Lemma~\ref{lemma::confandM}.
	\end{proof}
	Now, consider the case of graphic compactification $\bM(\Gr)$. Thanks to Remark~\ref{rem::wondcomp}, we can apply the description of cohomology rings of wonderful compactifications~\cite{de1995wonderful} to the case of graphic compactifications.
	\begin{prop}[Divisor presentation]\label{prop:divisorpres}
		The cohomology ring $H^{\bullet}(\bM_{\mathbb{C}}(\Gr))$ is the $\mathbb{Z}$-ring generated by Poincare dual classes $[D_G]$ of canonical divisors, labeled by non-trivial proper tubes $G\subset V_{\Gr}$,  satisfying the relations
		\begin{gather}
			[D_G][D_H]=0, \text{ for } G\cap H\neq \varnothing\text{ and }G,H\text{ are incomparable}
			\\
			\sum_{G: e\subset G} [D_G]=\sum_{G': e'\subset G'} [D_{G'}], \text{ for each pair of edges }e,e'\in E_{\Gr}.\label{eq::linearelforbM}
		\end{gather}
	\end{prop}
	\begin{proof}
		Follows from the description of cohomology rings of wonderful compactifications in the case of the graphic building set $\G(\Gr)$~\cite{de1995wonderful}.
	\end{proof}
	Now, we present an alternative description of cohomology rings.  For each tube $G \subset V_{\Gr}$, we have the projection:
	\begin{equation*}
		\rho_G: \bM_{\mathbb{C}}(\Gr) \rightarrow \Pro(W_{\Gr}/H_G)
	\end{equation*}
	Let us denote by $\xi_G$ the pullback of tautological bundle $\mathcal{O}_{\Pro}(1)$ via this projection: $\xi_G=\rho_G^*\mathcal{O}_{\Pro}(1)$. We define the $h$-class as the first Chern class of this line bundle: $h_G=c_1(\xi_G)$. In particular, for an edge $e$, we have $h_e=0$. Let us express a divisor presentation of $h$-classes.
	
	\begin{lemma}\label{lemma:divisorhpres}
		For a tube $G$, the class $h_G$ is given by the formula
		\[
		h_G = \sum_{K: e\subset K, G\not\subset K} [D_K],
		\] where the sum ranges over all tubes $K$ that contains fixed edge $e$ and does not contain $G$. This formula is valid for any edge $e\subset G$.
	\end{lemma}
	\begin{proof}
		Recall that the space of global sections of $\Orb_{\Pro(W_{\Gr}/H_G)}(1)$ is the dual space $(W_{\Gr}/H_G)^*=H_G^{\bot}$. For an edge $e=(v,w)\subset G$ we consider a global section $s_e=x^*_v-x^*_w$. Its pullback $\tau_e=(\rho_G)^*s_e$ defines a global section of $\xi_G$. Let us describe the zero locus of this section. Note that the restriction of this section on $\M(\Gr)$ is nowhere zero. From the description of contractad maps~\eqref{eq::rainsmaps}, we see that zero of this section belongs to divisors $D_K$ labelled by tubes $K$ containing edge $e$ and not containing tube $G$. Indeed, for $G\not\subset K$, we have the commutative diagram
		\[\begin{tikzcd}
			{\bM(\Gamma/K)\times\bM(\Gamma|_K)} & {\bM(\Gamma)} & {\mathbb{P}(W_{\Gamma}/H_G)} \\
			{\bM(\Gamma/K)} && {\mathbb{P}(W_{\Gamma/K}/H_{G/K})}
			\arrow["{\circ^{\Gamma}_K}", from=1-1, to=1-2]
			\arrow["{\rho_G}", from=1-2, to=1-3]
			\arrow["{\pi_{\Gamma/K}}"', from=1-1, to=2-1]
			\arrow["{\rho_{G/K}}", from=2-1, to=2-3]
			\arrow["{\mathbb{P}(j^G_{K})}", from=2-3, to=1-3]
		\end{tikzcd}\]
		So, we have $\tau_e|_{D_K}=\pi_{\Gr/K}^*\rho_{G/K}^*(\Pro(j^G_K)^*s_e)$. Note that the linear dual $(j^G_K)^{*}\colon H_{G}^{\bot}\twoheadrightarrow H_{G/K}^{\bot}$ is the restriction of the projection $\mathbb{C}^{G}\twoheadrightarrow \mathbb{C}^{G/K}$ given by $x^*_v\mapsto x^*_v$ for $v\in G\cap K$ and $x^*_{w}\mapsto x^*_{K}$ otherwise. So, when $K$ contains edge $e$, we have $(j^G_K)^{*}(x^*_v-x^*_w)=0$, hence $\Pro(j^G_K)^*s_e=0$. This implies that the restriction $\tau_e|_{D_K}$ is identically zero. Since the section $\tau_e$ has no poles, the Chern class $h_G=c_1(\xi_{G})$ is given by the zero locus of $\tau_e$ that implies the desired formula.
	\end{proof}
	Let us define a formal element $[D_{\Gr}]=\sum_{e\subset G}[D_G]$, that is independent of the choice of edge $e$. In this notation, the $h$-class is rewritten in a suitable way
	\[
	h_G=-\sum_{K: G\subset K} [D_G]
	\]
	In~\cite{yuzvinsky2002small}, these classes were introduced for arbitrary building sets (they denote these classes by symbol $\sigma$). Moreover, the authors described a presentation of cohomology rings in terms of $h$-classes. Applying their results to our case, we obtain a new presentation of $H^{\bullet}(\bM(\Gr))$.   
	\begin{prop}[h-presentation]~\cite{pagaria2021hodge}
		The cohomology ring $H^{\bullet}(\bM_{\mathbb{C}}(\Gr))$ is generated by h-classes $h_G$, for tubes $|G|\geq 3$, satisfying the relations
		\begin{gather}
			h_G^2=0 \text{, for } |G| =3
			\\
			h_G(h_G-h_H) \text{, if } G \text{ covers } H \text{ as sets}
			\\
			(h_{H_1\cup H_2}-h_{H_1})(h_{H_1\cup H_2}-h_{H_1}) \text{, if } H_1\cap H_2\neq \varnothing.
		\end{gather}
	\end{prop}
	The advantage of these classes is the compatibility with contractad maps. 
	\begin{prop}
		\label{prop::hypermaps}
		The contractad structure on the wonderful contractad $\bM_{\mathbb{C}}$ induces a cocontractad structure on the graphical collection of cohomology rings $H^{\bullet}(\bM_{\mathbb{C}})$. The infinitesimal compositions are homomorphisms of algebras given by the rule:
		\begin{gather}
			\triangle^{\Gr}_G:=(\circ^{\Gr}_G)^*\colon H^{\bullet}(\bM_{\mathbb{C}}(\Gr))\to H^{\bullet}(\bM_{\mathbb{C}}(\Gr/G))\otimes H^{\bullet}(\bM_{\mathbb{C}}(\Gr|_G))
			\\
			h_H\mapsto \begin{cases}
				1\otimes h_H\text{, if }H\subset G;
				\\
				h_{H/G}\otimes 1\text{, else,}
			\end{cases}
		\end{gather}
		where $H/G$ is the image of $H$ under the contraction $\Gr \rightarrow \Gr/G$.
	\end{prop}
	\begin{proof}
		The formula above follows from the descriptions of pullback $(\circ^{\Gr}_G)^*\xi_H$ and functorial properties of Chern classes.    
	\end{proof} 
	
	\subsection{Gravity and Hypercommutative contractads}
	\label{sec::gravity}
	In this subsection, we define contractads generalising the operads of Gravity and Hypercommutative algebras, first introduced by Getzler in~\cite{getzler1994two,getzler1995operads}.
	
	Consider the little 2-disks contractad $\D_2$ (see \S\ref{sec:disks}). On each component $\D_2(\Gr)$ there is the action of the circle $S^1$ induced by the rotation of the unit disk. Explicitly, for each element $\phi\in S^1$ and disk configurations $f\colon\coprod_{v\in V_{\Gr}}\Disc_v\to \Disc$, we have $\phi.f=R_{\phi}\circ f$, where $R_{\phi}$ is the rotation of $\Disc$ with angle $\phi$. With respect to this action, the contractad $\D_2$ forms a contractad in the category of $S^1$-spaces.
	
	The action of the group $S^1$ on $\D_2$, produces the action of Grassman algebra $H_{\bullet}(S^1)=\mathsf{k}[\Delta]/(\Delta^2)$, where $\Delta=[S^1]$ is the fundamental class, on the Gerstenhaber contractad $H_{\bullet}(\D_2)=\Gerst$ ( see \S\ref{sec::gerst}). Since $\Delta^2=0$, $\Delta$ defines a square-zero derivation of degree $1$, given on generators by the rule
	\begin{gather*}
		\Delta(m) = b
		\\
		\Delta(b) = 0
	\end{gather*}
	Since the $S^1$-action on $\D_2$ is free, we immediately get
	\begin{prop}
		\label{prp::Gerst::splits}
		The dg contractad $(\Gerst, \Delta)$ is acyclic. 
	\end{prop}
	The kernel of a derivation of a contractad is itself a contractad, so we may give the following definition.
	\begin{defi}
		The gravity contractad $\Grav$ is the kernel
		\[
		\Grav:=\ker \Delta \subset \Gerst.
		\]
	\end{defi}
	
	The main result of this subsection is an explicit presentation of the contractad $\Grav$.
	\begin{theorem}\label{thm:gravpres}
		The contractad $\Grav$ is  generated by symmetric generators $\lambda_{\Gr}$ of degree 1 in each component $\Gr$, satisfying the relations: For each graph $\Gr$, we have
		\begin{gather}
			\sum_{e \in E_{\Gr}} \lambda_{\Gr/e} \circ_e \lambda_{\Gr|_e}=0,\label{eq:edgetype}
			\\
			\sum_{e \in E_{\Gr|_G}} \lambda_{\Gr/e} \circ_e \lambda_{\Gr|_e}=\lambda_{\Gr/G} \circ^{\Gr}_G \lambda_{\Gr|_G}\text{, for tubes  with at least 3 vertices.}\label{eq:tubetype}
		\end{gather} Moreover, this contractad is Koszul.
	\end{theorem}
	\begin{proof}
		For graph $\Gr$, let $\lambda_{\Gr}$ be
		\[
		\lambda_{\Gr}:=\Delta(m_{\Gr})=\sum_{e\in E_{\Gr}} m_{\Gr/e}\circ^{\Gr}_e b.
		\] In particular, for path $\Path_2$, we have $\lambda_{\Path_2}:=b$.
		
		Let us show that the elements $\lambda_{\Gr}$ generate the contractad $\Grav$. Recall that we have an isomorphism $\Gerst \cong \Com\circ \Susp^{-1}\Lie$ on the level of graphical collections, hence each component $\Gerst(\Gr)$ is spanned by monomials of the form $(m_{\Gr/I};b^{(I_1)},b^{(I_2)},...,b^{(I_k)})$, where $I=\{I_1,...,I_k\}$ is a graph partition of $\Gr$ and $b^{(I_j)}\in\Susp^{-1}\Lie(\Gr|_{I_j})$. Since the complex $(\Gerst,\Delta)$ is acyclic, we have $\Grav_{\bullet}(\Gr)=\Delta(\Gerst_{\bullet-1}(\Gr))$. So, each component of $\Grav$ is spanned by monomials of the form
		\begin{multline*}
			\Delta(m_{\Gr/I};b^{(I_1)},b^{(I_2)},...,b^{(I_k)})=
			\\
			=(\Delta m_{\Gr/I};b^{(I_1)},b^{(I_2)},...,b^{(I_k)})+ \sum_{i=1}^k \pm (m_{\Gr/I};b^{(I_1)},b^{(I_2)},...,\Delta b^{(I_i)},...,b^{(I_k)})=
			\\
			=(\lambda_{\Gr/I};b^{(I_1)},b^{(I_2)},...,b^{(I_k)}).
		\end{multline*} 
		From these observations and identity $b=\lambda_{\Path_2}$, we conclude that elements $\{\lambda_{\Gr}\}$ generate $\Grav$ as a contractad. Next, let us verify the relations. The relations of the first type~\eqref{eq:edgetype} follow from the identity
		\[
		\Delta(\lambda_{\Gr})=\Delta\left(\sum_{e\in E_{\Gr}} m_{\Gr/e}\circ^{\Gr}_e b\right)=\sum_{e\in E_{\Gr}} \lambda_{\Gr/e}\circ^{\Gr}_e \lambda_{\Gr|_e},
		\] and the fact that $\lambda_{\Gr}$ is a boundary, $\Delta(\lambda_{\Gr})=0$. The relations of the second type~\eqref{eq:tubetype} follow from the identity
		\begin{multline}
			\lambda_{\Gr/G}\circ^{\Gr}_G\lambda_{\Gr|_G}=\Delta(m_{\Gr/G}\circ^{\Gr}_G \lambda_{\Gr|_G})=\Delta\left(m_{\Gr/G}\circ^{\Gr}_G\left(\sum_{e\in E_{\Gr|_{G}}} m_{\Gr/e}\circ^{\Gr}_e b\right)\right)=
			\\
			=\Delta\left(\sum_{e\in E_{\Gr|_G}} m_{\Gr/e}\circ^{\Gr}_e b\right) = \sum_{e \in E_{\Gr|_G}} \lambda_{\Gr/e} \circ_e \lambda_{\Gr|_e}.
		\end{multline} 
		All in all, we have the onto morphism of contractads $\Pop \twoheadrightarrow \Grav$, where $\Pop$ is the contractad obtained from the generators and relations of this form. We claim that this morphism is exactly an isomorphism. Since $(\Gerst, \Delta)$ is an acyclic complex, we have dimension equality $\dim \Grav(\Gr)=\frac{1}{2}\dim \Gerst(\Gr)$. Therefore, by dimension reasons, it suffices to check that, for each  graph $\Gr$, we have the inequality $2\dim \Pop(\Gr)\leq \dim \Gerst(\Gr)$. 
		
		To establish this inequality, we use Gr\"obner bases for shuffle contractads from \S\ref{sec::Grobner::contractad}. Consider the shuffle version $\Pop^{\forget}:=\T_{\Sha}(\lambda_{\Gr})/\langle \R^{\forget} \rangle$ and the modified $\grpermlex$-order on tree monomials $\Tree_{\{\lambda_{\Gr}\}}$ as follows. For elements $\lambda_{\Gr}, \lambda_{\Omega}$ we put $\lambda_{\Gr}<\lambda_{\Omega}$ if $|V_{\Gr}|>|V_{\Omega}|$. Consider some total refinement of this order and $\mathsf{deglex}$-order on words $\{\lambda_{\Gr}\}^*$ associated with this extension. For a tree monomial $T\in \Tree_{\{\lambda_{\Gr}\}}$, we define its $2$-weight as the minus number of internal vertices of valency $3$. To compare two tree monomials, we first compare their 2-weights, and if their 2-weights are equal, we compare monomials with respect to the $\grpermlex$-extension  of monomial order on $\{\lambda_{\Gr}\}^*$.
		
		The leading terms of quadratic relations $\R^{\forget}(\Gr,<)$ with respect to this order have the form
		\begin{gather}
			\lambda_{\Gr/G}\circ^{\Gr}_{G} \lambda_{\Gr|_G}\text{, for tubes } G  \text{ with at least 3 vertices}\label{eq:lead1}
			\\
			\lambda_{\Gr/e}\circ_e^{\Gr}\lambda_{\Gr|_e}\text{, for edge } e \in E_{\Gr}\text{, containing minimal vertex and its minimal adjacent vertex}.\label{eq:lead2}
		\end{gather}
		
		Let $T$ be a normal monomial with respect to $\R^{\forget}$. The leading terms of the form~\eqref{eq:lead1}, ensure us that all vertices of the underlying tree, except maybe the root vertex, are of the valency $3$. Also, the relations of the form~\ref{eq:lead2} imply, that $\Pop$-normal binary trees coincide with $\Lie$-normal monomials from Example~\ref{ex::grobner_for_classical_contractads} (if we replace labelling $\lambda_{\Path_2}$ with $b$).
		
		For a normal monomial $T$, we define two monomials $\mathrm{m}(T),\mathrm{b}(T)$ labeled by $m,b$ as follows. Suppose that the root vertex of $T$ is labeled by an element $\lambda_{\Omega}$. The monomial $\mathrm{m}(T)$ is obtained from $T$ by replacing $\lambda_{\Omega}$ with the $\Com$-monomial $m_{\Omega}$, and vertices $\lambda_{\Path_2}$ with $b$. Similarly, we define the monomial $\mathrm{b}(T)$ with vertices $\lambda_{\Path_2}$ replaced with $b$ and vertex $\lambda_{\Gr}$ replaced with monomial $m_{\Gr/e}\circ^{\Gr}_e b$, where $e$ is the edge consisting of minimal vertex and its minimal adjacent. According to the previous observations, the resulting monomials $\mathrm{m}(T),\mathrm{b}(T)$ are $\Gerst$-monomials from Example~\ref{ex::grobner_for_classical_contractads}. Furthermore, we have
		\[
		\mathrm{m}(\mathsf{Norm}_{\Pop}(\Gr))\cup \mathrm{b}(\mathsf{Norm}_{\Pop}(\Gr))=\mathsf{Norm}_{\Gerst}(\Gr), \quad \mathrm{m}(\mathsf{Norm}_{\Pop}(\Gr))\cap \mathrm{b}(\mathsf{Norm}_{\Pop}(\Gr))=\varnothing.
		\] In particular, we have $2|\mathsf{Norm}_{\Pop}(\Gr)|=\dim\Gerst(\Gr)$. Since normal monomials span the contractad, we conclude the desired inequality $\dim \Pop(\Gr)\leq \frac{1}{2}\dim \Gerst(\Gr)$. So, we obtain the desired isomorphism $\Pop\cong \Grav$.
	\end{proof}
	
	Moreover, during the proof, we constructed the quadratic Gr\"obner basis for $\Grav^{\forget}$.
	\begin{sled}\label{cor::gravgrobner}
		The contractad $\Grav$ has a quadratic Gr\"obner basis with respect to the 2-weighted $\grpermlex$-order described above.
	\end{sled}
	
	Since the contractad $\Grav$ is quadratic, we may define its Koszul dual contractad.
	\begin{defi}\label{def:hyper}
		The HyperCommutative contractad $\Hyper$ is the desuspension of the Koszul dual of the gravity contractad:
		\[
		\Hyper:=\Susp^{-1}(\Grav)^{!}.
		\]
	\end{defi}
	Let us give a presentation of the contractad $\Hyper$ by generators and relations.
	\begin{prop}\label{prop::hyperpres}
		The contractad $\Hyper$ is generated by symmetric elements $\nu_{\Gr} \in \Hyper(\Gr)$ of degree $2(|V_\Gamma|-2)$, satisfying the following relations: for each connected graph $\Gr$ and each pair of edges $e,e'\in E_{\Gr}$, we have
		\begin{equation}\label{eq:hyperrel}
			\sum_{G\colon e\subset G} \nu_{\Gr/G}\circ^{\Gr}_G \nu_{\Gr|_G}=\sum_{G\colon e'\subset G} \nu_{\Gr/G}\circ^{\Gr}_G \nu_{\Gr|_G}.
		\end{equation}
	\end{prop}
	\begin{proof}
		Let us denote by $\nu_{\Gr}$ the generator dual to the generator $\lambda_{\Gr}$ of the contractad $\Grav$. According to the general formula for the Koszul dual contractad, we have 
		\[
		\Hyper := \Susp^{-1}(\Susp^{-1}(\Grav^{\text{!`}})^*)=\Susp^{-2}(\Grav^{\text{!`}})^*,
		\]
		so the homological degree of $\nu_{\Gamma}$ is $|\nu_{\Gamma}|=2(|V_\Gamma|-1)-2=2(|V_\Gamma|-2)$. Moreover, since we took the double suspension of $(\Grav^{\text{!`}})^*$, these elements are $\Aut(\Gamma)$-invariant.  One can check by a direct inspection that the relations \eqref{eq:hyperrel}
		\[
		\R_{e,e'}=\sum_{G\colon e\subset G} \nu_{\Gr/G}\circ^{\Gr}_G \nu_{\Gr|_G}-\sum_{G\colon e'\subset G} \nu_{\Gr/G}\circ^{\Gr}_G \nu_{\Gr|_G}.
		\]
		are ortogonal to all relations of $\Grav$. 
		
		Let us verify that we found all relations. Since of orthogonality, for each graph $\Gr$, we have $\dim \R_{\Hyper}(\Gr)=\dim \Grav^{(2)}(\Gr)=\dim \Grav_{2}(\Gr)=|E_{\Gr}|-1$. So, it suffices to find $(|E_{\Gr}|-1)$-linearly independent relations. Consider some ordering $e_1,e_2,\cdots,e_k$ of edges of graph $\Gr$, where $k=|E_{\Gr}|$, and, for $i=1,2,\cdots, (k-1)$, we let $\R_i:=\R_{e_i,e_{i+1}}$. Suppose that, for some coefficients $c_i$, we have $\sum^{k-1}_{i=1} c_i\R_i=0$. Collecting the coefficient of a particular monomial $\nu_{\Gr/e_i}\circ_{e_i}^{\Gr} \nu_{\Gr_T}$, we see that 
		\[
		c_1=0;\quad c_{k-1}=0;\quad  c_i-c_{i+1}=0 \text{ for }i=1,2,\cdots,k-2,
		\] that implies $c_i=0$ for all $i$. So, we found the desired number of linear independent relations.
		
	\end{proof}
	\subsection{Homology of Wonderful contractad} 
	\label{sec::H::wonderful}
	In this subsection, we describe the homology contractad $H_{\bullet}(\bM_{\mathbb{C}})$ of the Wonderful contractad $\bM_{\mathbb{C}}$ in terms of generators and relations. Also, we prove that this contractad is indeed Koszul.

	\begin{lemma}\label{lemma: hypergen}
		The contractad $H_{\bullet}(\bM_{\mathbb{C}})(\Gr)$ is generated by fundamental classes $[\bM_{\mathbb{C}}(\Gr)]$ ranging over all connected graphs $\Gr$. In particular, the contractad  $H_{\bullet}(\bM_{\mathbb{C}})(\Gr)$ admits the weight-grading given by the rule $\mathrm{w}([\bM_{\mathbb{C}}((T))])=|\Ver(T)|$.
	\end{lemma}
	\begin{proof}
		By Proposition~\ref{prop:divisorpres}, the homology of $\bM_{\mathbb{C}}(\Gr)$ are spanned by fundamental classes of intersections of canonical divisors $D_G$. As we have seen on page~\pageref{page::intersectionsofdivisors},  each nonempty intersection of these divisors has the form $\bM_{\mathbb{C}}((T))$, for some stable $\Gr$-admissible tree $T$, and this subvariety $\bM_{\mathbb{C}}((T))\cong \prod_{v\in\Ver(T)}\bM_{\mathbb{C}}(\In(v))$ is a product of graphic wonderful compactifications for smaller graphs.
	\end{proof}

	Next, using the technique of logarithmic forms, we introduce a cocontractad structure on the cohomology of the uncompactified part $\M_{\mathbb{C}}$. Consider the embedding of $\M_{\mathbb{C}}(\Gr/G)\times\M_{\mathbb{C}}(\Gr|_G)$ as a strata of $\bM_{\mathbb{C}}(\Gr)$ via the contractad embedding $\circ^{\Gr}_G\colon \bM_{\mathbb{C}}(\Gr/G)\times\bM_{\mathbb{C}}(\Gr|_G)\to \bM_{\mathbb{C}}(\Gr)$. Since the complement $\partial\bM_{\mathbb{C}}(\Gr)=\bM_{\mathbb{C}}(\Gr)\setminus \M_{\mathbb{C}}(\Gr)$ is a divisor with normal crossing, the theory of logarithmic forms alongside divisors implies the residue map
	\[
	\mathrm{Res}\colon H^{\bullet}(\M_{\mathbb{C}}(\Gr))\to H^{\bullet-1}(\M_{\mathbb{C}}(\Gr/G)\times \M_{\mathbb{C}}(\Gr|_G)).
	\] To make these morphisms homogeneous, we replace cohomological grading with a new one given by the rule: $|\alpha|'=-1-|\alpha|$.
	These morphisms assemble the graphical collection $H^{-1-\bullet}(\M_{\mathbb{C}})$ into a graded cocontractad\footnote{This notation means that we consider $H^{-1-\bullet}(\M_{\mathbb{C}})$ as a graded cocontractad with grading components $H(\M_{\mathbb{C}})_{\bullet}:=H^{-1-\bullet}(\M_{\mathbb{C}})$}. The examination of cocontractad axioms is analogous to ones in~\cite{aras2017}.
	Let us state the main result of this section.
	\begin{theorem}\label{thm:hyperkoszul}
		The contractad $H_{\bullet}(\bM_{\mathbb{C}})$ is Koszul. Moreover, there is a quasi-isomorphism of cocontractads
		\[
		\Susp^{-2} H^{-\bullet-1}(\M_{\mathbb{C}})\to \mathsf{B}(H_{\bullet}(\bM_{\mathbb{C}})).
		\]
	\end{theorem}
	\begin{proof}
		According to \cite{deligne1971theorie}, if a smooth projective complex algebraic variety $M$ is represented as $M=U\sqcup D$ where $D$ is a normal crossing divisor $D$ with components $D_1$,\ldots, $D_N$, one can define the sheaf of logarithmic differential forms $\mathcal{E}_M^\bullet(\log D)$
		\[
		H^\bullet(U)\cong H^\bullet(M,\mathcal{E}_M^\bullet(\log D)),
		\]
		and there is the Deligne spectral sequence   
		\[
		E_1^{-p,q}=H^{-2p+q}(D^p,\epsilon_p)\twoheadrightarrow E_{\infty}^{-p,q}=\mathrm{gr}_W H^{-p+q}(U)
		\]
		where 
		\[
		D^p=\bigsqcup_{i_1<\cdots<i_p} D_{i_1}\cap\cdots D_{i_p},
		\]
		and $\epsilon_p$ is the locally constant line bundle which over the component $D_{i_1}\cap\cdots D_{i_p}$ is the sign representation of the group of permutations of $i_1,\ldots,i_p$ placed in homological degree $p$.
		In our case, if we take $M=\bM_{\mathbb{C}}(\Gr)$ and $U=\M_{\mathbb{C}}(\Gr)$, the components of $D$ are indexed by tubes with at least 2 vertices, and the intersections $D^p$ are indexed by the set $\Tree_{p+1}(\Gamma)\subset \Tree_{\mathrm{st}}(\Gr)$ consisting of $\Gr$-admissible stable trees with $p+1$ internal vertices.   
		\[
		H^{-2p+q}(D^p,\epsilon_p)=\bigoplus\limits_{T\in \Tree_{p+1}(\Gamma)}H^{-2p+q}(\bM_{\mathbb{C}}((T)),\epsilon_p).
		\]
		Since cohomology groups of $\bM_{\mathbb{C}}(\Gr)$ are concentrated in even degrees, we have $E^{-p,q}_1=0$ for odd $q$. Hence we can consider only even $q$. The dimension of the closure of $\bM_{\mathbb{C}}((T))$ is equal to $-2p+2(|V_{\Gr}|-2)$, hence by Poincare duality we have $H^{-2p+q}(\bM_{\mathbb{C}}((T)))\cong H_{2(|V_{\Gr}|-2)-q}(\bM_{\mathbb{C}}((T)))$. The differential $d_1$ of the spectral sequence is the composition
		\[\begin{tikzcd}
			\bigoplus\limits_{T\in \Tree_{p+1}(\Gamma)}H^{-2p+q}(\bM_{\mathbb{C}}((T)),\epsilon_p) \arrow[swap]{d}{\mathrm{PD}} & 
			\bigoplus\limits_{T\in \Tree_{p}(\Gamma)}H^{-2p+2+q}(\bM_{\mathbb{C}}((T)),\epsilon_p)  \\
			\bigoplus\limits_{T\in \Tree_{p+1}(\Gamma)}H_{2(|V_{\Gr}|-2)-q}(\bM_{\mathbb{C}}((T)),\epsilon_p)\arrow{r}{} 
			& \bigoplus\limits_{T\in \Tree_{p}(\Gamma)}H_{2(|V_{\Gr}|-2)-q}(\bM_{\mathbb{C}}((T)),\epsilon_p)\arrow{u}{\mathrm{PD}}
		\end{tikzcd}
		\]
		where the vertical arrows are induced by the Poincar\'e duality, and the horizontal arrow is induced by inclusions of strata. Thus, we see that the first page of the Deligne spectral sequence computes the homology of the bar construction of the contractad $H_{\bullet}(\bM_{\mathbb{C}})$. Note that the appearance of $\epsilon_p$ corresponds to the fact that the bar construction is the free contractad on the shift $s H_{\bullet}(\bM_{\mathbb{C}})$. 
		
		By Lemma~\ref{lemma: hypergen}, the bar construction of $H_{\bullet}(\bM_{\mathbb{C}})$ admits the well-defined syzygy degree grading. Under the above-mentioned identification, the syzygy degree $k$ component of the bar construction of $H_{\bullet}(\bM_{\mathbb{C}})$ is given by the formula 
		\[
		\mathsf{B}^k(H_{\bullet}(\bM_{\mathbb{C}}))(\Gr)\cong \bigoplus_{q-p=k} E^{-p,2q}_1.
		\]
		So, to show that the contractad $H_{\bullet}(\bM_{\mathbb{C}})$ is Koszul it suffices to show that the second page of the Deligne spectral sequence is concentrated on the diagonal $q-2p=0$. This statement follows from the standard Hodge theory argument combined with the following lemma.
		
		\begin{lemma}
			For a graph $\Gr$, the mixed Hodge structure of $H^k(\M_{\mathbb{C}}(\Gr))$ is pure of weight $2k$.
		\end{lemma}
		\begin{proof}
			Thanks to Proposition~\ref{prop::cohomologyConf}, the cohomology ring of $\Conf_{\Gr}(\mathbb{C})$ is generated by the logarithmic differential forms $\omega_e$ ; it follows that the mixed Hodge structure of $H^k(\Conf_{\Gr}(\mathbb{C}))$ is pure of weight $2k$. Thanks to Proposition~\ref{prop::cohomologyM}, the cohomology $H^{\bullet}(\M_{\mathbb{C}}(\Gr))$ is a subring of $H^{\bullet}(\Conf_{\Gr}(\mathbb{C}))$, and the result follows.
		\end{proof}
		\noindent Since the mixed Hodge structure of $H^p(\M_{\mathbb{C}}(\Gr))$ is manifestly pure of weight $2p$, we have 
		\[
		E_2^{-p,q}=
		\begin{cases}
			H^p(\M_{\mathbb{C}}(\Gr)), \text{ if } q=2p,\\
			\qquad  0 \qquad\qquad \text{ otherwise}. 
		\end{cases}
		\]
		In operadic terms, and using the homological degree convention, this means that the cohomology of the bar construction of $H_{\bullet}(\bM_{\mathbb{C}})$ is concentrated in syzygy degree 0. This implies that the contractad $H_{\bullet}(\bM_{\mathbb{C}})$ is Koszul, as required.
		
		Since the Deligne spectral sequence is collapsing on the second page, we conclude the exact sequence
		\[
		0\to H^p(\M_{\mathbb{C}}(\Gr))\to\bigoplus\limits_{T\in \Tree_{p+1}(\Gamma)}H^{0}(\bM_{\mathbb{C}}((T)),\epsilon_p) \to\bigoplus\limits_{T\in \Tree_{p}(\Gamma)}H^{2}(\bM_{\mathbb{C}}((T)),\epsilon_p)\to \ldots
		\] Combining these exact sequences and suitably desuspending for degree reasons we conclude the morphism of dg cocontractads 
		\[
		\Susp^{-2} H^{-\bullet-1}(\M_{\mathbb{C}})\to \mathsf{B}(H_{\bullet}(\bM_{\mathbb{C}})),
		\] that is quasi-isomorphism by above mentioned arguments. 
	\end{proof}
	
	\begin{sled}\label{cor::homologywond_koszul}
		The contractad $H_{\bullet}(\bM_{\mathbb{C}})$ is quadratic and Koszul, with Koszul dual cocontractad
		\[
		H_{\bullet}(\bM_{\mathbb{C}})^{\cokoszul}\cong \Susp^{-2}H^{-1-\bullet}(\M_{\mathbb{C}}).
		\]
	\end{sled}
	
	Let us describe the explicit presentation of contractad $H_{\bullet}(\bM_{\mathbb{C}})$ in terms of generators and relations.
	
	\begin{theorem}\label{thm::homologywond_hyper}
		The homology of the complex wonderful contractad is isomorphic to the hypercommutative contractad:
		\[
		H_{\bullet}(\bM_{\mathbb{C}})\cong \Hyper.
		\]
		Moreover, the Poincare residue cocontractad is isomorphic to the cocontractad dual to gravity:
		\[
		H^{-1-\bullet}(\M_{\mathbb{C}})\cong \Grav^{*}.
		\]  
	\end{theorem}
	\begin{proof}
		By Lemma~\ref{lemma: hypergen}, the contractad $H_{\bullet}(\bM_{\mathbb{C}})$ is generated by fundamental classes $[\bM_{\mathbb{C}}(\Gr)]\in H_{2(|V_{\Gr}|-2)}(\bM(\Gr))$. Thanks to Corollary, the contractad $H_{\bullet}(\bM_{\mathbb{C}})$ is quadratic and hence by grading reasons all non-trivial relations are concentrated in  $H_{\bullet}^{(2)}(\bM_{\mathbb{C}})(\Gr)=H_{2(|V_{\Gr}|-4)}(\bM_{\mathbb{C}}(\Gr))$. Recall from Proposition~\ref{prop:divisorpres}, that, for each triple of graph $\Gr$ and edges $e,e'\in E_{\Gr}$, we have the identity
		\[
		\sum_{G: e\subset G} [D_G]=\sum_{G': e'\subset G'} [D_{G'}]
		\] Since $D_G\cong \bM_{\mathbb{C}}(\Gr/G)\times\bM_{\mathbb{C}}(\Gr|_G)$, the relation above is rewritten in the contractad language as follows
		\begin{equation}\label{eq::relH(M)}
			\sum_{G\colon e\subset G} [\bM_{\mathbb{C}}(\Gr/G)]\circ^{\Gr}_G [\bM_{\mathbb{C}}(\Gr|_G)]=\sum_{G'\colon e'\subset G} [\bM_{\mathbb{C}}(\Gr/G')]\circ^{\Gr}_G [\bM_{\mathbb{C}}(\Gr|_G')].   
		\end{equation} So, by Proposition~\ref{prop::hyperpres}, the correspondence $\nu_{\Gr}\mapsto [\bM_{\mathbb{C}}(\Gr)]$ provides the surjective morphism of contractads $\Hyper\cong H_{\bullet}(\bM_{\mathbb{C}})$.
		
		Since both contractads are quadratic it suffices to determine that their weight two components have the same dimensions. By Proposition~\ref{prop::hyperpres}, $\dim \Hyper^{(2)}(\Gr)=\dim \T^{(2)}(\nu_{\Gr})-\dim \R_{\Hyper}(\Gr)=|Tube_{\geq 2}(\Gr)|-|E_{\Gr}|+1$, where $|Tube_{\geq 2}(\Gr)|$ is the number of proper tubes with at least $2$ vertices. So, the difference on the right hand side is equal to the number of tubes (including $V_{\Gr}$) with at least $3$ vertices. From Poincare duality, we have $H^{(2)}_{\bullet}(\bM_{\mathbb{C}})(\Gr)=H_{2(V_{\Gr}-2)-2}(\bM(\Gr))\cong H^2(\bM_{\mathbb{C}}(\Gr))$ and $h$ classes indexed by tubes with at least $3$ vertices (including tube $V_{\Gr}$) form a basis. So, we deduce the desired isomorphism $\Hyper\cong H_{\bullet}(\bM_{\mathbb{C}})$.
		
		By Definition~\ref{def:hyper} and definition of Koszul dual cocontractad, we have $\Hyper^{\cokoszul}\cong \Susp^{-2}\Grav^*$. Thanks to Theorem~\ref{thm:hyperkoszul}, we have $H_{\bullet}(\bM_{\mathbb{C}})^{\cokoszul}\cong \Susp^{-2}H^{-1-\bullet}(\M_{\mathbb{C}})$. Combining these observations with the isomorphism  $\Hyper\cong H_{\bullet}(\bM_{\mathbb{C}})$, we conclude the isomorphism of cocontractads $H^{-1-\bullet}(\M_{\mathbb{C}})\cong \Grav^*$.
	\end{proof}
	\subsection{Monomial basis of \texorpdfstring{$H_{\bullet}(\bM_{\mathbb{C}})$}{HMC} and Koszulity of Cohomology rings} 
	\label{sec::hycom::monomials}
	
	Thanks to Corollary~\ref{cor::gravgrobner}, the contractad $\Grav$ has a quadratic Gr\"obner basis with respect to 2-weighted $\grpermlex$-order. Recall~\cite{lyskov2023contractads}, that if quadratic shuffle contractad $\Pop$ has a quadratic Gr\"obner basis, then its Koszul dual $\Pop^!$ has a quadratic Gr\"obner basis with respect to inverse order. Hence with respect to the dual order, we have
	\begin{sled}
		The contractad $\Hyper$ admits a quadratic Gr\"obner basis.
	\end{sled}
	With respect to inverse 2-weighted $\grpermlex$-order, the leading terms of quadratic relations~\eqref{eq:hyperrel} of $\Hyper$, have the form
	\begin{equation}
		\nu_{\Gr/e}\circ^{\Gr}_{e} \nu_e\text{, for }e\neq e_{\min},   
	\end{equation} for each ordered graph $(\Gr,<)$, where $e_{\min}=(v_1,v_2)$ is the unique edge consisting of minimal vertex $v_1$ of $(\Gr,<)$ with its minimal adjacent vertex $v_2$. Recall that normal monomials with respect to leading terms of Gr\"obner basis form a basis of the corresponding contractad. So, we obtain the following result.
	\begin{theorem}[Monomial basis for Homology]\label{thm::monbasishyper}
		Let $\Gr$ be a connected graph with a fixed ordering of vertices. The homology groups of the graphic compactification $\bM_{\mathbb{C}}(\Gr)$ have the additive basis consisting of fundamental classes $[\bM_{\mathbb{C}}((T))]$, labeled by stable $\Gr$-admissible trees $T$, satisfying the following condition: For each 2-subtree $T'\subset T$ with a top internal vertex of valency $3$
		\[  
		T'=\vcenter
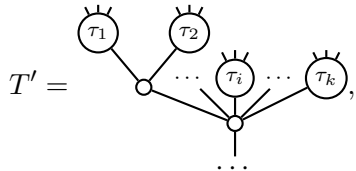
{\hbox{\begin{tikzpicture}[
					scale=0.6,
					vert/.style={inner sep=2pt, circle,  draw, thick},
					leaf/.style={inner sep=2pt,rectangle},
					edge/.style={-,black!30!black, thick},
					]
					\node[vert] (1l) at (-3,2){\scriptsize$\tau_1$};
					\node[leaf] (1ll) at (-3.35,2.7) {\scriptsize\space};
					\node[leaf] (1lm) at (-3,2.75) {\scriptsize\space};
					\node[leaf] (1lr) at (-2.65,2.7) {\scriptsize\space};
					\draw[edge] (1l)--(1ll);
					\draw[edge] (1l)--(1lm);
					\draw[edge] (1l)--(1lr);
					\node[vert] (2l) at (-1,2){\scriptsize$\tau_2$};
					\node[leaf] (2ll) at (-1.35,2.7) {\scriptsize\space};
					\node[leaf] (2lm) at (-1,2.75) {\scriptsize\space};
					\node[leaf] (2lr) at (-0.65,2.7) {\scriptsize\space};
					\draw[edge] (2l)--(2ll);
					\draw[edge] (2l)--(2lm);
					\draw[edge] (2l)--(2lr);
					\node[vert] (4l) at (0,1) {\scriptsize$\tau_i$};
					\node[leaf] (4ll) at (-0.35,1.7) {\scriptsize\space};
					\node[leaf] (4lm) at (0,1.75) {\scriptsize\space};
					\node[leaf] (4lr) at (0.35,1.7) {\scriptsize\space};
					\draw[edge] (4l)--(4ll);
					\draw[edge] (4l)--(4lm);
					\draw[edge] (4l)--(4lr);
					\node[vert] (kl) at (2,1){\scriptsize$\tau_k$};
					\node[leaf] (kll) at (1.65,1.7) {\scriptsize\space};
					\node[leaf] (klm) at (2,1.75) {\scriptsize\space};
					\node[leaf] (klr) at (2.35,1.7) {\scriptsize\space};
					\draw[edge] (kl)--(kll);
					\draw[edge] (kl)--(klm);
					\draw[edge] (kl)--(klr);
					
					\node[vert] (root) at (0,0) {\space};
					\node[vert] (bi) at (-2,0.8) {\space};
					\node[leaf] (3l) at (-1,1) {\scriptsize$\cdots$};
					
					\node[leaf] (5l) at (1,1) {\scriptsize$\cdots$};
					\node[leaf] (dotsatroot) at (0,-1) {$\cdots$};
					\draw[edge] (dotsatroot)--(root);
					\draw[edge] (root)--(bi);
					\draw[edge] (root)--(3l);
					\draw[edge] (root)--(4l);
					\draw[edge] (root)--(5l);
					\draw[edge] (root)--(kl);
					\draw[edge] (bi)--(1l);
					\draw[edge] (bi)--(2l);
		\end{tikzpicture}}}, 
		\] with $\tau_1$,$\tau_2$ subtrees outgoing from the top vertex, and $\tau_3,\cdots,\tau_k$ the other outgoing subtrees, we have:
		\begin{itemize}
			\item[(i)] For $i>1$, $\min L(\tau_i)>\min L(\tau_1)$,
			\item[(ii)] For $i>2$, if union $L(\tau_i)\cup L(\tau_1)\subset V_{\Gr}$ is a tube, then $\min L(\tau_i)>\min L(\tau_2)$,
		\end{itemize} where $L(\tau_i)$ is the leaf set of $\tau_i$, and $\min L(\tau_i)$ is a minimal leaf with respect to vertex-ordering.   
	\end{theorem}
	
	Let us explain the motivation of this basis. In~\cite{dotsenko2022homotopy}, Dotsenko showed that the cohomology ring $H^{\bullet}(\beM_{0,n+1})$ in $h$-presentations admits a quadratic Gr\"oebner basis, which implies that the cohomology ring $H^{\bullet}(\beM_{0,n+1})$ is Koszul. The main idea was the construction a one-to-one correspondence between normal algebraic monomials in $H^{\bullet}(\beM_{0,n+1})$ and normal operadic monomials in the Hypercommutative operad $\mathsf{Hyper}(n)$. In~\cite{coron2023supersolvability},  Coron generalized Dotsenko's result to a particular type of building sets. In particular, his results imply the Koszulity of cohomology rings $H^{\bullet}(\bM(\Gr))$ in the case $\Gr$ is a chordal graph.  We suspect that these results could be generalized to the case of arbitrary connected graph $\Gr$ using the contractad monomial basis of $\Hyper$ from Theorem~\ref{thm::monbasishyper}.
	\begin{conjecture}
		For a connected graph $\Gr$, the cohomology ring $H^{\bullet}(\bM(\Gr))$ is Koszul.
	\end{conjecture}
	\subsection{Hilbert series}
	In this subsection, we compute the Hilbert series of wonderful compactifications $\bM_{\mathbb{C}}(\Gr)$ using the machinery of graphic functions. 
	
	We consider the graphic function $\chi_q(\bM_{\mathbb{C}})$ that computes Hilbert series of the components $\bM(\Gr)$ of wonderful contractad
	\begin{gather*}
		\chi_q(\bM_{\mathbb{C}})(\Gr)=\sum_{i=0} \dim H^{2i}(\bM_{\mathbb{C}}(\Gr))q^i
	\end{gather*} Recall that the graphic compactification $\bM_{\mathbb{C}}(\Gr)$ has no non-trivial odd cohomology groups, hence the polynomial $\chi_q(\bM_{\mathbb{C}})(\Gr)$ completely determines Betti numbers of $\bM_{\mathbb{C}}(\Gr)$.
	\begin{lemma}\label{lemma::hilbertwond}
		We have
		\[
		\chi^{\mathrm{w}}_q(\Hyper)=q\chi_q(\bM_{\mathbb{C}})+(1-q)\epsilon.
		\]
	\end{lemma}
	\begin{proof}
		We have the following connection between weight and homological grading
		\[
		\Hyper^{(r)}(\Gr)=\Hyper(\Gr)_{2(|V_{\Gr}|-r-1)}\cong H_{2(|V_{\Gr}|-r-1)}(\bM_{\mathbb{C}}(\Gr))\overset{\mathsf{P.D.}}{\cong} H^{2(r-1)}(\bM_{\mathbb{C}}(\Gr)),
		\] for $\Gr\neq\Path_1$. So, for $\Gr\neq\Path_1$, we have
		\[
		\chi^{\mathrm{w}}_q(\Hyper)(\Gr)=\sum^{|V_{\Gr}|-1}_{r\geq 1} \dim H^{2(r-1)}(\bM_{\mathbb{C}}(\Gr)) q^{r}=q\chi_q(\bM_{\mathbb{C}})(\Gr)
		\] From which we deduce the identity $\chi^{\mathrm{w}}_q(\Hyper)=q\chi_q(\bM_{\mathbb{C}})+(1-q)\epsilon$.
	\end{proof}
	
	Let us state the useful formula
	\begin{lemma}\label{lemma::hilbertgrav}
		We have
		\[
		\chi^{\mathrm{w}}_{-q}(\Grav)=\frac{q}{q-1}\chi_{-q}(\Gerst)-\frac{1}{q-1}\epsilon.
		\]
	\end{lemma}
	\begin{proof}
		Since each generator of $\Grav$ has homological degree equal to $1$, the homological and weight gradings of this contractad coincides. So, $\chi^{\mathrm{w}}_{-q}(\Grav)=\chi_{-q}(\Grav)$. Since the dg contractad $(\Gerst,\Delta)$ splits exact,  we have an isomorphism of homologically graded graphical collections
		\[
		\Gerst_{\bullet}\cong \overline{\Grav}_{\bullet}\oplus s^{-1}\overline{\Grav}_{\bullet}\oplus \mathbb{1},
		\] from which we obtain the desired formula.
	\end{proof}
	Now we are ready to state one of the main results of this paper.
	\begin{theorem}[Recurrence for $\bM_{\mathbb{C}}$]\label{thm::rec_for_bM}
		We have
		\begin{equation}\label{eq::rec_for_bM1}
			\chi_q(\bM_{\mathbb{C}})*\left(\frac{q}{q-1}\epsilon-\frac{1}{q(q-1)}\mathfrak{X}(q)\right)=\epsilon,   
		\end{equation} where $\mathfrak{X}(q)$ is the chromatic graphic function (see Definition~\ref{def::chromgrfun}). Equivalently, we have
		\begin{equation}\label{eq::rec_for_bM2}
			\chi_q(\bM_{\mathbb{C}})*\left(\frac{q}{q-1}\mathbb{1}-\frac{1}{q-1}\mathbb{1}_q\right)=\mathbb{1}.   
		\end{equation}
		Explicitly, we have the reccurence of the form
		\begin{equation}
			\sum_{I\vdash\Gr}\chi_q(\bM_{\mathbb{C}})(\Gr/I)\prod_{G\in I}\frac{q-q^{|G|-1}}{q-1}=1.  
		\end{equation}
	\end{theorem}
	\begin{proof}
		By the construction, the contractad $\Hyper$ is Koszul dual to the contractad $\Susp\Grav$. Thanks to Theorem~\ref{thm::homologywond_hyper} and Corollary~\ref{cor::homologywond_koszul}, the contractad $\Hyper$ is Koszul, hence, by Theorem~\ref{thm::hilbertkoszul}, we have
		\begin{equation}\label{eq::hilbert_hyper_grav}
			\chi^{\mathrm{w}}_q(\Hyper)*\chi^{\mathrm{w}}_{-q}(\Susp\Grav)=\epsilon.   
		\end{equation}
		Since suspension $\Susp$ does not affect on the weight-grading, we have $\chi^{\mathrm{w}}_{-q}(\Susp\Grav)=\chi_{-q}^{\mathrm{w}}(\Grav)$. So, thanks to Lemma~\ref{lemma::hilbertwond} and Lemma~\ref{lemma::hilbertwond}, the identity~\eqref{eq::hilbert_hyper_grav} is rewritten in the form
		\[
		\left(q\chi_q(\bM_{\mathbb{C}})+(1-q)\epsilon\right)*\left(\frac{q}{q-1}\chi_{-q}(\Gerst)-\frac{1}{q-1}\epsilon\right)=\epsilon.
		\] Since the product $*$ is linear on the left side, we have
		\begin{multline*}
			\left(q\chi_q(\bM_{\mathbb{C}})+(1-q)\epsilon\right)*\left(\frac{q}{q-1}\chi_{-q}(\Gerst)-\frac{1}{q-1}\epsilon\right)=
			\\ 
			= q\chi_q(\bM_{\mathbb{C}})*\left(\frac{q}{q-1}\chi_{-q}(\Gerst)-\frac{1}{q-1}\epsilon\right)-q\chi_{-q}(\Gerst)+\epsilon.
		\end{multline*}
		So, we have
		\begin{equation}\label{eq::hilbert_hyper_grav2}
			\chi_q(\bM_{\mathbb{C}})*\left(\frac{q}{q-1}\chi_{-q}(\Gerst)-\frac{1}{q-1}\epsilon\right)=\chi_{-q}(\Gerst).   
		\end{equation}
		By equation~\eqref{eq::hilbertgerst}, we have $\chi_{-q}(\Gerst)=\mathbb{1}*(\mathbb{1}_q\cdot\mu)$, hence $\chi_{-q}(\Gerst)*\mathbb{1}_q=\mathbb{1}$. So, by taking the product $(-)*\mathbb{1}_q$ on both sides of identity~\eqref{eq::hilbert_hyper_grav2}, we obtain the desired formula~\eqref{eq::rec_for_bM2}
		\[
		\chi_q(\bM_{\mathbb{C}})*\left(\frac{q}{q-1}\mathbb{1}-\frac{1}{q-1}\mathbb{1}_q\right)=\mathbb{1}.
		\] Thanks to Lemma~\ref{lemma::formula_for_chrom}, we have $\mathfrak{X}(q)=q\cdot(\mathbb{1}_q*\mu)$. Hence the identity~\eqref{eq::rec_for_bM1} is obtained from identity~\eqref{eq::rec_for_bM2} by taking the product $(-)*\mu$ on both sides.
	\end{proof}
	
	\begin{example}
		Let us determine the Hilbert series of components of the wonderful contractad $\bM_{\mathbb{C}}$ for one-parameter families of graphs. For simplicity of notations, we let $\psi:=\frac{q}{q-1}\epsilon-\frac{1}{q(q-1)}\mathfrak{X}(q)$. 
		\begin{itemize}
			\item[(i)] For the paths, thanks to Proposition~\ref{prop::one_parameter_comp} and Theorem~\ref{thm::rec_for_bM}, we have $F_{\Path}(\psi)(F_{\Path}(\chi_q(\bM_{\mathbb{C}}))(t))=t$. By direct computations, we conclude
			\begin{equation}
				F_{\Path}(\bM_{\mathbb{C}})=t+\sum_{n\geq 2} \left[\sum^{n-2}_{i=0} \dim H^{2i}(\bM_{\mathbb{C}}(\Path_n))q^i\right]t^n=\frac{1-(1-q)t-\sqrt{(1+(1-q)t)^2-4t}}{2q},  
			\end{equation}
			what recovers the result of~\cite[Th.~4.4.1]{dotsenko2019toric}. In particular, we get that Betti numbers of $\bM_{\mathbb{C}}(\Path_n)$ are given by Narayana numbers $\dim H^{2k}(\bM_{\mathbb{C}}(\Path_n))=\mathsf{N}(n,k+1)=\frac{1}{n-1}\binom{n-1}{k}\binom{n-1}{k+1}$, and the Euler characteristics of $\bM_{\mathbb{C}}(\Path_n)$ are given by Catalan numbers $\chi(\bM_{\mathbb{C}}(\Path_n))=\frac{1}{n}\binom{2(n-1)}{n-1}$.
			\item[(ii)] For Cycle graphs, thanks to Proposition~\ref{prop::one_parameter_comp}, we have
			\begin{multline}
				F_{\Cyc}(\bM_{\mathbb{C}})=t+\sum_{n\geq 2} \left[\sum^{n-2}_{i=0}\dim H^{2i}(\bM_{\mathbb{C}}(\Cyc_n))q^i\right]\frac{t^n}{n}=t+F_{\Path}(\bM_{\mathbb{C}})-F_{\Cyc}(\psi)(F_{\Path}(\bM_{\mathbb{C}}))=\\=t-\frac{1}{q(q-1)}\log\left(1-(q-1)F_{\Path}(\bM_{\mathbb{C}}))\right)
				-\frac{1}{q}\log\left(1+F_{\Path}(\bM_{\mathbb{C}}))\right).
			\end{multline}
			\item[(iii)] For star graphs $\St_n=\K_{(1,n)}$, thanks to $F_{\St}(f*g)=F_{\St}(f)\cdot F_{\St}(g)$, we have
			\begin{equation}
				F_{\St}(\bM_{\mathbb{C}})=1+\sum_{n\geq 1} \left[\sum^{n-1}_{i=0}\dim H^{2i}(\overline{\EuScript{L}}_{0,n})q^i\right]\frac{t^n}{n!}=\frac{1}{F_{\St}(\psi)}=\frac{q-1}{q-e^{(q-1)t}},   
			\end{equation}
			that recovers the result of Losev and Manin~\cite[Th.~2.3]{losev2000new}. In particular, the Euler characteristic of the moduli space $\overline{\EuScript{L}}_{0,n}$ is $\chi(\overline{\EuScript{L}}_{0,n})=n!$.
			\item[(iv)] For complete graphs, thanks to formula $F_{\K}(f*g)(t)=F_{\K}(f)(F_{\K}(g)(t))$, the generating series
			\[
			F_{\mathsf{K}}(\beM)(q,t)=t+\sum_{n\geq 2} \left[\sum^{n-2}_{i=0}\dim H^{2i}(\beM_{0,n+1})q^i\right]\frac{t^n}{n!}
			\] is inverse (with respect to variable $t$) to the function
			\[
			g(q,t)=F_{\K}(\psi)=\frac{1}{q-1}[qt-\frac{1}{q}((t+1)^q-1)].
			\] This functional identity recovers the result of Getzler~\cite[Th.~5.9]{getzler1995operads}
		\end{itemize}
	\end{example}
	Recall from Theorem~\ref{thm::modular=wonderful}, that for complete multipartite graph $\K_{\lambda}$, the graphic compactification $\bM_{\mathbb{C}}(\K_{\lambda})$ is isomorphic to the moduli space $\beM_{0,\K_{\lambda}}(\mathbb{C})$ of complex $\K_{\lambda}$-pointed stable curves. Let us state the main result of this section
	\begin{theorem}[Functional equation for $\beM(\mathbb{C})$]\label{thm::hilbert_complex_modular}
		The generating series
		\begin{multline}\label{eq::modularhilbert}
			F_{\mathsf{Y}}(\beM(\mathbb{C}))(z)=\sum_{l(\lambda)\geq 2} \left[\sum^{|\lambda|-2}_{i=0}\dim H^{2i}(\beM_{0,\K_{\lambda}}(\mathbb{C}))q^i\right]\frac{m_{\lambda}}{\lambda!}+
			\\
			+ \sum_{n\geq 1, |\lambda|\geq 0} \left[\sum^{|\lambda|+n-2}_{i=0}\dim H^{2i}(\beM_{0,\K_{(1^n)\cup \lambda}}(\mathbb{C}))q^i\right]\frac{m_{\lambda}}{\lambda!}\frac{z^n}{n!}, 
		\end{multline} 
		is inverse (with respect to composition in variable $z$) to the symmetric function
		\begin{equation}\label{eq::inversemodularhilbert}
			G(z)=\frac{q}{q-1}z-\frac{1}{q(q-1)}\left[\left(1+z+\sum_{n\geq 1}\frac{p_n}{n!}\right)^q-1-\sum_{n\geq 1}\frac{p_nq^n}{n!}\right].   
		\end{equation}   In other word, we have the functional equation: $G(F_{\mathsf{Y}}(\beM)(z))=F_{\mathsf{Y}}(\beM)(G(z))=z$.
	\end{theorem}
	\begin{proof}
		Thanks to equation~\eqref{eq::chrommultipartitefull}, we see that the Young generating series for $\psi=\frac{q}{q-1}\epsilon-\frac{1}{q(q-1)}\mathfrak{X}(q)$ is exactly the series $G(z)$.
	\end{proof}
	\begin{example}
		If we specialize $x_1=t,x_2=0,x_3=0,\cdots$, then the generating series~\eqref{eq::modularhilbert} has the form
		\[
		F_q(t,z)=F_{\mathsf{Y}}(\chi_q(\beM))(z)|_{x_1=t,x_2=0,x_3=0,\cdots}=\sum_{n\geq 1,m\geq 0} \left[\sum^{n+m-2}_{i=0}\dim H^{2i}(\beM_{0,(1^n,m)})q^i\right]\frac{t^m}{m!}\frac{z^n}{n!}
		\] Recall that the moduli space $\beM_{0,\K_{(1^n,m)}}$ is exactly the Losev-Manin moduli space $\overline{\EuScript{L}}_{0;n,m}$ of stable curves with $n$ white and $m$ black points. In a similar way, for the generating series~\eqref{eq::inversemodularhilbert}, we have
		\[
		G_q(t,z)=G(z)|_{x_1=t,x_2=0,x_3=0,\cdots}=\frac{q}{q-1}z-\frac{1}{q(q-1)}\left[(z+e^t)^q-e^{qt}\right].
		\] Hence the generating function $F_q(t,z)$ is inverse to the function $G_q(t,z)$ (with respect to variable $z$). So, by Lagrange inversion theorem, for $n\geq 1$, we have
		\begin{sled}
			For $n\geq 1$, we have
			\begin{equation}
				\sum_{m\geq 0} \left[\sum^{n+m-2}_{i=0}\dim H^{2i}(\overline{\EuScript{L}}_{0;n,m})q^i\right]\frac{t^m}{m!}=(q-1)^{n-1}\lim_{w\to 0} \frac{d^{n-1}}{dw^{n-1}}\left[\frac{w}{qw-\frac{1}{q}((w+e^t)^q-e^{qt})}\right]^n.  
			\end{equation}
		\end{sled}
	\end{example}
	\section{(Co)Homology of the Real Wonderful contractad \texorpdfstring{$\bM_{\mathbb{R}}$}{MR}}
	\label{sec::H::wonder::R}
	Let us note that the De Concini--Procesi "wonderful compactification" of the complement of arrangement is defined in combinatorial terms using the linear algebra data. In particular, if the arrangement is defined over a field $\Bbbk$, then the "wonderful compactification" is a smooth algebraic variety defined over the same field $\Bbbk.$ In particular, the topological spaces $\bM(\Gr)$ of the wonderful contractad are smooth algebraic varieties defined over rational numbers (even over integer numbers).
	Therefore, the real locus $\bM_{\mathbb{R}}(\Gr)$ are smooth manifolds that assemble a contractad in the category of topological spaces. The corresponding operad is called the \emph{mosaic operad} after Devadoss (\cite{devadoss1999tessellations}, its homology was computed in~\cite{etingof2010cohomology} and the rational homotopy type was described in~\cite{khoroshkin2019real}.
	
	In this section, we describe the structure of the homology contractad and the corresponding graphical collection generalizing the one known for operads. As a corollary, we obtain a description of the generating series for the rational Betti numbers associated with the graphs we are dealing with.
	Morally, all results of this section about $\bM_{\mathbb{R}}(\Gr)$ are straightforward generalizations of the known one obtained for the operads in~\cite{etingof2010cohomology,khoroshkin2019real}.

	\subsection{Cell structure of the real locus} 
	\label{sec::cell::M::R}
	
	Let us provide a bit more details of the little intervals contractad $\D_1$ discussed already in \S\ref{sec:Ass}. Recall that each component of this contractad is a disjoint union of contractible spaces. Consider the associative contractad $\Ass=H_{0}(\D_1)$ encoding connected components of $\D_1$.  Combinatorially, the component $\Ass(\Gr)$ consists of vertex-orderings $(v_1,\cdots,v_n)$ modulo the equivalence $(\cdots,v_i,v_{i+1},\cdots)\sim (\cdots,v_{i+1},v_i,\cdots)$, for $v_i$ and $v_{i+1}$ are non-adjacent in $\Gr$. The contractad $\D_1$ admits an outer involution $\tau^{!}$ that changes the orientation of the unit interval. On the level of homology, the automorphism $\tau^{!}$ swaps the order in the tuples in $\Ass$ 
	$$\tau^{!}\colon (v_1,\cdots,v_n) \mapsto (v_n,\cdots,v_1).$$ 
	
	Recall that the centering map $r\colon \D_1(\Gr)\to \Conf_{\Gr}(\mathbb{R})$ provides a component wise homotopy equivalence. Consider the graphic configuration space of a real line $\Conf_{\Gr}(\mathbb{R})$ with the corresponding action of affine transformations $\Aff^{1}(\mathbb{R})=\mathbb{R}^+\rtimes \mathbb{R}^{\times}$. Similarly to Lemma~\ref{lemma::confandM}, this action is free and its quotient is isomorphic to the open part of the graphic compactification
	\[
	\M_{\mathbb{R}}(\Gr)\cong \Conf_{\Gr}(\mathbb{R})/\Aff^1(\mathbb{R}).
	\] Let us describe this quotient in more detail. For a configuration $\mathbf{x}=(x_v)_{v\in V_{\Gr}}\in \Conf_{\Gr}(\mathbb{R})$, we define its barycenter and radius by the rule
	\[
	b(\mathbf{x})=\frac{1}{|V_{\Gr}|}\sum_{v\in V_{\Gr}} x_v, \quad r(\mathbf{x}):=\max_{v\in V_{\Gr}}( |x_v-b(\mathbf{x})|).
	\] Let $\NConf_{\Gr}(\mathbb{R})$ be the subspace  of normalised configurations with $b(\mathbf{x})=0$ and $r(\mathbf{x})=1$. We consider the renormalisation map
	\[
	p\colon \Conf_{\Gr}(\mathbb{R})\rightarrow \NConf_{\Gr}(\mathbb{R}),\quad (x_v)_{v\in V_{\Gr}}\mapsto \frac{1}{r(\mathbf{x})}(x_v-b(\mathbf{x})).
	\] This map is a deformation retraction and a principal $\Aff^{+}_1(\mathbb{R})$-fibration, where $\Aff^{+}_1(\mathbb{R})=\mathbb{R}^+\rtimes \mathbb{R}^{>0}$ is a subgroup of orientation-preserving affine transformations. In particular, we have
	\[
	\NConf_{\Gr}(\mathbb{R})\cong \Conf_{\Gr}(\mathbb{R})/\Aff^+_1(\mathbb{R}).
	\]
	The induced action of $\Aff_1(\mathbb{R})/\Aff^+_1(\mathbb{R})\cong \mathbb{Z}_2$ of the cyclic group on normalised configurations $\NConf_{\Gr}(\mathbb{R})$ is given by reflection $\mathbf{x}\mapsto -\mathbf{x}$. Combining both isomorphisms, we have
	\[
	\M_{\mathbb{R}}(\Gr)\cong \Conf_{\Gr}(\mathbb{R})/\Aff_1(\mathbb{R})\cong \NConf_{\Gr}(\mathbb{R})/\mathbb{Z}_2.
	\]
	Recall that the configuration space $\Conf_{\Gr}(\mathbb{R})$ is a disjoint union of cones carved by diagonal hyperplanes $H_e=\{x_v=x_w\}$ indexed by edges $e\in E_{\Gr}$. Also, the connected components $\pi_0(\Conf_{\Gr}(\mathbb{R}))$ are indexed by elements of the contractad $\Ass(\Gr)$. So, $\NConf_{\Gr}(\mathbb{R})$ is homeomorphic to a disjoint union of $(|V_{\Gr}|-2)$-dimensional open balls, one for each element $(v_1,v_2,\cdots,v_n)\in \Ass(\Gr)$.
	
	The reflection $\mathbf{x}\mapsto -\mathbf{x}$ produces the involution $\tau^!$ of the Associative contractad $\Ass$ given by relabeling $(v_1,v_2,\cdots,v_n)\mapsto (v_n,\cdots,v_2,v_1)$. This involution glue together the corresponding open balls in $\NConf_{\Gr}(\mathbb{R})$. So, the space $\M_{\mathbb{R}}(\Gr)$ is homeomorphic to a disjoint union of balls indexed by $\tau^{!}$-coinvariants $\Ass(\Gr)_{\tau^!}$.
	
	Recall from~\eqref{eq::stratification} that the graphic compactification  $\bM_{\mathbb{R}}(\Gr)$ admits a locally open stratification $\M((T))$ indexed by stable $\Gr$-admissible trees $T$, where $\M((T))=\prod_{v\in \Ver(T)}\M_{\mathbb{R}}(\In(v))$ is a product of uncompactified parts for smaller graphs. So, $\bM_{\mathbb{R}}(\Gr)$ admits a stratification by open balls that gives the cell decomposition of $\bM_{\mathbb{R}}(\Gamma)$.
	
	Recall that the associative contractad $\Ass$ is self Koszul dual. The Koszul dual involution $\tau$ differs from $\tau^{!}$ just with the sign:
	$$\tau\colon (v_1,\cdots,v_n) \mapsto (-1)^{n-1}(v_n,\cdots,v_1).$$ This sign is important in the description of cell complexes. The observations of this subsection imply the following result
	
	\begin{theorem}\label{thm::cell_structure}
		The connected components of the strata elements $\M_{\mathbb{R}}((T)), T\in \Tree_{\mathrm{st}}(\Gr),$ of the graphic compactification $\bM_{\mathbb{R}}(\Gamma)$ form a cell decomposition compatible with the contractad structure.
		The corresponding contractad $C^{\mathrm{cell}}_{\bullet}(\bM_{\mathbb{R}})$ of cell complexes is the free contractad that is a cobar construction of the $\tau$-coinvariants of the associative cocontractad $\Ass^*$
		\begin{equation}
			\label{eq::cell::mosaic}
			C^{\mathrm{cell}}_{\bullet}(\bM_{\mathbb{R}})\cong \Cobar_{\bullet}({\Susp^{-1}}[\Ass^{\tau}]^*).
		\end{equation}
	\end{theorem}
	\begin{proof}
		The proof is a straightforward generalization of the one given for operads in~\cite{khoroshkin2019real}.
	\end{proof}
	We will describe the homology contractad in \S\ref{sec::realhomology}. 
	
	\subsection{Invariant and odd Poisson contractads}
	\label{sec::Pois::inv}
	
	The involution~$\tau$ on the contractad $\Ass$ is compatible with the standard filtration $\calF$ described in Proposition~\ref{prop::Ass::Pois} because it preserves the Lie bracket. Moreover, for the associated graded contractad $\mathrm{gr}_{\calF}\Ass\cong \Pois$ the involution $\tau$ is easily defined on the generators:
	$$
	{\text{ In }\Pois(\Path_2)} \quad \tau(b) = b \ \&\  \tau(m)=-m.
	$$
	
	Consider the Poisson contractad $\Pois$ with involution $\tau(m)=-m, \tau(b)=b$. Consider the commutative and Lie subcontractads $\Com,\Lie\subset \Pois$.
	\begin{lemma}
		\label{lem::Pois_odd::collection}
		On the level of graphical collections, we have
		\[
		\Pois^{\tau}\cong \Com^{\tau}\circ\Lie.
		\]
	\end{lemma}
	\begin{proof} Recall from Example~\ref{ex::distib::law}, that the Poisson contractad is obtained from the commutative and Lie contractads by distributive law, in particular, we have $\Pois\cong \Com\circ\Lie$. For a monomial $\alpha=(m_{\Gr/I};b^{(G_1)},b^{(G_2)},\cdots,b^{(G_k)})$ which is obtained by iterated multiplication $m_{\Gr/I}$ and iterated composition of $b$, we have $\tau(\alpha)=(-1)^{|I|-1}\alpha$. So, we obtain the desired decomposition.
	\end{proof}
	Let us describe the presentation of the contractad $\Pois^{\tau}$.
	\begin{theorem}[$\tau$-invariants of Commutative and Poisson contractads]\label{thm::invariant_com_and_pois}\hfill\break \begin{itemize}
			\item[(i)] The contractad $\Com^{\tau}$ of $\tau$-invariants for $\Com$ is the quadratic Koszul contractad generated by symmetric ternary operation $\mu(x_1,x_2,x_3):=m(m(x_1,x_2),x_3)$ in the components $\Path_3$ and $\K_3$ respectively, satisfying the relations: For each graph $\Gr$ on $5$ vertices $\{1,2,3,4,5\}$, we have
			\begin{equation}\label{eq::comodd_rel}
				\mu(x_{i_1},x_{i_2},\mu(x_{i_3},x_{i_4},x_{i_5}))=\mu(x_{j_1},x_{j_2},\mu(x_{j_3},x_{j_4},x_{j_5})),    
			\end{equation} where $\{i_3,i_4,i_5\}$ and $\{j_3,j_4,j_5\}$ are tubes of $\Gr$. 
			\item[(ii)]The contractad $\Pois^{\tau}$ of $\tau$-invariants for $\Pois$ is the quadratic Koszul contractad obtained from contractads $\Com^{\tau}$ and $\Lie$ by the following rewriting rule: For each pair $(\Gr,G)$, where $\Gr$ is a graph on $4$ vertices $\{1,2,3,4\}$ and 3-tube $G=\{1,2,3\}$, we have the following analogue of Leibnitz identity:
			\begin{gather}\label{eq::rewr_rule_for_inariant}
				b(\mu(x_1,x_2,x_3),x_4)=\epsilon_{3,4}\mu(x_1,x_2,b(x_3,x_4))+\epsilon_{2,4}\mu(x_1,b(x_2,x_4),x_3)+\epsilon_{1,4}\mu(b(x_1,x_4),x_2,x_3),
			\end{gather} where $\epsilon_{v,w}=1$ for $v,w$ are adjacent and zero otherwise.
		\end{itemize}
	\end{theorem}
	\begin{proof}
		(i) For a generator $m_{\Gr}\in \Com(\Gr)$, we have $\tau(m_{\Gr})=(-1)^{|V_{\Gr}|-1}m_{\Gr}$, so $m_{\Gr}$ is invariant if and only if $\Gr$ has an odd number of vertices. Hence the contractad $\Com^{\tau}$ has non-trivial one-dimensional components only for graphs with an odd number of vertices. Therefore this contractad is generated by ternary operations $\mu_{\Gr}(x_1,x_2,x_3):=m(m(x_1,x_2),x_3)$ for $\Gr=\Path_3,\K_3$ and relations~\eqref{eq::comodd_rel} are obvious.
		Let $\Pop$ be a contractad generated by $\mu_{\Path_3}$ and $\mu_{\K_3}$ subject to relations~\eqref{eq::comodd_rel} and it remains to verify the isomorphism of vector spaces $\Pop(\Gr)\cong \Com^{\tau}(\Gr)$ for all graphs.  
		We show that $\Pop(\Gr)$ is one-dimensional for a graph $\Gamma$ with an odd number of vertices using Gr\"obner bases for contractads from \S\ref{sec::Grobner::contractad}. Consider the shuffle version $\Pop^{\forget}$ and the monomial order reverse to $\grpermlex$-order. The leading terms of relations~\eqref{eq::comodd_rel} with respect to this order are
		\[
		\mu_{\Gr/G}\circ^{\Gr}_{G} \mu_{\Gr|_G}, G\neq G_{\min},
		\] where $G_{\min}=\{v_1,v_2,v_3\}$ is a unique $3$-tube of an ordered graph with the following property: $v_1$ is a minimal vertex of graph $\Gr$, $v_2$ is a minimal adjacent to $v_1$, $v_3$ is a minimal vertex such that the union $\{v_1,v_2,v_3\}$ is a tube. Similarly to the monomial basis for $\Com^{\forget}$ from Example~\ref{sec::Grobner::contractad}, for each odd graph $\Gr$, there is a unique normal monomial $\mu_{\Gr}$ defined recursively by the rule
		\[
		\mu_{\Gr}=\mu_{\Gr/G_{\min}}\circ \mu_{\Gr|_G},
		\] where $G_{\min}$ is a unique minimal $3$-tube. Since normal monomials linearly span a contractad, we conclude $\Pop\cong \Com^{\tau}$ as required. Moreover, since normal monomials form a basis, this implies that quadratic relations form a Gr\"obner basis, so $\Com^{\tau}$ is Koszul.
		
		(ii) Thanks to Lemma~\ref{lem::Pois_odd::collection}, the contractad $\Pois^{\tau}$ is generated by subcontractads $\Com^{\tau}$ and $\Lie$. Next, let us verify the rewriting rule~\eqref{eq::rewr_rule_for_inariant}
		\begin{multline*}
			b(\mu(x_1,x_2,x_3),x_4)=b(m(m(x_1,x_2),x_3),x_4) = \\
			=\epsilon_{3,4}m(m(x_1,x_2),b(x_3,x_4))+\epsilon_{\{1,2\},4}m(x_3,b(m(x_1,x_2),x_4))=\\=\epsilon_{3,4}\mu(x_1,x_2,b(x_3,x_4))+\epsilon_{\{1,2\},4}[\epsilon_{2,4}m(x_3,m(x_1,b(x_2,x_4)))+\epsilon_{1,4}m(x_3,m(x_2,b(x_1,x_4)))]=\\=\epsilon_{3,4}\mu(x_1,x_2,b(x_3,x_4))+\epsilon_{2,4}\mu(x_1,b(x_2,x_4),x_3)+\epsilon_{1,4}\mu(b(x_1,x_4),x_2,x_3).
		\end{multline*} So, we have a surjective morphism of contractads $\Com^{\tau}\vee_{\lambda}\Lie\twoheadrightarrow \Pois^{\tau}$, where $\Com^{\tau}\vee_{\lambda}\Lie$ is the quadratic contractad obtained from rewriting rule~\eqref{eq::rewr_rule_for_inariant}. Thanks to Proposition~\ref{prop::distributive_law} and Lemma~\ref{lem::Pois_odd::collection}, we have a surjective endomorphism of the graphical collection $$\Com^{\tau}\circ\Lie\twoheadrightarrow \Com^{\tau}\vee_{\lambda}\Lie\twoheadrightarrow\Pois^{\tau}=\Com^{\tau}\circ\Lie$$ 
		that automatically implies the isomorphism $\Pois^{\tau}\cong \Com^{\tau}\vee_{\lambda}\Lie$. Hence, by Proposition~\ref{prop::distributive_law}, the rewriting rule $\lambda$ defines a distributive law. Since $\Com^{\tau}$ and $\Lie$ are Koszul, thanks to Proposition~\ref{prop::distributive_law}, $\Pois^{\tau}$ is also Koszul.
	\end{proof}
	We will also denote the contractad $\Com^{\tau}$ by $\Com_{\mathrm{odd}}$ because it consists of a unique operation for each graph with an odd number of vertices.
	Let $\Lie_{\mathrm{odd}}$, $\Pois_{\mathrm{odd}}$ be contractads Koszul dual to $\Com^{\tau}$ and $\Pois^{\tau}$ respectively. As an immediate consequence of Theorem~\ref{thm::invariant_com_and_pois}, we have
	\begin{theorem}[Odd Lie and Poisson contractads]\label{thm::odd_lie_and_pois} \hfill\break \begin{itemize}
			\item[(i)] The odd Lie contractad $\Lie_{\mathrm{odd}}$ is the contractad with anti-symmetric ternary generators $p_{\Path_3},p_{\K_3}$ of homological degree $1$, that is
			\[
			p_{\Gr}(x_{\sigma(1)},x_{\sigma(2)},x_{\sigma(3)})=(-1)^{\sigma}p_{\Gr}(x_1,x_2,x_3), \quad \Gr=\Path_3,\K_3.
			\] satisfying the relations: for each graph $\Gr$ on $5$ vertices $\{1,\cdots,5\}$, we have
			\[
			\sum (-1)^{(i_1,i_2,i_3,i_4,i_5)}p(x_{i_1},x_{i_2},p(x_{i_3},x_{i_4},x_{i_5}))=0,
			\] where the sum ranges over all 3-tubes $\{i_3,i_4,i_5\}$, $i_1<i_2$ remaining vertices and $(-1)^{(i_1,i_2,j_1,j_2,j_3)}$ is the sign of the corresponding partition.
			
			\item[(ii)] The odd Poisson contractad $\Pois_{\mathrm{odd}}$ is the contractad obtained from contractads $\Com$ and $\Lie_{\mathrm{odd}}$ by the rewriting rule:  for each pair $(\Gr,e)$, where $\Gr$ is a graph on 4 vertices $\{1,2,3,4\}$ and edge $e=\{1,2\}$, we have have the following analogue of Leibnitz identity:
			\begin{equation}\label{eq::rewr_rule_for_odd}
				p(m(x_1,x_2),x_3,x_4)=\epsilon_{2,3,4}m(x_1,p(x_2,x_3,x_4))+\epsilon_{1,3,4}m(p(x_1,x_3,x_4),x_2),    
			\end{equation} where $\epsilon_{v,w,u}=1$ for $\{v,w,u\}$ a tube and zero otherwise. In particular, this contractad is Koszul and on the level of graphical collections, we have 
			\[
			\Pois_{\mathrm{odd}}\cong \Com\circ \Lie_{\mathrm{odd}}.
			\]
		\end{itemize}
	\end{theorem}
	\begin{proof}(i) By direct computations, the signs in relations appear due to suspensions (ii) Recall that a Koszul dual of contractad $\Pop\vee_{\lambda}\Q$ obtained from rewriting rules is again a contractad obtained from rewriting rule. Applying to our case, we have
		\[
		\Pois_{\mathrm{odd}}=[\Pois^\tau]^!=(\Com^{\tau}\vee_{\lambda}\Lie)^!\cong \Lie^!\vee_{\lambda^!}(\Com^{\tau})^!=\Com\vee_{\lambda^!} \Lie_{\mathrm{odd}}. 
		\] By direct inspection, we see that the Koszul dual $\lambda^!$ to rewriting rule~\eqref{eq::rewr_rule_for_inariant} coincides with rewriting rule~\eqref{eq::rewr_rule_for_odd}. Thanks to Theorem~\ref{thm::invariant_com_and_pois} and Proposition~\ref{prop::distributive_law} we conclude that this rule provides a distributive law, so $\Pois_{\mathrm{odd}}\cong \Com\circ\Lie_{\mathrm{odd}}$ and the contractad $\Pois_{\mathrm{odd}}$ is Koszul.
	\end{proof}
	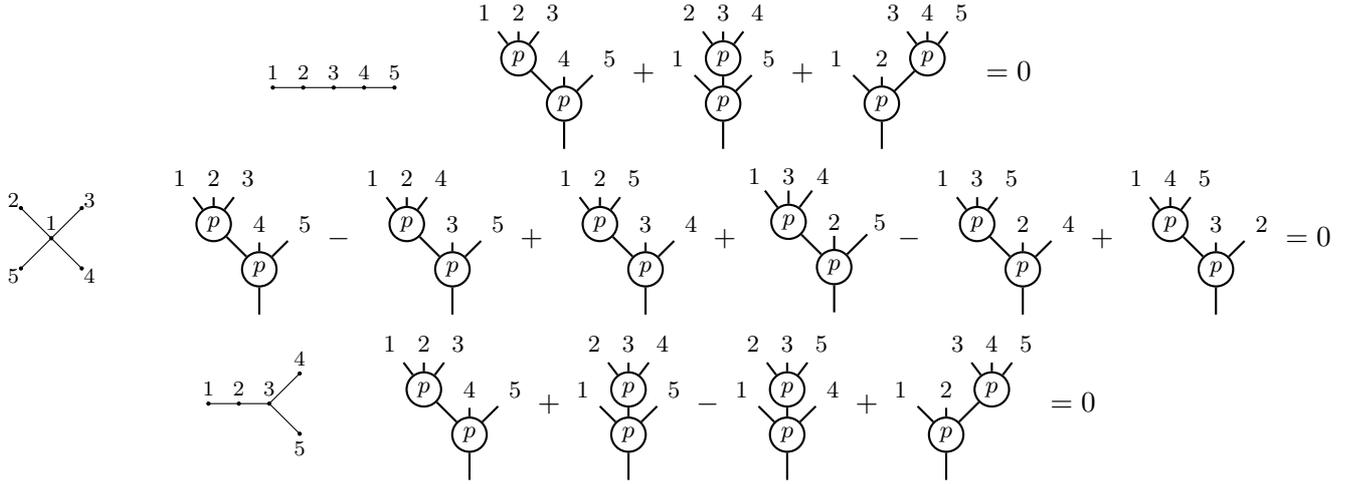
\begin{figure}[ht]
		\begin{gather*}
			\vcenter{\hbox{\begin{tikzpicture}[scale=0.4]
						\fill (0,0) circle (2pt);
						\node at (0,0.5) {\scriptsize$1$};
						\fill (1,0) circle (2pt);
						\node at (1,0.5) {\scriptsize$2$};
						\fill (2,0) circle (2pt);
						\node at (2,0.5) {\scriptsize$3$};
						\fill (3,0) circle (2pt);
						\node at (3,0.5) {\scriptsize$4$};
						\fill (4,0) circle (2pt);
						\node at (4,0.5) {\scriptsize$5$};
						\draw (0,0)--(1,0)--(2,0)--(3,0)--(4,0);    
			\end{tikzpicture}}}\quad\quad
			\vcenter{\hbox{\begin{tikzpicture}[
						scale=0.6,
						vert/.style={circle,  draw=black!30!black, thick, minimum size=1mm, inner sep=2pt},
						leaf/.style={rectangle, thick, minimum size=1mm},
						edge/.style={-,black!30!black, thick},
						]
						\node[vert] (bot) at (0,1) {\footnotesize$p$};
						\node[vert] (up) at (-1,2) {\footnotesize$p$};
						\node[leaf] (bot1) at (1,2) {\footnotesize$5$};
						\node[leaf] (bot2) at (0,2) {\footnotesize$4$};
						\node[leaf] (up1) at (-1.75,3) {\footnotesize$1$};
						\node[leaf] (up2) at (-1,3) {\footnotesize$2$};
						\node[leaf] (up3) at (-0.25,3) {\footnotesize$3$};
						\draw[edge] (0,0)--(bot);
						\draw[edge] (bot)--(bot1);
						\draw[edge] (bot)--(bot2);
						\draw[edge] (bot)--(up);
						\draw[edge] (up)--(up1);
						\draw[edge] (up)--(up2);
						\draw[edge] (up)--(up3);
			\end{tikzpicture}}}+
			\vcenter{\hbox{\begin{tikzpicture}[
						scale=0.6,
						vert/.style={circle,  draw=black!30!black, thick, minimum size=1mm, inner sep=2pt},
						leaf/.style={rectangle, thick, minimum size=1mm},
						edge/.style={-,black!30!black, thick},
						]
						\node[vert] (bot) at (0,1) {\footnotesize$p$};
						\node[vert] (up) at (0,2) {\footnotesize$p$};
						\node[leaf] (bot1) at (1,2) {\footnotesize$5$};
						\node[leaf] (bot2) at (-1,2) {\footnotesize$1$};
						\node[leaf] (up1) at (-0.75,3) {\footnotesize$2$};
						\node[leaf] (up2) at (0,3) {\footnotesize$3$};
						\node[leaf] (up3) at (0.75,3) {\footnotesize$4$};
						\draw[edge] (0,0)--(bot);
						\draw[edge] (bot)--(bot1);
						\draw[edge] (bot)--(bot2);
						\draw[edge] (bot)--(up);
						\draw[edge] (up)--(up1);
						\draw[edge] (up)--(up2);
						\draw[edge] (up)--(up3);
			\end{tikzpicture}}}+
			\vcenter{\hbox{\begin{tikzpicture}[
						scale=0.6,
						vert/.style={circle,  draw=black!30!black, thick, minimum size=1mm, inner sep=2pt},
						leaf/.style={rectangle, thick, minimum size=1mm},
						edge/.style={-,black!30!black, thick},
						]
						\node[vert] (bot) at (0,1) {\footnotesize$p$};
						\node[vert] (up) at (1,2) {\footnotesize$p$};
						\node[leaf] (up1) at (0.25,3) {\footnotesize$3$};
						\node[leaf] (up2) at (1,3) {\footnotesize$4$};
						\node[leaf] (up3) at (1.75,3) {\footnotesize$5$};
						\node[leaf] (bot1) at (-1,2) {\footnotesize$1$};
						\node[leaf] (bot2) at (0,2) {\footnotesize$2$};
						\draw[edge] (0,0)--(bot);
						\draw[edge] (bot)--(bot1);
						\draw[edge] (bot)--(bot2);
						\draw[edge] (bot)--(up);
						\draw[edge] (up)--(up1);
						\draw[edge] (up)--(up2);
						\draw[edge] (up)--(up3);
			\end{tikzpicture}}}=0
			\\
			\vcenter{\hbox{\begin{tikzpicture}[scale=0.4]
						\fill (0,0) circle (2pt);
						\node at (0,0.5) {\scriptsize$1$};
						\fill (-1,1) circle (2pt);
						\node at (-1.25,1.25) {\scriptsize$2$};
						\fill (1,1) circle (2pt);
						\node at (1.25,1.25) {\scriptsize$3$};
						\fill (1,-1) circle (2pt);
						\node at (1.25,-1.25) {\scriptsize$4$};
						\fill (-1,-1) circle (2pt);
						\node at (-1.25,-1.25) {\scriptsize$5$};
						\draw (0,0)--(-1,1);
						\draw (0,0)--(1,1);
						\draw (0,0)--(1,-1);
						\draw (0,0)--(-1,-1);
			\end{tikzpicture}}}\quad\quad
			\vcenter{\hbox{\begin{tikzpicture}[
						scale=0.6,
						vert/.style={circle,  draw=black!30!black, thick, minimum size=1mm, inner sep=2pt},
						leaf/.style={rectangle, thick, minimum size=1mm},
						edge/.style={-,black!30!black, thick},
						]
						\node[vert] (bot) at (0,1) {\footnotesize$p$};
						\node[vert] (up) at (-1,2) {\footnotesize$p$};
						\node[leaf] (bot1) at (1,2) {\footnotesize$5$};
						\node[leaf] (bot2) at (0,2) {\footnotesize$4$};
						\node[leaf] (up1) at (-1.75,3) {\footnotesize$1$};
						\node[leaf] (up2) at (-1,3) {\footnotesize$2$};
						\node[leaf] (up3) at (-0.25,3) {\footnotesize$3$};
						\draw[edge] (0,0)--(bot);
						\draw[edge] (bot)--(bot1);
						\draw[edge] (bot)--(bot2);
						\draw[edge] (bot)--(up);
						\draw[edge] (up)--(up1);
						\draw[edge] (up)--(up2);
						\draw[edge] (up)--(up3);
			\end{tikzpicture}}}-
			\vcenter{\hbox{\begin{tikzpicture}[
						scale=0.6,
						vert/.style={circle,  draw=black!30!black, thick, minimum size=1mm, inner sep=2pt},
						leaf/.style={rectangle, thick, minimum size=1mm},
						edge/.style={-,black!30!black, thick},
						]
						\node[vert] (bot) at (0,1) {\footnotesize$p$};
						\node[vert] (up) at (-1,2) {\footnotesize$p$};
						\node[leaf] (bot1) at (1,2) {\footnotesize$5$};
						\node[leaf] (bot2) at (0,2) {\footnotesize$3$};
						\node[leaf] (up1) at (-1.75,3) {\footnotesize$1$};
						\node[leaf] (up2) at (-1,3) {\footnotesize$2$};
						\node[leaf] (up3) at (-0.25,3) {\footnotesize$4$};
						\draw[edge] (0,0)--(bot);
						\draw[edge] (bot)--(bot1);
						\draw[edge] (bot)--(bot2);
						\draw[edge] (bot)--(up);
						\draw[edge] (up)--(up1);
						\draw[edge] (up)--(up2);
						\draw[edge] (up)--(up3);
			\end{tikzpicture}}}+
			\vcenter{\hbox{\begin{tikzpicture}[
						scale=0.6,
						vert/.style={circle,  draw=black!30!black, thick, minimum size=1mm, inner sep=2pt},
						leaf/.style={rectangle, thick, minimum size=1mm},
						edge/.style={-,black!30!black, thick},
						]
						\node[vert] (bot) at (0,1) {\footnotesize$p$};
						\node[vert] (up) at (-1,2) {\footnotesize$p$};
						\node[leaf] (bot1) at (1,2) {\footnotesize$4$};
						\node[leaf] (bot2) at (0,2) {\footnotesize$3$};
						\node[leaf] (up1) at (-1.75,3) {\footnotesize$1$};
						\node[leaf] (up2) at (-1,3) {\footnotesize$2$};
						\node[leaf] (up3) at (-0.25,3) {\footnotesize$5$};
						\draw[edge] (0,0)--(bot);
						\draw[edge] (bot)--(bot1);
						\draw[edge] (bot)--(bot2);
						\draw[edge] (bot)--(up);
						\draw[edge] (up)--(up1);
						\draw[edge] (up)--(up2);
						\draw[edge] (up)--(up3);
			\end{tikzpicture}}}+
			\vcenter{\hbox{\begin{tikzpicture}[
						scale=0.6,
						vert/.style={circle,  draw=black!30!black, thick, minimum size=1mm, inner sep=2pt},
						leaf/.style={rectangle, thick, minimum size=1mm, inner sep=2pt},
						edge/.style={-,black!30!black, thick},
						]
						\node[vert] (bot) at (0,1) {\footnotesize$p$};
						\node[vert] (up) at (-1,2) {\footnotesize$p$};
						\node[leaf] (bot1) at (1,2) {\footnotesize$5$};
						\node[leaf] (bot2) at (0,2) {\footnotesize$2$};
						\node[leaf] (up1) at (-1.75,3) {\footnotesize$1$};
						\node[leaf] (up2) at (-1,3) {\footnotesize$3$};
						\node[leaf] (up3) at (-0.25,3) {\footnotesize$4$};
						\draw[edge] (0,0)--(bot);
						\draw[edge] (bot)--(bot1);
						\draw[edge] (bot)--(bot2);
						\draw[edge] (bot)--(up);
						\draw[edge] (up)--(up1);
						\draw[edge] (up)--(up2);
						\draw[edge] (up)--(up3);
			\end{tikzpicture}}}-
			\vcenter{\hbox{\begin{tikzpicture}[
						scale=0.6,
						vert/.style={circle,  draw=black!30!black, thick, minimum size=1mm, inner sep=2pt},
						leaf/.style={rectangle, thick, minimum size=1mm},
						edge/.style={-,black!30!black, thick},
						]
						\node[vert] (bot) at (0,1) {\footnotesize$p$};
						\node[vert] (up) at (-1,2) {\footnotesize$p$};
						\node[leaf] (bot1) at (1,2) {\footnotesize$4$};
						\node[leaf] (bot2) at (0,2) {\footnotesize$2$};
						\node[leaf] (up1) at (-1.75,3) {\footnotesize$1$};
						\node[leaf] (up2) at (-1,3) {\footnotesize$3$};
						\node[leaf] (up3) at (-0.25,3) {\footnotesize$5$};
						\draw[edge] (0,0)--(bot);
						\draw[edge] (bot)--(bot1);
						\draw[edge] (bot)--(bot2);
						\draw[edge] (bot)--(up);
						\draw[edge] (up)--(up1);
						\draw[edge] (up)--(up2);
						\draw[edge] (up)--(up3);
			\end{tikzpicture}}}+
			\vcenter{\hbox{\begin{tikzpicture}[
						scale=0.6,
						vert/.style={circle,  draw=black!30!black, thick, minimum size=1mm, inner sep=2pt},
						leaf/.style={rectangle, thick, minimum size=1mm},
						edge/.style={-,black!30!black, thick},
						]
						\node[vert] (bot) at (0,1) {\footnotesize$p$};
						\node[vert] (up) at (-1,2) {\footnotesize$p$};
						\node[leaf] (bot1) at (1,2) {\footnotesize$2$};
						\node[leaf] (bot2) at (0,2) {\footnotesize$3$};
						\node[leaf] (up1) at (-1.75,3) {\footnotesize$1$};
						\node[leaf] (up2) at (-1,3) {\footnotesize$4$};
						\node[leaf] (up3) at (-0.25,3) {\footnotesize$5$};
						\draw[edge] (0,0)--(bot);
						\draw[edge] (bot)--(bot1);
						\draw[edge] (bot)--(bot2);
						\draw[edge] (bot)--(up);
						\draw[edge] (up)--(up1);
						\draw[edge] (up)--(up2);
						\draw[edge] (up)--(up3);
			\end{tikzpicture}}}=0
			\\
			\vcenter{\hbox{\begin{tikzpicture}[scale=0.4]
						\fill (0,0) circle (2pt);
						\node at (0,0.5) {\scriptsize$1$};
						\fill (1,0) circle (2pt);
						\node at (1,0.5) {\scriptsize$2$};
						\fill (2,0) circle (2pt);
						\node at (2,0.5) {\scriptsize$3$};
						\fill (3,1) circle (2pt);
						\node at (3,1.5) {\scriptsize$4$};
						\fill (3,-1) circle (2pt);
						\node at (3,-1.5) {\scriptsize$5$};
						\draw (0,0)--(1,0)--(2,0);
						\draw (2,0)--(3,1);
						\draw (2,0)--(3,-1);
			\end{tikzpicture}}}\quad\quad
			\vcenter{\hbox{\begin{tikzpicture}[
						scale=0.6,
						vert/.style={circle,  draw=black!30!black, thick, minimum size=1mm, inner sep=2pt},
						leaf/.style={rectangle, thick, minimum size=1mm},
						edge/.style={-,black!30!black, thick},
						]
						\node[vert] (bot) at (0,1) {\footnotesize$p$};
						\node[vert] (up) at (-1,2) {\footnotesize$p$};
						\node[leaf] (bot1) at (1,2) {\footnotesize$5$};
						\node[leaf] (bot2) at (0,2) {\footnotesize$4$};
						\node[leaf] (up1) at (-1.75,3) {\footnotesize$1$};
						\node[leaf] (up2) at (-1,3) {\footnotesize$2$};
						\node[leaf] (up3) at (-0.25,3) {\footnotesize$3$};
						\draw[edge] (0,0)--(bot);
						\draw[edge] (bot)--(bot1);
						\draw[edge] (bot)--(bot2);
						\draw[edge] (bot)--(up);
						\draw[edge] (up)--(up1);
						\draw[edge] (up)--(up2);
						\draw[edge] (up)--(up3);
			\end{tikzpicture}}}+
			\vcenter{\hbox{\begin{tikzpicture}[
						scale=0.6,
						vert/.style={circle,  draw=black!30!black, thick, minimum size=1mm, inner sep=2pt},
						leaf/.style={rectangle, thick, minimum size=1mm},
						edge/.style={-,black!30!black, thick},
						]
						\node[vert] (bot) at (0,1) {\footnotesize$p$};
						\node[vert] (up) at (0,2) {\footnotesize$p$};
						\node[leaf] (bot1) at (1,2) {\footnotesize$5$};
						\node[leaf] (bot2) at (-1,2) {\footnotesize$1$};
						\node[leaf] (up1) at (-0.75,3) {\footnotesize$2$};
						\node[leaf] (up2) at (0,3) {\footnotesize$3$};
						\node[leaf] (up3) at (0.75,3) {\footnotesize$4$};
						\draw[edge] (0,0)--(bot);
						\draw[edge] (bot)--(bot1);
						\draw[edge] (bot)--(bot2);
						\draw[edge] (bot)--(up);
						\draw[edge] (up)--(up1);
						\draw[edge] (up)--(up2);
						\draw[edge] (up)--(up3);
			\end{tikzpicture}}}-
			\vcenter{\hbox{\begin{tikzpicture}[
						scale=0.6,
						vert/.style={circle,  draw=black!30!black, thick, minimum size=1mm, inner sep=2pt},
						leaf/.style={rectangle, thick, minimum size=1mm},
						edge/.style={-,black!30!black, thick},
						]
						\node[vert] (bot) at (0,1) {\footnotesize$p$};
						\node[vert] (up) at (0,2) {\footnotesize$p$};
						\node[leaf] (bot1) at (1,2) {\footnotesize$4$};
						\node[leaf] (bot2) at (-1,2) {\footnotesize$1$};
						\node[leaf] (up1) at (-0.75,3) {\footnotesize$2$};
						\node[leaf] (up2) at (0,3) {\footnotesize$3$};
						\node[leaf] (up3) at (0.75,3) {\footnotesize$5$};
						\draw[edge] (0,0)--(bot);
						\draw[edge] (bot)--(bot1);
						\draw[edge] (bot)--(bot2);
						\draw[edge] (bot)--(up);
						\draw[edge] (up)--(up1);
						\draw[edge] (up)--(up2);
						\draw[edge] (up)--(up3);
			\end{tikzpicture}}}+
			\vcenter{\hbox{\begin{tikzpicture}[
						scale=0.6,
						vert/.style={circle,  draw=black!30!black, thick, minimum size=1mm, inner sep=2pt},
						leaf/.style={rectangle, thick, minimum size=1mm},
						edge/.style={-,black!30!black, thick},
						]
						\node[vert] (bot) at (0,1) {\footnotesize$p$};
						\node[vert] (up) at (1,2) {\footnotesize$p$};
						\node[leaf] (up1) at (0.25,3) {\footnotesize$3$};
						\node[leaf] (up2) at (1,3) {\footnotesize$4$};
						\node[leaf] (up3) at (1.75,3) {\footnotesize$5$};
						\node[leaf] (bot1) at (-1,2) {\footnotesize$1$};
						\node[leaf] (bot2) at (0,2) {\footnotesize$2$};
						\draw[edge] (0,0)--(bot);
						\draw[edge] (bot)--(bot1);
						\draw[edge] (bot)--(bot2);
						\draw[edge] (bot)--(up);
						\draw[edge] (up)--(up1);
						\draw[edge] (up)--(up2);
						\draw[edge] (up)--(up3);
			\end{tikzpicture}}}=0
		\end{gather*}
		\caption{Quadratic relations in $\Lie_{\mathrm{odd}}$ for trees on $5$ vertices. Note that arranging of leaves is important by signs reasons}
	\end{figure}
	\begin{figure}[ht]
		\begin{gather*}
			\vcenter{\hbox{\begin{tikzpicture}[scale=0.5]
						\fill (0,0) circle (2pt);
						\node at (0,0.5) {\footnotesize$1$};
						\fill (1,0) circle (2pt);
						\node at (1,0.5) {\footnotesize$2$};
						\fill (2,0) circle (2pt);
						\node at (2,0.5) {\footnotesize$3$};
						\fill (3,0) circle (2pt);
						\node at (3,0.5) {\footnotesize$4$};
						\draw (0,0)--(1,0)--(2,0)--(3,0);    
			\end{tikzpicture}}}\quad\quad
			\vcenter{\hbox{\begin{tikzpicture}[
						scale=0.8,
						vert/.style={circle,  draw=black!30!black, thick, minimum size=1mm, inner sep=3pt},
						leaf/.style={rectangle, thick, minimum size=1mm},
						edge/.style={-,black!30!black, thick},
						]
						\node[vert] (bot) at (0,1) {$p$};
						\node[vert] (up) at (-1,2) {$m$};
						\node[leaf] (bot1) at (1,2) {$4$};
						\node[leaf] (bot2) at (0,2) {$3$};
						\node[leaf] (up1) at (-1.6,3) {$1$};
						\node[leaf] (up3) at (-0.4,3) {$2$};
						\draw[edge] (0,0)--(bot);
						\draw[edge] (bot)--(bot1);
						\draw[edge] (bot)--(bot2);
						\draw[edge] (bot)--(up);
						\draw[edge] (up)--(up1);
						\draw[edge] (up)--(up3);
			\end{tikzpicture}}}=  
			\vcenter{\hbox{\begin{tikzpicture}[
						scale=0.8,
						vert/.style={circle,  draw=black!30!black, thick, minimum size=1mm, inner sep=3pt},
						leaf/.style={rectangle, thick, minimum size=1mm},
						edge/.style={-,black!30!black, thick},
						]
						\node[vert] (bot) at (0,1) {$m$};
						\node[vert] (up) at (0.75,2) {$p$};
						\node[leaf] (up1) at (0,3) {$2$};
						\node[leaf] (up2) at (0.75,3) {$3$};
						\node[leaf] (up3) at (1.5,3) {$4$};
						\node[leaf] (bot1) at (-0.75,2) {$1$};
						\draw[edge] (0,0)--(bot);
						\draw[edge] (bot)--(bot1);
						\draw[edge] (bot)--(up);
						\draw[edge] (up)--(up1);
						\draw[edge] (up)--(up2);
						\draw[edge] (up)--(up3);
			\end{tikzpicture}}},\quad 
			\vcenter{\hbox{\begin{tikzpicture}[
						scale=0.8,
						vert/.style={circle,  draw=black!30!black, thick, minimum size=1mm, inner sep=3pt},
						leaf/.style={rectangle, thick, minimum size=1mm},
						edge/.style={-,black!30!black, thick},
						]
						\node[vert] (bot) at (0,1) {$p$};
						\node[vert] (up) at (0,2) {$m$};
						\node[leaf] (bot1) at (1,2) {$4$};
						\node[leaf] (bot2) at (-1,2) {$1$};
						\node[leaf] (up1) at (-0.6,3) {$2$};
						\node[leaf] (up3) at (0.6,3) {$3$};
						\draw[edge] (0,0)--(bot);
						\draw[edge] (bot)--(bot1);
						\draw[edge] (bot)--(bot2);
						\draw[edge] (bot)--(up);
						\draw[edge] (up)--(up1);
						\draw[edge] (up)--(up3);
			\end{tikzpicture}}}=0;
			\\
			\vcenter{\hbox{\begin{tikzpicture}[scale=0.5]
						\fill (0,0) circle (2pt);
						\fill (0,1.5) circle (2pt);
						\fill (1.5,0) circle (2pt);
						\fill (1.5,1.5) circle (2pt);
						\draw (0,0)--(1.5,0)--(1.5,1.5)--(0,1.5)-- cycle;
						\node at (-0.25,1.75) {\footnotesize$1$};
						\node at (1.75,1.75) {\footnotesize$2$};
						\node at (1.75,-0.25) {\footnotesize$3$};
						\node at (-0.25,-0.25) {\footnotesize$4$};
			\end{tikzpicture}}}\quad\quad
			\vcenter{\hbox{\begin{tikzpicture}[
						scale=0.8,
						vert/.style={circle,  draw=black!30!black, thick, minimum size=1mm, inner sep=3pt},
						leaf/.style={rectangle, thick, minimum size=1mm},
						edge/.style={-,black!30!black, thick},
						]
						\node[vert] (bot) at (0,1) {$p$};
						\node[vert] (up) at (-1,2) {$m$};
						\node[leaf] (bot1) at (1,2) {$4$};
						\node[leaf] (bot2) at (0,2) {$3$};
						\node[leaf] (up1) at (-1.6,3) {$1$};
						\node[leaf] (up3) at (-0.4,3) {$2$};
						\draw[edge] (0,0)--(bot);
						\draw[edge] (bot)--(bot1);
						\draw[edge] (bot)--(bot2);
						\draw[edge] (bot)--(up);
						\draw[edge] (up)--(up1);
						\draw[edge] (up)--(up3);
			\end{tikzpicture}}}=
			\vcenter{\hbox{\begin{tikzpicture}[
						scale=0.8,
						vert/.style={circle,  draw=black!30!black, thick, minimum size=1mm},
						leaf/.style={rectangle, thick, minimum size=1mm, inner sep=3pt},
						edge/.style={-,black!30!black, thick},
						]
						\node[vert] (bot) at (0,1) {$m$};
						\node[vert] (up) at (0.75,2) {$p$};
						\node[leaf] (up1) at (0,3) {$2$};
						\node[leaf] (up2) at (0.75,3) {$3$};
						\node[leaf] (up3) at (1.5,3) {$4$};
						\node[leaf] (bot1) at (-0.75,2) {$1$};
						\draw[edge] (0,0)--(bot);
						\draw[edge] (bot)--(bot1);
						\draw[edge] (bot)--(up);
						\draw[edge] (up)--(up1);
						\draw[edge] (up)--(up2);
						\draw[edge] (up)--(up3);
			\end{tikzpicture}}}+
			\vcenter{\hbox{\begin{tikzpicture}[
						scale=0.8,
						vert/.style={circle,  draw=black!30!black, thick, minimum size=1mm, inner sep=3pt},
						leaf/.style={rectangle, thick, minimum size=1mm},
						edge/.style={-,black!30!black, thick},
						]
						\node[vert] (bot) at (0,1) {$m$};
						\node[vert] (up) at (-0.75,2) {$p$};
						\node[leaf] (bot1) at (0.75,2) {$2$};
						\node[leaf] (up1) at (-1.5,3) {$1$};
						\node[leaf] (up2) at (-0.75,3) {$3$};
						\node[leaf] (up3) at (0,3) {$4$};
						\draw[edge] (0,0)--(bot);
						\draw[edge] (bot)--(bot1);
						\draw[edge] (bot)--(up);
						\draw[edge] (up)--(up1);
						\draw[edge] (up)--(up2);
						\draw[edge] (up)--(up3);
			\end{tikzpicture}}}
		\end{gather*}
		\caption{Distributive law between $\Lie_{\mathrm{odd}}$ and $\Com$ in $\Pois_{\mathrm{odd}}$ for $\Path_4$ and $\Cyc_4$. We mention only minimal generating set of rewriting rules, remaining relations are obtained from graph-automorphisms}
		\label{fig:enter-label}
	\end{figure}
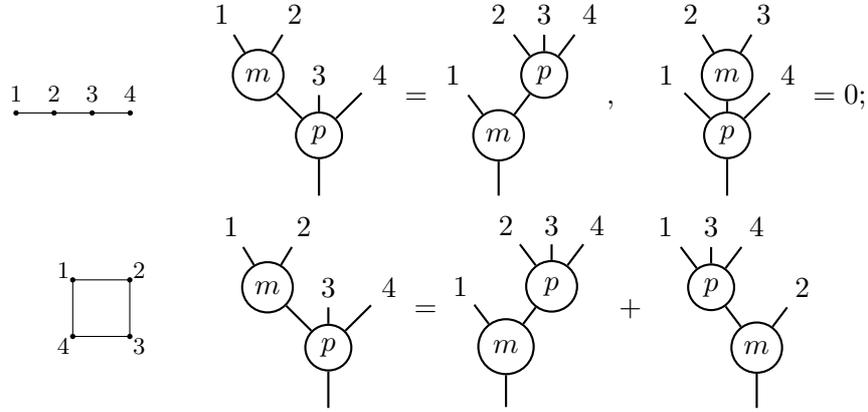
	\subsection{Homology of the real locus}
	\label{sec::realhomology}
	
	Finally, we can describe the rational cohomology of the contractad $\bM_{\mathbb{R}}$. The proof repeats the proof of the analogous fact discovered for the operad $\bM_{\mathbb{R}}$ in~\cite{khoroshkin2019real}.
	We start from the following observation
	\begin{theorem}
		The algebraic contractad $[\Ass]^{\tau}$ of $\tau$-invariants on the associative contractads is generated by binary operation $b\in\Ass(\Path_2)$ and by ternary operations $\mu$ in components $\Path_3$ and $\K_3$ subject to the following list of relations:
		\begin{enumerate}
			\item 
			$b$  generates the contractad $\Lie$,  i.e. $b$ satisfies the Jacobi identities; 
			\item 
			The Leibnitz identity~\eqref{eq::rewr_rule_for_inariant} holds for interchanging between $b$ and $\mu$.
			\item 
			For each graph $\Gr$ on $5$ vertices $\{1,2,3,4,5\}$ and a pair of tubes $(i_3,i_4,i_5)$, $(j_3,j_4,j_5)$ in $\Gamma$ we have an inhomogeneous identity: 
			\begin{equation}
				\label{eq::Ass::inv::identity}    
				\mu(x_{i_1},x_{i_2},\mu(x_{i_3},x_{i_4},x_{i_5})) - \mu(x_{j_1},x_{j_2},\mu(x_{j_3},x_{j_4},x_{j_5})) =     
				\phi_3(x_{i_1},x_{i_2},x_{i_3},x_{i_4},x_{i_5}) +
				\phi_4(x_{i_1},x_{i_2},x_{i_3},x_{i_4},x_{i_5}),
			\end{equation}
			where $\phi_3$ is a cubic expression that uses one operation $\mu$ and two operations $b$ and $\phi_4$ is a quartic expression that uses $4$ operations $b$.    
		\end{enumerate}
	\end{theorem}
	\begin{proof}
		Recall that in Example~\ref{sec:Ass} we explained that contractad $\Ass$ is generated by a binary nonsymmetric operation $\nu$ and the operation $b:=\nu-\nu^{\op}$ defines a Lie bracket and $m:=\nu+\nu^{\op}$ defines a Jordan operation.
		As we explained in~\S\ref{sec::Pois::inv} the involution $\tau$ on $\Ass$ preserves the filtration $\calF$ is given by the number of Jordan operations. The associated graded to this filtration is a quadratic Koszul contractad $\Pois^{\tau}$ whose presentation in generators and relations was given in Theorem~\ref{thm::invariant_com_and_pois}. 
		Therefore, the contractad $\Ass^{\tau}$ is generated by the Lie bracket $b$ and by operations 
		$$\mu(x_1,x_2,x_3):=\sum \nu(\nu(x_{i_1},x_{i_2}),x_{i_3}) + \nu(x_{i_3},\nu(x_{i_1},x_{i_2})),$$ 
		where summation goes through all edges $(i_1 i_2)$ in a graph $\Gamma$ on $3$ vertices.
		We see that Jacobi relations hold for $b$ because $b$ generates the subcontractad isomorphic to $\Lie$. Second, we know that $b$ differentiates associative product $\mu$ and, therefore, $b$ differentiates any expression of $\mu$'s which implies the generalized Leibnitz identity~\eqref{eq::rewr_rule_for_inariant}.
		Finally, we know that the left hand side of~\eqref{eq::Ass::inv::identity} holds in the associated graded contractad $\Pois^{\tau}$. Therefore, the right hand side belongs to the lower components of the filtration $\calF$ and can be presented with at most one operation $\mu$. However, this should be an operation with $5$ arguments, so, this expression either contains one $\mu$ and two $b$ or has $4$ operation $b$ involved. These components correspond to cubic and quartic components on the right-hand side of~\eqref{eq::Ass::inv::identity}.

		Note, that the associated graded to the relations we explained are quadratic relations for $\Pois^{\tau}$ that assemble a quadratic Gr\"obner basis for the contractad $\Pois^{\tau}$. Therefore, the relations we are considering also assemble an inhomogeneous Gr\"obner basis for $\Ass^{\tau}$, which implies that the list of relations we presented is full.
	\end{proof}

	\begin{theorem}\label{thm::homology_real_locus}
		The algebraic contractad $H_{\bullet}(\bM_{\mathbb{R}})$ is isomorphic to the quadratic Koszul contractad $\Pois_{\mathrm{odd}}$. In particular, we have an isomorphism of graphical collections:
		\begin{equation}
			\label{eq::M_R::homology::collectio}
			H_{\bullet}(\bM_{\mathbb{R}};\mathbb{Q})\cong 
			\Pois_{\mathrm{odd}} \cong
			\Com\circ \Lie_{\mathrm{odd}}.
		\end{equation}
		However, this contractad is not formal, meaning that the contractads $C^{\mathrm{cell}}_{\bullet}(\bM_{\mathbb{R}};\mathbb{Q})$ and $H_{\bullet}(\bM_{\mathbb{R}};\mathbb{Q})$ are not equivalent as differential graded contractads.
	\end{theorem}
	\begin{proof}
		Thanks to Isomorphism~\eqref{eq::cell::mosaic} we know that the cell complex $C^{\mathrm{cell}}_{\bullet}(\bM_{\mathbb{R}};\mathbb{Q})$ is isomorphic to the cobar complex of the cocontractad dual to the contractad $\Ass^{\tau}$.
		Consider, the filtration on this cobar complex induced by filtration $\calF$ on $\Ass^{\tau}$. The associated graded complex is the cobar complex of the cocontractad $(\Pois^{\tau})^{*}$. 
		We know that the latter is Koszul and its cohomology coincides with $\Pois_{\mathrm{odd}}$. Therefore, by degree reasons, we see that the corresponding spectral sequence on $\Omega(\Susp^{-1}[\Ass^{\tau}]^*)$ degenerates in the first term, which leads to the coincidence of the homology of $\bM_{\mathbb{R}}$ and $\Pois_{\mathrm{odd}}$. 
		
		If we suppose that $C^{\mathrm{cell}}_{\bullet}(\bM_{\mathbb{R}};\mathbb{Q})$ and $H_{\bullet}(\bM_{\mathbb{R}};\mathbb{Q})$ are equivalent, then their bar complexes should be also equivalent contractads, but we have:
		$$
		\mathsf{B}(C^{\mathrm{cell}}_{\bullet}(\bM_{\mathbb{R}};\mathbb{Q})) \simeq 
		\mathsf{B}\Cobar({\Susp^{-1}}[\Ass^{\tau}]^*) \sim 
		[\Ass^{\tau}]^* \not\simeq [\Pois^{\tau}]^* \sim 
		\mathsf{B}\Cobar({\Susp^{-1}}[\Pois^{\tau}]^*) \simeq
		\mathsf{B}(H_{\bullet}(\bM_{\mathbb{R}};\mathbb{Q})).$$
	\end{proof}
	
	We would like to finish this section with several interesting observations.
	\begin{remark}
		For $\Pois$ and $\Ass$ the relations in $\Path_3$ coincide. This implies that the restriction of contractads $\Pois$ and $\Ass$ to trees coincide. 
		In particular, this implies that $\bM_{\mathbb{R}}$ is a formal contractad in the class of trees.
	\end{remark}
	\begin{remark}
		Isomorphism~\eqref{eq::M_R::homology::collectio} of graphical collections can be derived directly from the paper~\cite{rains2010homology} where he describes the homology of the real locus of a wonderful compactification in terms of the homology of a partition posets with odd parts. 
		
		Moreover, what follows from the same paper of Rains is that the integer homology of  $H_{\bullet}(\bM_{\mathbb{R}};\mathbb{Z})$ has only $2$-torsion and one has the following isomorphism for $\bZ_2$-coefficients:
		$$H^{\bullet}(\bM_{\mathbb{R}}(\Gr);\mathbb{Z}_2)\cong H^{2\bullet}(\bM_{\mathbb{C}}(\Gr);\mathbb{Z}_2) \cong  \Hyper_{2\bullet}(\Gr)\otimes \mathbb{Z}_2.$$ This isomorphism holds on the level of multiplicative and contractad structures.
	\end{remark}

	\subsection{Hilbert Series}
	\label{sec::M::real::Hilb}
	Similarly to the complex case, we compute Hilbert series of real wonderful compactifications $\bM_{\mathbb{R}}(\Gr)$ using the machinery of graphic functions. 
	
	We consider the graphic function $\chi_{-q}(\bM_{\mathbb{R}})$ that computes Hilbert series of components of the real wonderful contractad
	\[
	\chi_{-q}(\bM_{\mathbb{R}})(\Gr)=\sum_{i}(-q)^i\dim H^i(\bM_{\mathbb{R}}(\Gr);\mathbb{Q}).
	\] For simplicity of notations, let $\mathbb{1}^{\mathrm{odd}}_q:=\chi^{\mathrm{w}}_q(\Com_{\mathrm{odd}})$ and $\mu^{\mathrm{odd}}_q:=\chi^{\mathrm{w}}_{-q}(\Lie_{\mathrm{odd}})$. In particular, we have $\mathbb{1}^{\mathrm{odd}}_q(\Gr)=\sqrt{q}^{|V_{\Gr}|-1}$, for graphs on odd vertices, and zero otherwise. Thanks to Theorem~\ref{thm::invariant_com_and_pois}, $\Com_{\mathrm{odd}}$ is Koszul contractad, hence we have
	\begin{equation}
		\mathbb{1}^{\mathrm{odd}}_q*\mu^{\mathrm{odd}}_q=\mu^{\mathrm{odd}}_q*\mathbb{1}^{\mathrm{odd}}_q=\epsilon.
	\end{equation} Thanks to Isomorphism~\eqref{eq::M_R::homology::collectio} we have
	\begin{theorem}[Reccurence for $\bM_{\mathbb{R}}$]
		The following equalities on graphic functions are satisfied:
		\[
		\chi_{-q}(\bM_{\mathbb{R}})=\mathbb{1}*\mu^{\mathrm{odd}}_q \Leftrightarrow \chi_{-q}(\bM_{\mathbb{R}})*\mathbb{1}^{\mathrm{odd}}_q=\mathbb{1}.
		\] In particular, we have the recurrence
		\begin{equation}
			\sum_{I}\chi_{-q}(\bM_{\mathbb{R}})(\Gr/I)\cdot \sqrt{q}^{|V_{\Gr}|-|I|}=1,    
		\end{equation} where the sum ranges over all partitions $I$ of $\Gr$ made up of odd blocks.
	\end{theorem}
	What follows, that for complete multipartite graphs $\K_\lambda$ we can compute explicitly the generating series for rational Betti numbers of real locus of modular compactifications $\beM_{0,\K_{\lambda}}(\mathbb{R})$.
	\begin{theorem}[Hilbert series for $\beM(\mathbb{R})$]\label{thm::hilbert_real_modular}
		We have
		\begin{multline}
			\label{eq::M_R::series}
			F_{\mathsf{Y}}(\beM(\mathbb{R}))=
			\\
			=\sum_{l(\lambda)\geq 2} \left[\sum_i (-q)^i\dim H^i(\beM_{0,\K_{\lambda}}(\mathbb{R});\mathbb{Q})\right] \frac{m_{\lambda}}{\lambda!}+\sum_{n\geq 1,|\lambda|\geq 0} \left[\sum_i (-q)^i\dim H^i(\beM_{0,\K_{(1^n)\cup \lambda}}(\mathbb{R});\mathbb{Q})\right] \frac{m_{\lambda}}{\lambda!}\frac{z^n}{n!}=\\
			= \left[\sqrt{q}(z+\mathsf{SINH}_q)+\sqrt{q(z+\mathsf{SINH}_q)^2+1}\right]^{\frac{1}{\sqrt{q}}}-1-\sum_{n\geq 1}\frac{p_n}{n!},
		\end{multline}
		where  $\mathsf{SINH}_q=\sum_{n\geq 0}\frac{p_{2n+1}q^n}{(2n+1)!}$. 
	\end{theorem}
	\begin{proof}
		Similarly to computations in Example~\ref{ex::completemultipartite_for_com_lie}, we have
		\begin{equation}
			F_{\mathsf{Y}}(\mathbb{1}^{\mathrm{odd}}_q)=\frac{1}{\sqrt{q}}\sinh(\sqrt{q}(z+p_1))-\mathsf{SINH}_q
		\end{equation} Solving the functional equation for $\mu^{\mathrm{odd}}_q$, we have
		\begin{multline*}
			F_{\mathsf{Y}}(\mu^{\mathrm{odd}}_q)=\frac{1}{\sqrt{q}}\mathrm{arcsinh}(\sqrt{q}(z+\mathsf{SINH}_q))-p_1=
			\\ =
			\frac{1}{\sqrt{q}}\ln\left(\sqrt{q}(z+\mathsf{SINH}_q)+\sqrt{q(z+\mathsf{SINH}_q)^2+1}\right)-p_1,
		\end{multline*} thanks to the identity $\mathrm{arcsinh}(x)=\ln(x+\sqrt{x^2+1})$. The desired formula follows from
		\[
		F_{\mathsf{Y}}(\beM(\mathbb{R}))=F_{\mathsf{Y}}(\mathbb{1})(F_{\mathsf{Y}}(\mu^{\mathrm{odd}}_q))=\exp(F_{\mathsf{Y}}(\mu^{\mathrm{odd}}_q)+p_1)-1-\sum_{n\geq 1} \frac{p_n}{n!}.
		\]
	\end{proof}
	Let us list the generating series for one-parameter series of graphs. 
	\begin{sled}[Hilbert series for one-parameter families]
		For one-parameter families of graphs, we have
		\begin{gather}
			F_{\Path}(\bM_{\mathbb{R}})=t+\sum_{n\geq 2} 
			\left[\sum_{i\geq 0}\dim H_i(\bM_{\mathbb{R}}(\Path_n);\mathbb{Q})(-q)^i\right]t^n=\frac{\sqrt{1+4qt^2}-1}{2qt+1-\sqrt{1+4qt^2}};
			\\
			F_{\St}(\overline{\EuScript{L}}_{\mathbb{R}})=1+\sum_{n\geq 1} \left[\sum_{i\geq 0}\dim H_i(\overline{\EuScript{L}}_{0,n}(\mathbb{R});\mathbb{Q})(-q)^i\right]\frac{t^n}{n!}=\frac{e^t}{\cosh(\sqrt{q}t)};
			\\
			F_{\K}(\beM_{\mathbb{R}})=1+t+\sum_{n\geq 2} \left[\sum_{i\geq 0}\dim H_i(\beM_{0,n+1}(\mathbb{R});\mathbb{Q})(-q)^i\right]\frac{t^n}{n!}=\left(\sqrt{q}t+\sqrt{qt^2+1}\right)^{\frac{1}{\sqrt{q}}};
			\\
			\begin{array}{rl}
				F_{\Cyc}(\bM_{\mathbb{R}})= & t+\sum_{n\geq 2} \left[\sum_{i\geq 0}\dim H_i(\bM_{\mathbb{R}}(\Cyc_n);\mathbb{Q}) (-q)^i\right]\frac{t^n}{n}=   
				\\
				& \phantom{.........}=t-\log(1-L_q(t))-\frac{1}{2\sqrt{q}}\log\left(\frac{1+\sqrt{q}L_q(t)}{1-\sqrt{q}L_q(t)}\right);
			\end{array}
		\end{gather} 
		where $L_q(t)=\frac{\sqrt{1+4qt^2}-1}{2qt}$.
	\end{sled}

	\bibliographystyle{alpha}
	\bibliography{biblio.bib}

\begin{thebibliography}{MSvAX18}

\bibitem[BV73]{boardman1968homotopy}
John~Michael Boardman and Rainer~M Vogt.
\newblock {\em Homotopy-everything H-spaces}, volume 347 of {\em Lecture Notes
  in Mathematics}.
\newblock Springer-Verlag, 1973.

\bibitem[Cor22]{coron2022matroids}
Basile Coron.
\newblock Matroids, {Feynman} categories, and {Koszul} duality.
\newblock {\em arXiv preprint arXiv:2211.12370}, 2022.

\bibitem[Cor23]{coron2023supersolvability}
Basile Coron.
\newblock Supersolvability of built lattices and {Koszulness} of generalized
  {Chow} rings.
\newblock {\em arXiv preprint arXiv:2302.13072}, 2023.

\bibitem[DCP95]{de1995wonderful}
Corrado De~Concini and Claudio Procesi.
\newblock Wonderful models of subspace arrangements.
\newblock {\em Selecta Mathematica}, 1:459--494, 1995.

\bibitem[Del71]{deligne1971theorie}
Pierre Deligne.
\newblock Th{\'e}orie de {Hodge}: {II}.
\newblock {\em Publications Math{\'e}matiques de l'IH{\'E}S}, 40:5--57, 1971.

\bibitem[Dev99]{devadoss1999tessellations}
Satyan~Linus Devadoss.
\newblock {\em Tessellations of moduli spaces and the mosaic operad}.
\newblock The Johns Hopkins University, 1999.

\bibitem[DK10]{dotsenko2010grobner}
Vladimir Dotsenko and Anton Khoroshkin.
\newblock Gr{\"o}bner bases for operads.
\newblock {\em Duke Mathematical Journal}, 153(2):363--396, 2010.

\bibitem[DK13]{dotsenko2013quillen}
Vladimir Dotsenko and Anton Khoroshkin.
\newblock Quillen homology for operads via gr{\"o}bner bases.
\newblock {\em Documenta Mathematica}, 18:707--747, 2013.

\bibitem[DKL24]{dotsenko2024reconnectads}
Vladimir Dotsenko, Adam Keilthy, and Denis Lyskov.
\newblock Reconnectads.
\newblock {\em Algebraic Combinatorics}, 7(3):801--842, 2024.

\bibitem[Dot22]{dotsenko2022homotopy}
Vladimir Dotsenko.
\newblock Homotopy invariants for via {Koszul} duality.
\newblock {\em Inventiones mathematicae}, pages 1--30, 2022.

\bibitem[DSV19]{dotsenko2019toric}
Vladimir Dotsenko, Sergey Shadrin, and Bruno Vallette.
\newblock Toric varieties of {Loday's} associahedra and noncommutative
  cohomological field theories.
\newblock {\em Journal of topology}, 12(2):463--535, 2019.

\bibitem[EHKR10]{etingof2010cohomology}
Pavel Etingof, Andr{\'e} Henriques, Joel Kamnitzer, and Eric~M Rains.
\newblock The cohomology ring of the real locus of the moduli space of stable
  curves of genus 0 with marked points.
\newblock {\em Annals of mathematics}, pages 731--777, 2010.

\bibitem[Erg17]{aras2017}
Aras Ergus.
\newblock Operads, {Duality} and the {Gravity} {Operad}.
\newblock
  \url{https://aergus.net/academic/documents/theses/aergus-msc-thesis.pdf},
  2017.

\bibitem[Foi16]{foissy2016chromatic}
Lo{\"\i}c Foissy.
\newblock Chromatic polynomials and bialgebras of graphs.
\newblock {\em arXiv preprint arXiv:1611.04303}, 2016.

\bibitem[Get94]{getzler1994two}
Ezra Getzler.
\newblock Two-dimensional topological gravity and equivariant cohomology.
\newblock {\em Communications in mathematical physics}, 163:473--489, 1994.

\bibitem[Get95]{getzler1995operads}
Ezra Getzler.
\newblock Operads and {Moduli} {Spaces} of {Genus} 0 {Riemann} {Surfaces}.
\newblock {\em The Moduli Space of Curves}, pages 199--230, 1995.

\bibitem[GK94]{ginzburg1994koszul}
Victor Ginzburg and Mikhail Kapranov.
\newblock Koszul duality for operads.
\newblock {\em Duke mathematical journal}, 76(1):203--272, 1994.

\bibitem[KK22]{kharitonov2022grobner}
Vladislav Kharitonov and Anton Khoroshkin.
\newblock Gr{\"o}bner bases for coloured operads.
\newblock {\em Annali di Matematica Pura ed Applicata (1923-)},
  201(1):203--241, 2022.

\bibitem[KW19]{khoroshkin2019real}
Anton Khoroshkin and Thomas Willwacher.
\newblock Real moduli space of stable rational curves revisted.
\newblock {\em arXiv preprint arXiv:1905.04499}, 2019.

\bibitem[LM00]{losev2000new}
Andrey Losev and Yuri Manin.
\newblock New moduli spaces of pointed curves and pencils of flat connections.
\newblock {\em Michigan Mathematical Journal}, 48(1):443--472, 2000.

\bibitem[LM04]{losev2004extended}
A~Losev and Yu~Manin.
\newblock Extended modular operad.
\newblock {\em Frobenius Manifolds: Quantum Cohomology and Singularities},
  pages 181--211, 2004.

\bibitem[LV12]{loday2012algebraic}
Jean-Louis Loday and Bruno Vallette.
\newblock {\em Algebraic operad}.
\newblock Springer, 2012.

\bibitem[LV14]{lambrechts2014formality}
Pascal Lambrechts and Ismar Voli{\'c}.
\newblock Formality of the little $n$-disks operad.
\newblock {\em Memoirs of the American Mathematical Society}, 230(1079), 2014.

\bibitem[Lys24a]{lyskov2023contractads}
Denis Lyskov.
\newblock A generalization of operads based on subgraph contractions.
\newblock {\em International Mathematics Research Notices}, page rnae096, 2024.

\bibitem[Lys24b]{lyskov2024Hamilton}
Denis Lyskov.
\newblock Operadic structure on {Hamiltonian} paths and cycles.
\newblock {\em arXiv preprint arXiv:2406.06931}, 2024.

\bibitem[Mar96]{markl1996distributive}
Martin Markl.
\newblock Distributive laws and {Koszulness}.
\newblock {\em Annales scientifiques de l'\'Ecole Normale Sup\'erieure},
  29(6):685--696, 1996.

\bibitem[May72]{may1972geometry}
J.~Peter May.
\newblock {\em The Geometry of Iterated Loop Spaces}, volume 271 of {\em
  Lecture Notes in Mathematics}.
\newblock Springer-Verlag, 1972.

\bibitem[MSS02]{markl2002operads}
Martin Markl, Steve Shnider, and Jim Stasheff.
\newblock Operads in algebra, topology and physics.
\newblock {\em Mathematical surveys and monographs}, 96, 2002.

\bibitem[MSvAX18]{moon2018birational}
Han-Bom Moon, Charles Summers, James von Albade, and Ranze Xie.
\newblock Birational contractions of $\overline{M}_{0,n}$ and combinatorics of
  extremal assignments.
\newblock {\em Journal of Algebraic Combinatorics}, 47:51--90, 2018.

\bibitem[OT13]{orlik2013arrangements}
Peter Orlik and Hiroaki Terao.
\newblock {\em Arrangements of hyperplanes}, volume 300.
\newblock Springer Science \& Business Media, 2013.

\bibitem[PP21]{pagaria2021hodge}
Roberto Pagaria and Gian~Marco Pezzoli.
\newblock Hodge theory for polymatroids.
\newblock {\em arXiv preprint arXiv:2105.04214}, 2021.

\bibitem[Rai10]{rains2010homology}
Eric~M Rains.
\newblock The homology of real subspace arrangements.
\newblock {\em Journal of Topology}, 3(4):786--818, 2010.

\bibitem[Sch94]{schmitt1994incidence}
William~R Schmitt.
\newblock Incidence {Hopf} algebras.
\newblock {\em Journal of Pure and Applied Algebra}, 96(3):299--330, 1994.

\bibitem[Smy09]{smyth2009towards}
David~Ishii Smyth.
\newblock Towards a classification of modular compactifications of the moduli
  space of curves.
\newblock {\em arXiv preprint arXiv:0902.3690}, 2009.

\bibitem[Sta63]{stasheff1963homotopy}
James~Dillon Stasheff.
\newblock Homotopy associativity of {H}-spaces. ii.
\newblock {\em Transactions of the American Mathematical Society},
  108(2):293--312, 1963.

\bibitem[Sta95]{stanley1995symmetric}
Richard~P Stanley.
\newblock A symmetric function generalization of the chromatic polynomial of a
  graph.
\newblock {\em Advances in Mathematics}, 111(1):166--194, 1995.

\bibitem[WK17]{ward2017feynman}
Benjamin Ward and Ralph Kaufmann.
\newblock Feynman categories.
\newblock {\em Ast{\'e}risque}, 387, 2017.

\bibitem[Yuz01]{yuzvinsky2001orlik}
SA~Yuzvinsky.
\newblock Orlik-{Solomon} algebras in algebra and topology.
\newblock {\em Russian Mathematical Surveys}, 56(2):293, 2001.

\bibitem[Yuz02]{yuzvinsky2002small}
Sergey Yuzvinsky.
\newblock Small rational model of subspace complement.
\newblock {\em Transactions of the American Mathematical Society},
  354(5):1921--1945, 2002.

\end{thebibliography}

\end{document}